\pdfoutput=1
\documentclass[11pt,a4paper,headings=normal,numbers=noenddot,parskip=never,abstract=true]{scrartcl}

\usepackage[T1]{fontenc}
\usepackage[utf8]{inputenc}
\usepackage{lmodern}
\usepackage[final]{microtype}

\usepackage[
paper=a4paper,
margin=1.22in,
includehead=false,
includefoot=false
]{geometry}

\setkomafont{disposition}{\normalfont\bfseries}
\setkomafont{title}{\normalfont\Large}

\setkomafont{author}{\normalfont\fontsize{13}{15.5}\selectfont}
\setkomafont{date}{\normalfont\normalsize}
\setkomafont{section}{\normalfont\Large\bfseries}
\setkomafont{subsection}{\normalfont\large\bfseries}
\setkomafont{subsubsection}{\normalfont\normalsize\bfseries}
\AtBeginDocument{%
	\setlength{\abovedisplayskip}{10pt plus 3pt minus 3pt}%
	\setlength{\belowdisplayskip}{10pt plus 3pt minus 3pt}%
	\setlength{\abovedisplayshortskip}{7pt plus 3pt minus 2pt}%
	\setlength{\belowdisplayshortskip}{8pt plus 3pt minus 3pt}%
}

\usepackage{scrlayer-scrpage}
\clearpairofpagestyles
\usepackage{amsmath,amsfonts,amssymb,amsthm,mathtools}
\usepackage{mathrsfs}
\usepackage{wasysym}
\usepackage{adjustbox}
\usepackage{comment}
\usepackage{framed}
\usepackage{tikz-cd}
\usetikzlibrary{decorations.pathmorphing}

\usepackage{enumitem}
\setlist[enumerate]{leftmargin=2.2em,itemsep=0.22em,topsep=0.45em}
\setlist[itemize]{leftmargin=2.0em,itemsep=0.20em,topsep=0.45em}

\usepackage[numbers,sort&compress]{natbib}
\usepackage{xcolor}
\usepackage{hyperref}
\hypersetup{
	colorlinks=true,
	linkcolor=blue!45!black,
	citecolor=blue!45!black,
	urlcolor=blue!45!black,
	filecolor=blue!45!black,
	menucolor=blue!45!black,
	linktocpage=true,
	pdfauthor={Fei Han and Yuanchu Li},
	pdftitle={Differential Models for Anderson Dual to Twisted Spin^c-Bordism and Twisted Anomaly Map}
}

\makeatletter
\newcommand{\andersonkeywords}{}
\providecommand{\keywords}[1]{\gdef\andersonkeywords{#1}}

\makeatother

\newcommand*{\Spinc}{\mathrm{Spin}^c}
\newcommand*{\Spin}{\mathrm{Spin}}
\newcommand*{\SO}{\mathrm{SO}}
\newcommand*{\uU}{\mathrm{U}}

\newcommand*{\PUH}{\mathrm{PU}(\mathcal{H})}

\newcommand*{\so}{\mathfrak{so}}
\newcommand*{\spinc}{\mathfrak{spin}^c}
\newcommand*{\FunPic}{\mathrm{FunPic}}
\newcommand*{\Fred}{\mathrm{Fred}}

\newcommand*{\TMF}{\mathrm{TMF}}
\newcommand*{\RZ}{\mathbb{R}/\mathbb{Z}}
\newcommand*{\pt}{\mathrm{pt}}

\newcommand*{\GL}{\mathrm{GL}}

\newcommand*{\Cl}{\mathbb{C}\mathrm{l}}

\newcommand*{\ch}{\mathrm{ch}}

\newcommand*{\CS}{\mathrm{CS}}
\newcommand*{\im}{\mathrm{im}}
\newcommand*{\id}{\mathrm{id}}
\newcommand*{\pr}{\mathrm{pr}}
\newcommand*{\cw}{\mathrm{cw}}
\newcommand*{\dR}{\mathrm{dR}}
\newcommand*{\clo}{\mathrm{clo}}
\newcommand*{\Hom}{\mathrm{Hom}}
\newcommand*{\coker}{\mathrm{coker}}

\newcommand*{\End}{\mathrm{End}}

\newcommand*{\tr}{\mathrm{tr}}
\newcommand*{\cts}{\mathrm{cts}}
\newcommand*{\CP}{{\mathbb{C}\mathrm{P}}}
\newcommand*{\Mfld}{\mathsf{Mfld}}
\newcommand*{\modZ}{\bmod{\mathbb{Z}}}
\newcommand*{\Struc}[1]{\mathbf{Spin^c_{\tau}}(#1)}
\newcommand*{\difStruc}[1]{\widehat{\mathbf{Spin^c_{\widehat{\tau}}}}(#1)}

\newcommand*{\DD}{Dixmier--Douady }

\newcommand{\be}{\begin{equation}}
	\newcommand{\ee}{\end{equation}}

\theoremstyle{plain}
\newtheorem{thm}{Theorem}[section]
\newtheorem{prop}[thm]{Proposition}

\newtheorem{lem}[thm]{Lemma}
\newtheorem{defi}[thm]{Definition}

\theoremstyle{remark}
\newtheorem{rmk}[thm]{Remark}
\numberwithin{equation}{section}
\newtheorem{conjecture}{Conjecture}[section]

\title{\makebox[\textwidth][c]{\parbox{1.10\textwidth}{\centering Differential Models for the Anderson Dual to Twisted $\mathrm{Spin}^c$-Bordism and a Twisted Anomaly Map}}}

\author{Fei Han and Yuanchu Li}

\keywords{differential cohomology, twisted cohomology, twisted $\mathrm{Spin}^c$-bordism, twisted $K$-theory, Anderson duality, bundle gerbes}
\date{}

\begin{document}
	
	\maketitle
	%\vspace{-0.99\baselineskip}
	
	\begin{abstract}

			We construct differential models for degree-3 twisted $\mathrm{Spin}^c$-bordism and for its Anderson dual. The model for the differential Anderson dual is based on the framework of Yamashita--Yonekura. Using these differential models, we define a twisted anomaly map from differential twisted $K$-theory with inverse twist to the differential Anderson dual of twisted $\mathrm{Spin}^c$-bordism. The construction is described geometrically in terms of bundle gerbes, gerbe modules, and reduced eta-invariants of Dirac operators associated to the twisted data. Conceptually, this map is expected to be related to the anomalies of twisted $1|1$-dimensional supersymmetric field theories, in line with the perspectives of Stolz--Teichner and Freed--Hopkins.
	\end{abstract}
	
	%\printkeywords
	
	\tableofcontents

	\section{Introduction}

	The purpose of this paper is to construct 
	concrete (co)cycle differential models for
	{twisted} $\Spinc$-bordism, its Anderson dual, and 
	a twisted anomaly map. 
	More precisely, we first give definitions for differential extensions to \textit{twisted} (co)homology theories in the spirit of~\cite{BS10, BunNik},  then build a concrete differential model for the twisted \textit{homology} theory, namely twisted $\Spinc$-bordism.  Using this differential model, we build a related differential twisted \textit{cohomology} theory, following the differential Anderson dual formalism developed in~\cite{Yam1}. 
	After that, we formulate these models in a compatible gerbe-theoretic language~\cite{Mur96,BCMMS}, and give a geometric description of the twisted anomaly map
	\[
	\widehat{\Phi}_{\widehat{\mathcal G}}\colon \widehat K^0(X,\widehat{\mathcal G}^{-1})\to (\widehat{I\Omega^{\Spinc}_{\mathrm{dR}}})^{2k}(X,\widehat{\mathcal G})
	\]
	in terms of differential-geometric data. 
	Our background motivation is from field theory, namely the Freed--Hopkins description of invertible field theories via Anderson duals
	of bordism spectra \cite{freed2021reflection}, and the Stolz--Teichner program which relates supersymmetric Euclidean field theories to generalized
	cohomology \cite{stolz2004elliptic, ST11}.

	\subsection{Background}
	
	Let $G=\{G_d,s_d,\rho_d\}_{d\in\mathbb Z_{\ge 0}}$ be a system of compact Lie groups, equipped for each
	$d$ with structure maps
	\[
	s_d\colon G_d\longrightarrow G_{d+1},
	\qquad
	\rho_d\colon G_d\longrightarrow O(d,\mathbb R),
	\]
	with $\rho_{d+1}\circ s_d$ compatible with the standard inclusions $O(d,\mathbb R)\hookrightarrow O(d+1,\mathbb R)$. 
	Write $\Omega^G$ for the associated stable tangential $G$-bordism theory and $(I\Omega^G)^*$ for its Anderson dual (see \cite[App.~B]{HS} and \cite[App.~B]{freed2007uncertainty}).
	Freed and Hopkins formulated the following conjecture for invertible quantum field theories:
	
	\begin{conjecture}[Freed--Hopkins {\cite[Conj.~8.37]{freed2021reflection}}]\label{Conj1}
		There is a natural $1\!:\!1$ correspondence
		\[
		\left\{
		\begin{array}{l}
			\text{deformation classes of reflection-positive,}\\
			\text{invertible, $n$-dimensional, fully extended  }\\
			\text{field theories with symmetry type $G$}
		\end{array}
		\right\}
		\;\simeq\;
		(I\Omega^G)^{n+1}(\pt).
		\]
	\end{conjecture}
	\noindent They prove the statement in the topological case after restricting the right hand side to the torsion subgroup \cite[Thm.~1.1]{freed2021reflection}. 
	As supporting evidence to the conjecture, a differential refinement, and in turn a physically motivated de Rham model for the theory $(I\Omega^G)^*$ are constructed by Yamashita and Yonekura by abstracting certain properties of invertible QFTs~\cite{Yam1}. 
	
	Another guiding vision is due to Stolz and Teichner \cite{ST11}: supersymmetric Euclidean field theories should furnish cocycles for generalized cohomology. 
	In dimension one, there is an established identification (see \cite[§3.2]{stolz2004elliptic}, \cite{hohnhold2010minimal})
	\begin{equation}\label{K}
		\left\{
		\begin{array}{l}
			\text{degree $-n$, 1-dimensional supersymmetric}\\
			\text{field theories over $X$}
		\end{array}
		\right\}\Big/\text{concordance}
		\;\simeq\;
		K^{-n}(X).
	\end{equation}
	In dimension two, the analogous statement remains a far-reaching conjecture for the spectrum $\TMF$ of topological modular forms 
	\begin{conjecture}[Segal--Stolz--Teichner {\cite{ST11}}]
		There is a $1\!:\!1$ correspondence
		\begin{equation}
			\left\{
			\begin{array}{l}
				\text{fully extended, degree $-n$, 2-dimensional}\\
				\text{supersymmetric field theories over $X$}
			\end{array}
			\right\}\Big/\text{concordance}
			\;\simeq\;
			\TMF^{-n}(X).
		\end{equation}
	\end{conjecture}
	\noindent See also \cite{Tachi} for related work on Anderson duality in the TMF context.
	\medskip
	
	Now we consider the map of spectra
	\begin{equation}\label{eq: KU-IO-spectra}
		KU \longrightarrow \Sigma^2 I_{\mathbb Z} MT\Spinc
	\end{equation}
	induced by the composition
	\[
	KU\wedge MT\Spinc \xrightarrow{\ \mathrm{id}\wedge \mathrm{ABS}\ } KU\wedge KU
	\xrightarrow{\ \mu\ } KU \xrightarrow{\ \gamma_K\ } \Sigma^2 I_{\mathbb Z},
	\]
	where $\mathrm{ABS}$ is the Atiyah--Bott--Shapiro orientation for the Madsen--Tillmann spectrum $MT\Spinc$, $\mu$ is $KU$-multiplication, $I_{\mathbb Z}$ is the Anderson dual of the sphere, and $\gamma_K$ is induced by the self-dual class in $ (I_\mathbb{Z}KU)^2(\pt)$. 
	By Bott periodicity, for each $k\in\mathbb Z$, \eqref{eq: KU-IO-spectra} yields a map
	\begin{equation}\label{anomalyK}
		\Phi\colon KU \longrightarrow \Sigma^{2k} I_{\mathbb Z} MT\Spinc .
	\end{equation}
	In \cite{YamII} these maps admit differential refinements:
	\begin{equation}\label{anomalyKsmooth}
		\widehat\Phi\colon \widehat{K}^0(X) \longrightarrow 
		\big(\widehat{I\Omega^{\Spinc}_{\dR}}\big)^{2k}(X). 
	\end{equation}
	In light of the correspondence \eqref{K} and Conjecture~\ref{Conj1}, one may interpret
	\eqref{anomalyKsmooth} as an anomaly map: a $1$-dimensional supersymmetric field theory over $X$ is sent to an invertible field theory over $X$ encoding its anomaly (cf.~\cite[§9]{freed2021reflection}).
	
	\medskip
	
	In this paper, we are interested in the \textit{twisted} version of \eqref{anomalyKsmooth}.
	For twisted generalized cohomology theories, see the works of Ando--Blumberg--Gepner and of Ando--Blumberg--Gepner--Hopkins--Rezk~\cite{ando2010twists,ABGHR,ABGHR2}, also May--Sigurdsson~\cite{MaySig}. 
	Let $\tau\colon X\to B^{2}\uU(1)$ be a degree $3$ twist, represented by a bundle gerbe, and let 
	$\widehat\tau\colon X\to B^{2}_{\mathrm{conn}}\uU(1)$ denote its differential refinement. 
	Our objective is to construct a twisted anomaly map
	\begin{equation}\label{anomaly_intro1}
		\widehat\Phi_{\widehat\tau}\colon 
		\widehat K^{0}
		\big(X,\widehat\tau^{-1}\big)\longrightarrow 
		\big(\widehat{I\Omega^{\Spinc}_{\dR}}\big)^{2k}
		\big(X,\widehat\tau\big).
	\end{equation}
	Here the source is the differential twisted $K$-theory which, in the Stolz--Teichner program (see \cite[§5]{ST11}), is expected to classify concordance classes of $\widehat\tau^{-1}$-twisted $1|1$-dimensional supersymmetric field theories over $X$.
	The map $\widehat\Phi_{\widehat\tau}$ is then intended to model the associated anomaly: it assigns to a twisted supersymmetric field theory a canonically determined twisted invertible field theory over $X$ encoding its anomaly.
	
	To realize (\ref{anomaly_intro1}),  we proceed in two stages. 
	First, we construct concrete (co)cycle models for differential twisted $\Spinc$-bordism
	$
	\widehat{\Omega^{\Spinc}_{*}}(X,\widehat\tau)
	$,
	together with its Anderson dual $\big(\widehat{I\Omega^{\Spinc}_{\dR}}\big)^{*}(X,\widehat\tau)$ which serves as the target theory.
	Second, we develop a compatible model of differential twisted $K$-theory
	$
	\widehat K^{0}(X,\widehat\tau^{-1})
	$,
	and from these data, we construct the anomaly map in a differential-geometric way. 
	More details of the constructions are given in the next subsection.

	\subsection{Main Results}
	
	The constructions in this paper are organized by the formalism of differential
	extension, which refines the topological classes of a generalized (co)homology theory by differential data. 
	Roughly, for a cohomology theory $E$, a differential extension assigns to each smooth manifold $X$ 
	a group $\widehat{E}^*(X)$ that fits into an exact sequence
	\[
	E^{*-1}(X) \xrightarrow{\ \ch \ }
	{\Omega^{*-1}(X; V_E^\bullet)}/{\mathrm{im}d}
	\xrightarrow{\ a\ }
	\widehat{E}^*(X)
	\xrightarrow{\ I\ }
	E^*(X)
	\longrightarrow 0,
	\]
	and the curvature homomorphism
	\[
	R \colon \widehat{E}^*(X) \longrightarrow \Omega^*_{\mathrm{clo}}(X; V_E^\bullet),
	\]
	assigns to each differential cohomology class its closed differential form representative,
	thereby encoding simultaneously the topological and geometric information. 
	For developments of differential generalized cohomology, 
	see~\cite{HS,BS10,bunke2016differential}. 
	In the presence of a differential twist, both the de Rham complex and the
	Chern character need to be deformed by the twist. For example,
	Bunke and Nikolaus \cite{BunNik} construct the twisted differential spectra model as a differential extension for a commutative ring spectrum. 
	
	Building on the characteristic properties of the notion of differential extension~\cite{BS10,BunNik,Yam23}, 
	we provide the definitions of differential extensions of twisted (co)homology theories 
	for degree-$3$ twists, 
	formulated respectively in Definition~\ref{defi:diffextco} (for cohomology) 
	and Definition~\ref{defi:diffextho} (for homology). 
	Within this framework, we construct our differential models, in terms of concrete (co)cycles, 
	for both twisted $\mathrm{Spin}^c$-bordism and for its Anderson dual. 
	
	\medskip

	Our first main result is a cycle model for differential twisted 
	$\mathrm{Spin}^c$-bordism. 
	To this end, we introduce \emph{differential twisted} 
	$\mathrm{Spin}^c$-\emph{structures}, which refine
	the following notion of twisted $\mathrm{Spin}^c$-structures
	
	\begin{defi}[\cite{Wang07}]
		Let $\tau\colon X \to B^2\mathrm{U}(1)$ be a topological twist of degree~$3$, 
		let $f\colon M \to X$ be a smooth manifold over $X$, 
		and let $f_E\colon M \to B\mathrm{SO}$ classify a real vector bundle $E \to M$. 
		A \emph{$\tau$-twisted} $\mathrm{Spin}^c$-\emph{structure} on $E$ consists of a homotopy
		\[
		\eta\colon \tau \circ f \simeq W_3 \circ f_E,
		\]
		fitting into the commutative diagram
		\[
		\begin{tikzcd}
			M \arrow[r,"{f_E}"] \arrow[d,"f"'] & 
			B\mathrm{SO} \arrow[d,"{W_3}"] 
			\arrow[dl, dashed, Rightarrow, "\eta"'] \\
			X \arrow[r,"{\tau}"'] & B^2\mathrm{U}(1)
		\end{tikzcd}
		\]
		where $W_3\colon B\mathrm{SO}\to B^2\mathrm{U}(1)$ denotes 
		the integral third Stiefel--Whitney class.
	\end{defi}
	\noindent
	Our differential twisted $\mathrm{Spin}^c$-structure
	refines the above notion by incorporating the corresponding geometric data:
	
	\begin{defi}[Definition~\ref{def: difftwiSpincStru}]
		A \emph{differential $\widehat{\tau}$-twisted $\mathrm{Spin}^c$-structure} on $E$ 
		is given by a $1$-simplex 
		$\widehat{\eta}$ in $B_{\nabla}^{2}\mathrm{U}(1)(M)$ connecting the $0$-simplices
		\[
		\widehat{\eta}\colon 
		\iota_2\widehat{\tau} \circ f 
		\rightarrow
		W_3^{\nabla} \circ f^{\nabla}_{E},
		\]
		as depicted in the diagram
		\[
		\begin{tikzcd}
			M \arrow[r,"{f^{\nabla}_{E}}"] \arrow[d,"f"'] & 
			B_{\nabla}\mathrm{SO} \arrow[d,"{W_3^{\nabla}}"] 
			\arrow[dl, dashed, Rightarrow, "\widehat{\eta}"'] \\
			X \arrow[r,"{\iota_2\widehat{\tau}}"'] & 
			B_{\nabla}^{2}\mathrm{U}(1)
		\end{tikzcd}
		\]
		where $W_3^{\nabla}\colon B_{\nabla}\mathrm{SO} \to B_{\nabla}^{2}\mathrm{U}(1)$ 
		denotes the differential refinement of the integral third Stiefel--Whitney class. 
	\end{defi}
	
	With this preparation, we introduce the notion of 
	\emph{differential twisted} $\mathrm{Spin}^c$-\emph{bordism}, as the following system
	\begin{align}\label{diffSpincdata}
		\big(
		\widehat{\Omega^{\mathrm{Spin}^c}_*}(-,-),\;
		\mathcal{M}_*^{\mathrm{Spin}^c}(-,-),\;
		\ch'^{\mathrm{Spin}^c},\;
		R^{\mathrm{Spin}^c},\;
		I^{\mathrm{Spin}^c},\;
		a^{\mathrm{Spin}^c}
		\big),
	\end{align}
	where $\widehat{\Omega^{\mathrm{Spin}^c}_*}(-,-)$ is a covariant functor from 
	the category of manifolds endowed with differential twists to graded abelian groups, 
	and $\mathcal{M}_*^{\mathrm{Spin}^c}(-,-)$ a covariant functor to the category of chain complexes.  
	$\ch'^{\mathrm{Spin}^c}$, $R^{\mathrm{Spin}^c}$, $I^{\mathrm{Spin}^c}$, 
	and $a^{\mathrm{Spin}^c}$ are natural transformations.  
	Concretely, 
	let $\widehat{\tau}\colon X\to B^2_{\mathrm{conn}}\mathrm{U}(1)$ 
	be a differential refinement of a topological twist 
	$\tau\colon X\to B^2\mathrm{U}(1)$, and let $H$ denote the associated curvature $3$-form, the system reads
	\begin{itemize}
		\item The \emph{differential twisted} $\mathrm{Spin}^c$-\emph{bordism group} 
		$\widehat{\Omega^{\mathrm{Spin}^c}_n}(X,\widehat{\tau})$ 
		is generated by quintuples
		\[
		(M, f, f^{\nabla}_{TM}, \widehat{\eta}, \phi),
		\]
		called \emph{differential twisted} $\mathrm{Spin}^c$-\emph{cycles} over $X$, 
		where $\phi \in \Omega_{n+1}(X;V^{\mathrm{Spin}^c}_\bullet)/\mathrm{im}\,\partial_H$ 
		is represented by a de Rham current with coefficients in 
		$V^{\mathrm{Spin}^c}_\bullet = \Omega^{\mathrm{Spin}^c}_\bullet(\mathrm{pt}) \otimes \mathbb{R}$ 
		(see Definition~\ref{def: diffSpincbordism}).  
		The quadruple $(M, f, f^{\nabla}_{TM}, \widehat{\eta})$, 
		called a \emph{geometric twisted} $\mathrm{Spin}^c$-\emph{chain} over $X$, 
		consists of a compact oriented Riemannian manifold $M$ over $X$, 
		equipped with a differential twisted $\mathrm{Spin}^c$-structure $\widehat{\eta}$ on its tangent bundle $TM$, 
		whose classifying map is $f^{\nabla}_{TM}$.
		
		\item The chain complex 
		$\mathcal{M}_*^{\mathrm{Spin}^c}(X,\widehat{\tau}) := (\Omega_{*}(X;V^{\mathrm{Spin}^c}_\bullet),\,\partial_H)$ 
		is the \emph{twisted de Rham chain complex} with coefficients in the graded ring 
		$V^{\mathrm{Spin}^c}_\bullet$, with twisted boundary operator  
		\[
		\partial_H := \partial + H \wedge (\mathbb{C}\mathrm{P}^1 \times -),
		\]
		deformed by the curvature $3$-form $H$ of the twist $\widehat{\tau}$.  
		{\em The algebraic structures of $V^{\mathrm{Spin}^c}_\bullet$ and its dual 
			$N^{\bullet}_{\mathrm{Spin}^c}$, arising from the $\mathrm{Spin}^c$ group, 
			are essential in defining this deformed de Rham complex 
			and in constructing the corresponding twisted Chern--Weil map 
			(see Section~\ref{sec:diff_spinc_bord}).}
		
		\item $\ch'^{\mathrm{Spin}^c}$ is a homomorphism from the topological twisted 
		$\mathrm{Spin}^c$-bordism group to the homology of 
		$\mathcal{M}^{\mathrm{Spin}^c}_*(X,\widehat{\tau})$.
		
		\item $R^{\mathrm{Spin}^c}$ is the \emph{curvature map} sending a differential cycle 
		in $\widehat{\Omega^{\mathrm{Spin}^c}_*}(X,\widehat{\tau})$ 
		to a closed current in $\Omega_{*}(X;V^{\mathrm{Spin}^c}_\bullet)$.
		
		\item $I^{\mathrm{Spin}^c}$ is the \emph{forgetful map} that discards the differential data.
		
		\item $a^{\mathrm{Spin}^c}$ maps 
		$\phi \in \Omega_{*}(X;V^{\mathrm{Spin}^c}_\bullet)/\mathrm{im}\,\partial_H$ 
		to the formal cycle $(\emptyset,\emptyset,\emptyset,\emptyset,-\phi)$.
	\end{itemize}
	
	\noindent
	In Theorem~\ref{thm:diffspinc}, we show that the system~\eqref{diffSpincdata} 
	constitutes a differential extension of twisted $\mathrm{Spin}^c$-bordism 
	in the sense of Definition~\ref{defi:diffextho}.  
	Our model refines the construction in~\cite{Wang07}, and
	recovers the differential $G$-bordism theory in~\cite{Yam1,Yam23} in the case of $G = \mathrm{Spin}^c$ and the trivial twist.
	
	\medskip

	Applying Yamashita--Yonekura's formalism to our theory~\eqref{diffSpincdata}, 
	we construct a cocycle model for the differential Anderson dual of twisted 
	$\mathrm{Spin}^c$-bordism as the following system
	\begin{align}\label{diffIOdata}
		\Big(
		\bigl(\widehat{I\Omega^{\mathrm{Spin}^c}_{\mathrm{dR}}}\bigr)^{*}(-,-),\;
		\mathcal{M}^*_{I\Omega}(-,-),\;
		\ch'_{I\Omega},\;
		R_{I\Omega},\;
		I_{I\Omega},\;
		a_{I\Omega}
		\Big),
	\end{align}
	where $\bigl(\widehat{I\Omega^{\mathrm{Spin}^c}_{\mathrm{dR}}}\bigr)^{*}(-,-)$ 
	is a contravariant functor from the category of manifolds endowed with differential twists 
	to graded abelian groups, and $\mathcal{M}^*_{I\Omega}(-,-)$ is a contravariant functor 
	to the category of cochain complexes.  
	$\ch'_{I\Omega}$, $R_{I\Omega}$, $I_{I\Omega}$, and $a_{I\Omega}$ are natural transformations. Concretely, 
	for $\widehat{\tau}\colon X\to B^2_{\mathrm{conn}}\mathrm{U}(1)$, the system reads
	\begin{itemize}
		\item 
		$(\widehat{I\Omega_{\mathrm{dR}}^{\mathrm{Spin}^c}})^n(X,\widehat{\tau})$ 
		is the abelian group generated by pairs $(\omega,h)$, where $\omega$ is a twisted closed differential form 
		valued in the $\mathrm{Spin}^c$-characteristic classes of total degree $n$, 
		and $h$ is an $\mathbb{R}/\mathbb{Z}$-valued functional on 
		$(n-1)$-dimensional differential twisted $\mathrm{Spin}^c$-bordism cycles, 
		satisfying a natural compatibility condition.
		
		\item The complex 
		$\mathcal{M}^*_{I\Omega}(X,\widehat{\tau}) := (\Omega^{*}(X;N_{\mathrm{Spin}^c}^{\bullet}), D_H)$ 
		is a twisted de~Rham \emph{cochain} complex 
		with coefficients $N_{\mathrm{Spin}^c}^{\bullet} 
		= \mathrm{Hom}(\Omega^{\mathrm{Spin}^c}_{\bullet}(\mathrm{pt}), \mathbb{R})$, 
		whose differential is deformed by the curvature $3$-form $H$
		\[
		D_H = d + H \wedge \partial_{\zeta},
		\]
		where $\partial_\zeta$ is the degree $-2$ partial derivative operator with respect to the generator $\zeta$ of cohomological degree $2$.

		\item $\ch'_{I\Omega}$ is a homomorphism from topological twisted Anderson dual into the cohomology of $\mathcal{M}^*_{I\Omega}(X,\widehat{\tau})$.
		
		\item $R_{I\Omega}$ is the \emph{curvature map}, sending a pair $(\omega,h)$ to its curvature form $\omega$.
		
		\item $I_{I\Omega}$ is the \emph{forgetful map} that discards the differential data.
		
		\item $a_{I\Omega}$ assigns to each 
		$\alpha \in \Omega^{n-1}(X;N_{\mathrm{Spin}^c}^{\bullet})/\mathrm{im}\,D_H$ 
		a formal compatible pair in 
		$(\widehat{I\Omega_{\mathrm{dR}}^{\mathrm{Spin}^c}})^n(X,\widehat{\tau})$.
	\end{itemize}
	
	\noindent
	In Theorem~\ref{thm:diffio}, we prove that the system~\eqref{diffIOdata} 
	realizes a differential model for the \emph{Anderson dual of twisted} 
	$\mathrm{Spin}^c$-\emph{bordism}, in the sense of Definition~\ref{defi:diffextco}. Hence our construction extends the model of Yamashita--Yonekura~\cite{Yam1} 
	to the twisted setting for $G = \mathrm{Spin}^c$.

	\medskip

	In close analogy with the operations constructed
	for $(\widehat{I\Omega_{\mathrm{dR}}^{G}})^n(X)$ in~\cite{Yam1}, 
	we define the \emph{differential multiplication} and \emph{pushforward} operations 
	for the twisted theory 
	$(\widehat{I\Omega_{\mathrm{dR}}^{\mathrm{Spin}^c}})^n(X,\widehat{\tau})$.  
	To describe the twisted differential multiplication, 
	we introduce a differential \emph{cocycle model} 
	$\widehat{\Omega_{\mathrm{Spin}^c}^{-r}}(X,\widehat{\tau})$ 
	for the twisted $\mathrm{Spin}^c$-\emph{cobordism} group (see Definition \ref{def:diffSpincCobordismCochain}).  
	This construction may be viewed as a twisted generalization 
	of the differential cobordism theories developed in~\cite{bunke2009landweber} 
	and~\cite{Yam1}.  
	More precisely, let $\widehat{\tau}_1$, $\widehat{\tau}_2$, and 
	$\widehat{\tau}_3 = \widehat{\tau}_1 + \widehat{\tau}_2$ 
	be differential twists over $X$.  
	We construct a twisted differential multiplication map
	\begin{align}
		(\widehat{I\Omega^{\mathrm{Spin}^c}_{\mathrm{dR}}})^{n}(X,\widehat{\tau}_3)
		\otimes
		\widehat{\Omega_{\mathrm{Spin}^c}^{-r}}(X,\widehat{\tau}_2)
		\longrightarrow
		(\widehat{I\Omega^{\mathrm{Spin}^c}_{\mathrm{dR}}})^{n-r}(X,\widehat{\tau}_1).
	\end{align}
	Furthermore, for a proper submersion $p\colon N \to X$ 
	of relative dimension $r$, 
	equipped with a differential $\widehat{\tau}_2$-twisted 
	$\mathrm{Spin}^c$-structure on its stable relative tangent bundle, 
	we construct the corresponding \emph{differential pushforward map}
	\begin{align}
		(\widehat{I\Omega^{\mathrm{Spin}^c}_{\mathrm{dR}}})^{n}(N,p^*\widehat{\tau}_3)
		\longrightarrow
		(\widehat{I\Omega^{\mathrm{Spin}^c}_{\mathrm{dR}}})^{n-r}(X,\widehat{\tau}_1).
	\end{align}

	\medskip
	
	Now we turn to our construction of the \emph{twisted anomaly map}
	\begin{equation}\label{anomaly_intro}
		\widehat{\Phi}_{\widehat{\tau}}\colon 
		\widehat{K}^{0}\big(X,\widehat{\tau}^{-1}\big)
		\longrightarrow
		\big(\widehat{I\Omega^{\mathrm{Spin}^c}_{\mathrm{dR}}}\big)^{2k}\big(X,\widehat{\tau}\big).
	\end{equation}
	In order to understand the map from a more \emph{geometric} point of view, in the spirit of~\cite[Sec.~3.4.3]{YamII}, we seek a construction based on 
	\emph{bundle-theoretic} data and \emph{Dirac operators}.  
	However, in the twisted setting, ordinary vector bundles no longer suffice. One needs to use twisted geometric objects, \emph{bundle gerbes} and  \emph{gerbe modules}, which provide the appropriate setting for defining the analytic realization of the twisted anomaly map.
	
	\medskip
	
	The theory of \emph{bundle gerbes} and \emph{gerbe modules} is developed in the works of 
	Hitchin~\cite{hitchin1999lectures}, Murray~\cite{Mur96}, and
	Bouwknegt--Carey--Mathai--Murray--Stevenson~\cite{BCMMS}, among others.  
	Building on these foundations, we introduce \emph{gerbe-theoretic} models for differential twisted 
	$\mathrm{Spin}^c$-bordism and its Anderson dual.  
	These models provide a geometric interface to twisted spinor bundles and Dirac operators, 
	which form the analytical core of the construction of the twisted anomaly map~\eqref{anomaly_intro}.  
	Within this framework, a differential twist $\widehat{\tau}$ can be equivalently realized 
	as a \emph{bundle gerbe with connection and curving}, denoted by $\widehat{\mathcal{G}}$.  
	One may reformulate the notion of a differential twisted $\mathrm{Spin}^c$-structure 
	in the language of bundle gerbe modules as follows
	
	\begin{defi}[Definition~\ref{def: diffSpincBG}]
		Let $\widehat{\mathcal{G}}$ be a bundle gerbe with connection and curving over $X$.  
		Let $f\colon M \to X$ be a map, and $E \to M$ an oriented vector bundle 
		equipped with a connection $\nabla^{E}$.  
		A \emph{differential $\widehat{\mathcal{G}}$-twisted $\mathrm{Spin}^c$-structure} on $E$ 
		consists of the data
		\[
		(\nabla^{S^c}, \Psi)
		\]
		where $\nabla^{S^c}$ is a $f^{*}\widehat{\mathcal{G}}$-module connection with underlying gerbe module $S^c$,
		and $\Psi$ is a connection-preserving isomorphism of Azumaya bundles
		\[
		\Psi \colon \Cl^{+}(E) \xrightarrow{\;\cong\;} \End(S^c).
		\]
	\end{defi}
	\noindent
	With this definition in place, we can construct gerbe-theoretic models for 
	the differential twisted $\mathrm{Spin}^c$-bordism theory 
	$\widehat{\Omega^{\mathrm{Spin}^c}_{*}}(X,\widehat{\mathcal{G}})$ 
	and for its Anderson dual $
	\bigl(\widehat{I\Omega^{\mathrm{Spin}^c}_{\mathrm{dR}}}\bigr)^{*}(X,\widehat{\mathcal{G}})
	$, which
	are equivalent to the preceding models~\eqref{diffSpincdata} and~\eqref{diffIOdata}.

	\medskip

	Regarding the source of the twisted anomaly map, i.e. the differential twisted $K$-theory, we also give a gerbe-theoretic model. 
	Both twisted $K$-theory and differential $K$-theory have extensive literature. For topological twisted $K$-theory, see for example
	\cite{ando2010twists,donovan, rosenberg1989continuous, atiyah2004twisted, FHT1, BCMMS}; for differential $K$-theory, see \cite{HS, bunke2007smooth, simons2008structured, freed2010index}. 
	In particular, for torsion twists, Park \cite{park} constructs a model for differential twisted even $K$-theory via twisted vector bundles (i.e. finite rank module over the Hitchin--Chatterjee gerbe). Our model of differential twisted even $K$-theory
	$
	\widehat K^{0}\big(X,\widehat{\mathcal{G}}\big)
	$ 
	extends Park's model to the non-torsion case by allowing super $U_\tr$-gerbe modules in Section~\ref{sec: K}. 
	
	\medskip

	With these geometric models in hand, we construct the desired 
	\emph{twisted anomaly map}
	\begin{equation}\label{intro-Phi}
		\widehat{\Phi}_{\widehat{\mathcal{G}}}\colon
		\widehat{K}^{0}(X,\widehat{\mathcal{G}}^{-1})
		\longrightarrow
		\bigl(\widehat{I\Omega^{\mathrm{Spin}^c}_{\mathrm{dR}}}\bigr)^{2k}(X,\widehat{\mathcal{G}}).
	\end{equation}
	More precisely, let 
	$x=(\mathcal{E},\nabla^{\mathcal{E}},\rho)$  
	be a representative of a class in 
	$\widehat{K}^{0}(X,\widehat{\mathcal{G}}^{-1})$,  
	where $(\mathcal{E},\nabla^{\mathcal{E}})$  
	is a super $U_{\mathrm{tr}}$-module equipped with compatible module connections,  
	and $\rho \in \Omega^{\mathrm{odd}}(X)/\mathrm{im}(d+H)$.  
	The twisted anomaly map $\widehat{\Phi}_{\widehat{\mathcal{G}}}$  
	assigns to such a representative a compatible pair $(\omega_x,h_x)$,  
	where $\omega_x$ is a $D_H$-closed differential form, 
	and $h_x$ is a functional on differential twisted $\mathrm{Spin}^c$-bordism defined via the reduced eta-invariant and the Atiyah--Patodi--Singer index theorem. See Section~\ref{sec:ano_cons} for details. 
	
	\medskip
	
	The paper is organized as follows.  
	In Section~2, we review the theory of parametrized spectra following~\cite{MaySig}.  
	We also recall the essential properties of differential twisted cohomology developed in~\cite{BunNik},  
	and formulate the definition of differential twisted (co)homology for degree-3 twists.  In Section~3, we construct differential models for the twisted 
	$\mathrm{Spin}^c$-bordism theory and its Anderson dual.  
	In particular, Section~\ref{sec: chaincpx} and \ref{sec: cw} serve as the technical foundation of the paper,  
	where we carefully examine the algebraic structure of the twisted de~Rham chain and cochain complexes
	\[
	(\Omega_{*}(X;V^{\mathrm{Spin}^c}_{\bullet}),\,\partial_H)
	\quad\text{and}\quad
	(\Omega^{*}(X;N^{\bullet}_{\mathrm{Spin}^c}),\,D_H),
	\]
	as well as the related Chern--Weil map. 
	In Section~4, we introduce the gerbe-theoretic formulation of differential twists  
	and develop the corresponding differential models.  
	The last part of this section is devoted to the construction of the desired 
	twisted anomaly map~\eqref{intro-Phi} using eta invariants and the Atiyah--Patodi--Singer theorem.

	\section{Twisted generalized (co)homology and differential extensions}

	The modern framework for twisted cohomology is developed by Ando--Blumberg--Gepner and Ando--Blumberg--Gepner--Hopkins--Rezk~\cite{ando2010twists,ABGHR,ABGHR2}, while a parametrized spectra implementation is developed by May--Sigurdsson \cite{MaySig}. 
	In this paper, we work with $\Spinc$-bordism theory and its Anderson dual, both of which admit natural degree-$3$ twists.
	In this section, we recall the definitions of twisted (co)homology in terms of parametrized spectra, and formulate the notion of differential extensions in the case of degree-$3$ twists.
	
	Let $\Mfld$ be the site of smooth manifolds of finite CW
	type, with the Grothendieck topology of open covers. A \emph{simplicial presheaf} on $\Mfld$ is a contravariant functor
	$
	F\colon \Mfld^{\mathrm{op}}\longrightarrow \mathsf{sSet}.
	$
	We use the local model structure on simplicial presheaves. Recall that the standard stacky model for $K(\mathbb{Z},3)$ is the simplicial presheaf
	\[
	B^{2}\uU(1):=\mathrm{DK}
	\big(\underline{\uU(1)}[2]\big),
	\]
	which is the Dold--Kan image of the sheaf $\underline{\uU(1)}$ placed in degree $2$.
	Let $X\in\mathsf{Mfld}$ be a smooth manifold without boundary. We regard $X$ as a representable simplicial presheaf via the Yoneda embedding. A \emph{degree-$3$ topological twist} over $X$, or simply a \emph{topological twist} is a natural transformation
	\begin{equation}\label{def: toptwi}
		\tau\colon X\longrightarrow B^{2}\uU(1),
	\end{equation}
	equivalently a $0$-simplex of $B^{2}\uU(1)(X)$.
	Write $\Mfld/B^{2}\uU(1)$ for the slice category whose objects are pairs $(X,\tau)$ as above. A morphism
	$(X,\tau)\to (X',\tau')$
	consists of a smooth map $f\colon X\to X'$ together with a homotopy class $\gamma\colon \tau\simeq f^{*}\tau'$ of twist identifications.%

	\subsection{Review of bundle of spectra and twisted (co)homology}
	
	In the framework of parametrized spectra \cite{MaySig}, twisted (co)homology is described via a bundle of spectra associated to the topological twist. Let $k$ be a spectrum equipped with a $K(\mathbb{Z},2)$-action. The twisted $k$-cohomology and $k$-homology are functors
	\[
	k^{*}(-,-)\colon (\Mfld/B^{2}\uU(1))^{\mathrm{op}}\longrightarrow \mathsf{Ab}^{\mathbb{Z}},
	\quad
	k_{*}(-,-)\colon \Mfld/B^{2}\uU(1)\longrightarrow \mathsf{Ab}^{\mathbb{Z}},
	\]
	where $\mathrm{Ab}^{\mathbb{Z}}$ is the category of $\mathbb{Z}$-graded abelian groups. 
	Concretely, for $(X,\tau)$ there is a principal $K(\mathbb{Z},2)$-bundle $P\to X$ classified by $\tau\colon X\to K(\mathbb{Z},3)$, and an associated bundle of spectra
	\[
	P_{\tau}(k):=P\times_{K(\mathbb{Z},2)} k \longrightarrow X.
	\]
	Let $r\colon X\to \pt$ be the terminal map. Denote by $r_{*}$ the pushforward of sections, by $r_{!}$ the Thom pushforward, and by $F_{X}$ the internal function spectrum functor over $X$; let $S_{X}$ be the sphere spectrum over $X$. Then one writes
	\begin{equation}\label{def: twicoho}
		k_{n}(X,\tau):=\pi_{n}\big(r_{!}(P_{\tau}(k))\big),
		\qquad
		k^{n}(X,\tau):=\pi_{-n}\big(r_{*}F_{X}(S_{X},P_{\tau}(k))\big).
	\end{equation}
	For a morphism in $\Mfld/B^{2}\uU(1)$, the induced maps are defined naturally with $f$ and the twist identification $\gamma$, and functoriality follows formally.
	
	The twisted (co)homology theories~\eqref{def: twicoho} satisfy the axioms of homotopy invariance, exactness, excision, and additivity in the parametrized setting \cite[§20.1]{MaySig}.
	We also record the twisted Atiyah-Hirzebruch spectral sequence 
	
	\begin{prop}{\cite[Prop.~22.1.5]{MaySig}}\label{prop:AHSS}
		For the twisted theories represented by the bundle of spectra $P_{\tau}(k)\to X$, there are natural spectral sequences
		\[
		E^2_{p,q}\cong H_p\bigl(X;L_q(X,P_{\tau}(k))\bigr)\;\Longrightarrow\; k_{p+q}(X,\tau),
		\]
		\[
		E_2^{p,q}\cong H^p\bigl(X;L^q(X,P_{\tau}(k))\bigr)\;\Longrightarrow\; k^{p+q}(X,\tau). 
		\]
		Here $L_q(X,P_{\tau}(k))$ and $L^q(X,P_{\tau}(k))$ denote the 
		homological and cohomological coefficient systems associated to $P_{\tau}(k)$ with fiber $\pi_\ast(k)$, and in both cases the monodromy is induced by the principal $K(\mathbb{Z},2)$-bundle $P\to X$ and the $K(\mathbb{Z},2)$-action on the corresponding homotopy groups of $k$.
	\end{prop}

	\subsection{Differential extensions to twisted (co)homology theories}\label{sec: generaldiffext}
	
	Hopkins--Singer model differential cohomology via differential function spectra \cite{HS}. This perspective is reformulated by Bunke--Nikolaus--Völkl in the setting of sheaves of spectra on smooth manifolds \cite{bunke2016differential}. Bunke--Schick propose axioms for differential extensions of generalized cohomology and show the uniqueness property once an $S^1$-integration is given \cite{BS10}. Dually, Yamashita gives an axiomatic treatment of differential homology \cite{Yam23}. 
	
	In the twisted case, Bunke--Nikolaus develop twisted differential cohomology via twisted differential function spectra \cite{BunNik}, satisfying a similar system of properties in the presence of a twist. Motivated by these characteristic properties, we formulate our definitions of differential extensions for twisted (co)homology with degree-$3$ twists in Definition~\ref{defi:diffextco} and Definition~\ref{defi:diffextho}.
	
	We begin by recalling the differential refinement of topological twists. 
	Define the simplicial presheaf of differential twists by
	\[
	B_{\mathrm{conn}}^{2}\uU(1)
	:=
	\mathrm{DK}
	\bigl(
	\underline{\uU(1)}
	\xrightarrow{\tilde d}
	\Omega^{1}
	\xrightarrow{d}
	\Omega^{2}
	\bigr), 
	\]
	which is the Dold--Kan image of the smooth Deligne complex in degrees $0,1,2$, where $\tilde d=\tfrac{1}{2\pi i} d\log$. 
	A \emph{degree-$3$ differential twist}, or simply a \emph{differential twist} over $X$ is a natural transformation
	\[
	\widehat{\tau}\colon X\longrightarrow B_{\mathrm{conn}}^{2}\uU(1),
	\]
	equivalently a $0$-simplex of $B^{2}_{\mathrm{conn}}\uU(1)(X)$. We say that $\widehat{\tau}$ is a differential refinement of the topological twist $\tau$ if
	% https://q.uiver.app/#q=WzAsMyxbMSwwLCJcdEJfe1xcbWF0aHJte2Nvbm59fV57Mn1cXG1hdGhybXtVfSgxKSJdLFswLDAsIlgiXSxbMSwxLCJcdEJeezJ9XFxtYXRocm17VX0oMSkiXSxbMCwyXSxbMSwwLCJcXHdpZGVoYXR7XFx0YXV9Il0sWzEsMiwiXFx0YXUiLDJdXQ==
	\[\begin{tikzcd}
		X & {	B_{\mathrm{conn}}^{2}\mathrm{U}(1)} \\
		& {	B^{2}\mathrm{U}(1)}
		\arrow["{\widehat{\tau}}", from=1-1, to=1-2]
		\arrow["\tau"', from=1-1, to=2-2]
		\arrow[from=1-2, to=2-2]
	\end{tikzcd}\]
	commutes, where  
	the vertical arrow is the natural forgetful functor. 
	More explicitly, for a good open cover $\pi\colon \mathcal U\to X$, a $0$-simplex $\widehat{\tau}\in B^{2}_{\mathrm{conn}}\uU(1)(X)$ is presented by a \v{C}ech--Deligne cocycle $(\epsilon_{ijk},A_{ij},B_i)$ with
	\[
	\delta \epsilon=1,\qquad
	\delta A=\tfrac{1}{2\pi i}d\log\epsilon,\qquad
	\delta B=dA,
	\]
	where $\epsilon\in 
	C^{\infty}( {\mathcal{U}}^{[3]}, \uU(1))$, $A_{ij}\in \Omega^{1}({\mathcal{U}}^{[2]})$, and $B\in \Omega^{2}({\mathcal{U}})$, and $\delta$ is the alternating Čech differential. A $1$-simplex $(h_{ij},\lambda_i)$ from $(\epsilon,A,B)$ to $(\epsilon',A',B')$ satisfies
	\[
	\delta h=\epsilon'\epsilon^{-1},\quad
	\delta \lambda=A'-A-\tfrac{1}{2\pi i}\,d\log h,\quad
	d\lambda=B'-B,
	\]
	where $h\in C^{\infty}( {\mathcal{U}}^{[2]}, \uU(1))$ and $\lambda\in \Omega^1(\mathcal{U})$. 
	Since $\delta dB=0$, there is a uniquely determined global closed $3$-form $H\in \Omega^{3}_{\clo}(X)$ characterized by
	\[
	\pi^{*}H=dB, 
	\]
	which we call the \emph{curvature} of $\widehat{\tau}$, whose de Rham cohomology class is the real image of the Dixmier--Douady class $\mathrm{DD}(\tau)\in H^{3}(X;\mathbb Z)$. 
	We note that the construction of $H$ is independent of the choice of open cover. 
	
	\medskip
	
	Let $\Mfld/B^{2}_{\mathrm{conn}}\uU(1)$ be the slice category of manifolds equipped with differential twists. The natural forgetful functor induces a functor
	\[
	\Mfld/B^{2}_{\mathrm{conn}}\uU(1)\longrightarrow \Mfld/B^{2}\uU(1).
	\]
	An object $(X,\widehat{\tau})$ of $\Mfld/B^{2}_{\mathrm{conn}}\uU(1)$ is called a \emph{differential refinement} of $(X,\tau)$ in $\Mfld/B^{2}\uU(1)$ if $\widehat{\tau}$ refines $\tau$ as above.
	Motivated by the characteristic properties of differential extensions in \cite{BunNik,BS10,Yam23}, we now formulate our definition of differential extensions to twisted (co)homology in the presence of differential twists.

	\begin{defi}\label{defi:diffextco}
		A differential extension to the twisted cohomology theory $k^{*}(-,-)$ 
		consists of the following data
		\[
		(
		\widehat{k^*}(-,-), 
		\mathcal{M}^*(-,-), \ch', R, I, a
		),
		\]
		where
		\begin{itemize}
			\item 
			$\widehat{k^*}(-,-)$ is a contravariant functor 
			\[
			\widehat{k^*}(-,-): 
			\big(
			\mathsf{Mfld^{}}/
			B^2_{\mathrm{conn}}\uU(1)
			\big)^{op}
			\to
			\mathsf{Ab}^{\mathbb{Z}}.
			\]
			\item 
			$\mathcal{M}^*(-,-)$ is a contravariant functor 
			\[
			\mathcal{M}^*(-,-):
			\big(
			\mathsf{Mfld^{}}/
			B^2_{\mathrm{conn}}\uU(1)
			\big)^{op}
			\to
			\mathsf{Ch}_{\mathbb{R}}^*, 
			\]
			where $\mathsf{Ch}_{\mathbb{R}}^*$ denotes the category of cochain complexes over $\mathbb{R}$. 
			\item
			$\ch'$ is a natural transformation
			\[
			\ch':
			{k^*}(-,-)
			\to
			H^*(\mathcal{M}^*(-,-)).
			\]

			\item $R, I, a$ are natural transformations
			\begin{align*}
				R &\colon 
				\widehat{k^*}(-,-) \to 
				\mathcal{M}_{\mathrm{clo}}^*,
				\\
				I &\colon 
				\widehat{k^*}(-,-) \to {k^*}(-,-),
				\\
				a &\colon 
				\mathcal{M}^{*-1}/ \mathrm{im}d_{\mathcal{M}} 
				\to \widehat{k^*}(-,-),
			\end{align*}
			where $d_\mathcal{M}$ denotes the cochain differential of $\mathcal{M}^*$.
		\end{itemize}
		
		For each object $(X, \widehat{\tau})$, the following conditions hold
		\begin{enumerate}[label=(\roman*)]
			\item 
			$\ch'$ is a group homomorphism and induces an isomorphism:
			\[
			\ch'\otimes \mathbb{R}: 
			{k^*}(X, {\tau})
			\otimes \mathbb{R}
			\xrightarrow{\simeq}
			H^*(\mathcal{M}^*(X, \widehat{\tau})).
			\]
			\item 
			The following diagram commutes:
			% https://q.uiver.app/#q=WzAsNSxbMCwxLCJcdFxcd2lkZWhhdHtrXip9KFgsXFx3aWRlaGF0e1xcdGF1fSkiXSxbMSwxLCJcXG1hdGhjYWx7TX1fe1xcbWF0aHJte2Nsb319XiooWCwgXFx3aWRlaGF0e1xcdGF1fSkiXSxbMCwyLCJ7a14qfShYLHtcXHRhdX0pIl0sWzAsMCwiXFxtYXRoY2Fse019XnsqLTF9KFgsIFxcd2lkZWhhdHtcXHRhdX0pLyBcXG1hdGhybXtpbX1kX3tcXG1hdGhjYWx7TX19ICJdLFsxLDIsIkheKihcXG1hdGhjYWx7TX1eKihYLCBcXHdpZGVoYXR7XFx0YXV9KSkiXSxbMCwxLCJSIl0sWzMsMCwiYSIsMl0sWzAsMiwiSSIsMl0sWzEsNF0sWzIsNCwiXFxtYXRocm17Y2h9JyIsMl0sWzMsMSwiZF97XFxtYXRoY2Fse019fSAiXV0=
			\[\begin{tikzcd}
				{\mathcal{M}^{*-1}(X, \widehat{\tau})/ \mathrm{im}d_{\mathcal{M}} } \\
				{	\widehat{k^*}(X,\widehat{\tau})} & {\mathcal{M}_{\mathrm{clo}}^*(X, \widehat{\tau})} \\
				{{k^*}(X,{\tau})} & {H^*(\mathcal{M}^*(X, \widehat{\tau}))}
				\arrow["a"', from=1-1, to=2-1]
				\arrow["{d_{\mathcal{M}} }", from=1-1, to=2-2]
				\arrow["R", from=2-1, to=2-2]
				\arrow["I"', from=2-1, to=3-1]
				\arrow[from=2-2, to=3-2]
				\arrow["{\mathrm{ch}'}"', from=3-1, to=3-2]
			\end{tikzcd}\]
			\item The following sequence is exact:
			\begin{align*}\label{eq_exact_diffh}
				{k^{*-1}}(X,{\tau}) \xrightarrow{\mathrm{ch}' }	\mathcal{M}^{*-1}(X, \widehat{\tau})/ \mathrm{im}d_{\mathcal{M}} 
				\xrightarrow{a} 
				\widehat{k^*}(X,\widehat{\tau}) \xrightarrow{I} 
				{k^*}(X,{\tau})
				\to 0.
			\end{align*}
		\end{enumerate}
	\end{defi}
	
	\noindent
	In the sense of this definition, Bunke--Nikolaus' twisted differential function spectra model is a differential extension to the corresponding twisted {cohomology} theory. 
	
	\medskip
	
	Dually, we formulate  
	our definition for differential extension to twisted \emph{homology} theories as follows, 
	which is
	a twisted generalization of Yamashita's axioms for differential homology.
	
	\begin{defi}\label{defi:diffextho}
		A differential extension to the twisted homology theory $k_{*}(-,-)$ 
		consists of the following data
		\[
		(
		\widehat{k_*}(-,-), 
		\mathcal{M}_*(-,-), \ch', R, I, a
		),
		\]
		where
		\begin{itemize}
			\item 
			$\widehat{k_*}(-,-)$ is a covariant functor
			\[
			\widehat{k_*}(-,-): 
			\mathsf{Mfld^{}}/
			B^2_{\mathrm{conn}}\uU(1)
			\to
			\mathsf{Ab}^{\mathbb{Z}}. 
			\]
			
			\item 
			$\mathcal{M}_*(-,-)$ is a covariant functor
			\[
			\mathcal{M}_*(-,-):
			\mathsf{Mfld^{}}/
			B^2_{\mathrm{conn}}\uU(1)
			\to
			\mathsf{Ch}^{\mathbb{R}}_*, 
			\]
			where $	\mathsf{Ch}^{\mathbb{R}}_*$ denotes the category of chain complexes over $\mathbb{R}$.
			\item
			$\ch'$ is a natural transformation
			\[
			{k_*}(-,-)
			\to
			H_*(\mathcal{M}_*(-,-)),
			\]
			\item $R, I, a$ are natural transformations
			\begin{align*}
				R &\colon 
				\widehat{k_*}(-,-) \to 
				\mathcal{M}^{\mathrm{clo}}_*,
				\\
				I &\colon 
				\widehat{k_*}(-,-) \to {k_*}(-,-),
				\\
				a &\colon 
				\mathcal{M}_{*+1}/ \mathrm{im}d_{\mathcal{M}} 
				\to \widehat{k_*}(-,-),
			\end{align*}
			where $d_\mathcal{M}$ denotes the chain differential of $\mathcal{M}_*$.  
		\end{itemize}
		
		For each object $(X, \widehat{\tau})$, the following conditions hold
		\begin{enumerate}[label=(\roman*)]
			\item 
			$\ch'$ is a group homomorphism and induces an isomorphism:
			\[
			\ch'\otimes \mathbb{R}: 
			{k_*}(X, {\tau})
			\otimes \mathbb{R}
			\xrightarrow{\simeq}
			H_*(\mathcal{M}_*(X, \widehat{\tau})),
			\]
			\item 
			The following diagram commutes:
			% https://q.uiver.app/#q=WzAsNSxbMCwxLCJcdFxcd2lkZWhhdHtrXyp9KFgsXFx3aWRlaGF0e1xcdGF1fSkiXSxbMSwxLCJcXG1hdGhjYWx7TX1ee1xcbWF0aHJte2Nsb319XyooWCwgXFx3aWRlaGF0e1xcdGF1fSkiXSxbMCwyLCJ7a18qfShYLHtcXHRhdX0pIl0sWzAsMCwiXFxtYXRoY2Fse019X3sqKzF9KFgsIFxcd2lkZWhhdHtcXHRhdX0pLyBcXG1hdGhybXtpbX1kX3tcXG1hdGhjYWx7TX19ICJdLFsxLDIsIkhfKihcXG1hdGhjYWx7TX1fKihYLCBcXHdpZGVoYXR7XFx0YXV9KSkiXSxbMCwxLCJSIl0sWzMsMCwiYSIsMl0sWzAsMiwiSSIsMl0sWzEsNF0sWzIsNCwiXFxtYXRocm17Y2h9JyIsMl0sWzMsMSwiZF97XFxtYXRoY2Fse019fSAiXV0=
			\[\begin{tikzcd}
				{\mathcal{M}_{*+1}(X, \widehat{\tau})/ \mathrm{im}d_{\mathcal{M}} } \\
				{	\widehat{k_*}(X,\widehat{\tau})} & {\mathcal{M}^{\mathrm{clo}}_*(X, \widehat{\tau})} \\
				{{k_*}(X,{\tau})} & {H_*(\mathcal{M}_*(X, \widehat{\tau}))}
				\arrow["a"', from=1-1, to=2-1]
				\arrow["{d_{\mathcal{M}} }", from=1-1, to=2-2]
				\arrow["R", from=2-1, to=2-2]
				\arrow["I"', from=2-1, to=3-1]
				\arrow[from=2-2, to=3-2]
				\arrow["{\mathrm{ch}'}"', from=3-1, to=3-2]
			\end{tikzcd}\]
			\item The following sequence is exact:
			\begin{align*}\label{eq_exact_diffhom}
				{k_{*+1}}(X,{\tau}) \xrightarrow{\mathrm{ch}' }	\mathcal{M}_{*+1}(X, \widehat{\tau})/ \mathrm{im}d_{\mathcal{M}} 
				\xrightarrow{a} 
				\widehat{k_*}(X,\widehat{\tau}) \xrightarrow{I} 
				{k_*}(X,{\tau})
				\to 0.
			\end{align*}
		\end{enumerate}
	\end{defi}
	\noindent
	We suppress the pair $(X, \widehat{\tau})$ from the natural transformations when the context is clear.

	\section{Differential twisted $\Spinc$-bordism and its Anderson dual}
	
	The aim of this section is to construct differential extensions of twisted
	$\Spinc$-bordism and of its Anderson dual, in the sense of
	Definition~\ref{defi:diffextho} and Definition~\ref{defi:diffextco},
	respectively. We first recall, in Section~\ref{sec:reviewWang},
	Wang's construction~\cite{Wang07} of twisted $\Spinc$-structures and
	twisted $\Spinc$-bordism. Section~\ref{sec:diff_spinc_struc} then gives
	a differential refinement of the notion of twisted $\Spinc$-structure.
	
	With the above preparation, Section~\ref{sec:diff_spinc_bord} develops
	the differential model for twisted $\Spinc$-bordism. We begin by studying
	the coefficient rings
	$V^{\Spinc}_\bullet=\Omega^{\Spinc}_\bullet(\pt)\otimes\mathbb{R}$
	and
	$N^{\bullet}_{\Spinc}=\Hom(\Omega^{\Spinc}_\bullet(\pt),\mathbb{R})$,
	and record how the curvature form $H$ deforms the corresponding
	de Rham differentials. On this basis, we construct the twisted de Rham
	chain complex $\mathcal{M}^{\Spinc}_*(-,-)$ from de Rham currents
	$\Omega_*(X; V_{\bullet}^{\Spinc})$, together with the twisted
	Chern--Weil map defined through its pairing with
	$\Omega^*(X; N^{\bullet}_{\Spinc})$. We then verify that the resulting
	model satisfies Definition~\ref{defi:diffextho}.
	
	In Section~\ref{sec:Andersondual}, we define the Anderson dual to
	twisted $\Spinc$-bordism homotopically using parametrized spectra, then
	construct its differential model
	$
	\big(\widehat{I\Omega^{\Spinc}_{\dR}}\big)^{*}(X,\widehat{\tau})
	$
	following the formalism of~\cite{Yam1,Yam23}. Finally, in
	Sections~\ref{sec: cob} and~\ref{sec:mul}, we discuss the differential
	model for twisted $\Spinc$-cobordism, as well as twisted differential
	multiplication and pushforward.

	\subsection{Review of twisted $\Spinc$-structures}\label{sec:reviewWang}
	
	In this section, we recall the parametrized spectra and geometric cycle models of twisted $\Spinc$-bordism. 
	Let $\xi_k$ denote the universal bundle over $B\Spinc(k)$. The Madsen--Tillmann spectrum is defined to be the Thom space
	\[
	MT\Spinc(k):=\mathrm{Th}(-\xi_k). 
	\]
	We denote its $\Omega$-spectrification by $MT\Spinc$, which is canonically homotopy equivalent to the Thom spectrum $M\Spinc$.
	There is a homotopy fiber sequence
	\[
	K(\mathbb{Z},2)\longrightarrow B\Spinc \longrightarrow B\SO \xrightarrow{\,W_3\,} K(\mathbb{Z},3).
	\]
	The induced $K(\mathbb{Z},2)$-action on $B\Spinc$ lifts to an action on $MT\Spinc$. 
	Given a topological twist $\tau$ over $X$, let $P_\tau\to X$ be the principal $K(\mathbb{Z},2)$-bundle classified by $\tau$ and form the associated bundle of spectra
	\[
	P_{\tau}(MT\Spinc):=
	P_{\tau} \times_{K(\mathbb{Z},2)} MT\Spinc \longrightarrow X.
	\]
	The twisted homology theory is
	\[
	MT\Spinc_n(X,\tau)
	:=\pi_n\big(r_! P_{\tau}(MT\Spinc)\big)
	=\pi_n\big(P_{\tau}(MT\Spinc)/X\big),
	\]
	in the sense of \eqref{def: twicoho}. 
	
	\medskip

	The obstruction to $\Spinc$-structures is encoded by the natural transformation
	\begin{equation}\label{topW3}
		W_3\colon B\SO \longrightarrow B^{2}\uU(1).
	\end{equation}
	Let $M$ be a compact manifold with a smooth map $f\colon M\to X$, and let $E\to M$ be an oriented vector bundle with stable classifying map $f_E\colon M\to B\SO$. A twisted $\Spinc$-structure is described by
	
	\begin{defi}[\cite{Wang07}]\label{def:topstruc}
		A $\tau$-twisted $\Spinc$-structure on $E$ is a homotopy
		\[
		\eta\colon \tau \circ f \simeq W_3\circ f_E
		\]
		in the following diagram 
		% https://q.uiver.app/#q=WzAsNCxbMCwxLCJYIl0sWzAsMCwiTSJdLFsxLDAsIkJcXG1hdGhybXtTT30iXSxbMSwxLCJCXjJcXG1hdGhybXtVfSgxKSJdLFsyLDMsIldfMyJdLFsxLDIsImZfRSJdLFsxLDAsImYiLDJdLFswLDMsIlxcdGF1IiwyXSxbMiwwLCJcXGV0YSIsMix7ImxldmVsIjoyLCJzdHlsZSI6eyJib2R5Ijp7Im5hbWUiOiJkYXNoZWQifX19XV0=
		\[\begin{tikzcd}
			M & {B\mathrm{SO}} \\
			X & {B^2\mathrm{U}(1)}
			\arrow["{f_{E}}", from=1-1, to=1-2]
			\arrow["f"', from=1-1, to=2-1]
			\arrow["\eta"', Rightarrow, dashed, from=1-2, to=2-1]
			\arrow["{W_3}", from=1-2, to=2-2]
			\arrow["\tau"', from=2-1, to=2-2]
		\end{tikzcd}\]		
	\end{defi}
	\noindent
	Two such structures $\eta,\eta'$ are said to be equivalent if they are homotopic relative to the endpoints. Write $\Struc{E}$ for the groupoid whose objects are $\tau$-twisted $\Spinc$-structures on $E$ and whose morphisms are these homotopies.
	The groupoid $\Struc{E}$ is nonempty if and only if the Freed-Witten condition 
	\[
	f^{*}[\tau]=W_3(E),
	\]
	holds in $H^{3}(M;\mathbb{Z})$. The twisted cycles are then described as follows
	\begin{defi}\label{def:topman}
		A $\tau$-twisted $\Spinc$-manifold over $X$ is defined to be a quadruple
		\[
		(M, f, f_{TM}, \eta),
		\]
		where 
		\begin{itemize}
			\item $M$ is a compact oriented manifold.
			\item $f: M\to X$ is a smooth map.
			\item $f_{TM}$ is the stable classifying map of the tangent bundle $TM$.
			\item $\eta \in \Struc{TM}$ is a $\tau$-twisted $\Spinc$-structure on $TM$.  
		\end{itemize} 
	\end{defi}
	\noindent
	Two quadruples $(M,f,f_{TM},\eta)$ and $(M',f',f_{TM'},\eta')$ are said to be isomorphic if there is an orientation-preserving diffeomorphism $h:M\to M'$ 
	together with homotopies
	\[
	\alpha: f \simeq f'\circ h, \qquad \beta: f_{TM} \simeq f_{TM'} \circ h,
	\]
	such that the composition
	$
	(W_3\circ \beta)* (\eta'\circ h)* (\tau\circ \alpha)^{-1} 
	$ is homotopic to $\eta$.

	\begin{defi}\label{def:topbord}
		The $n$-dimensional $\tau$-twisted $\Spinc$-bordism group 
		\[
		\Omega^{\Spinc}_n(X, \tau):=
		\{
		(M, f, f_{TM}, \eta)
		\}/\sim
		\]
		is defined as 
		the group of all isomorphism classes of closed $\tau$-twisted $\Spinc$-manifolds over $X$
		modulo
		boundaries of ${\tau}$-twisted $(n+1)$-manifolds. The group structure is given by disjoint union.

	\end{defi}
	
	Wang's theorem can be stated as follows. 
	
	\begin{thm}[\cite{Wang07}]\label{thm:Wang}
		There is a Pontryagin-Thom identification: 
		\begin{align}
			\Omega^{\Spinc}_n(X, \tau)
			\cong
			M\Spinc_n(X,\tau)
			\cong
			MT\Spinc_n(X,\tau). 
		\end{align}
	\end{thm}
	\noindent
	This theorem identifies the above geometric cycle model for twisted $\Spinc$-bordism with the homotopy model. 
	
	\subsection{Differential twisted $\Spinc$-structures}\label{sec:diff_spinc_struc}
	
	This section develops the differential refinement of twisted $\Spinc$-structures (Definition~\ref{def: difftwiSpincStru}). We begin with a truncated Deligne model:
	\[
	B_{\nabla}^{2}\uU(1)
	:=
	\mathrm{DK}
	\bigl(
	\underline{\uU(1)}
	\xrightarrow{\tilde d}
	\Omega^{1}
	\bigr),
	\]
	which is the Dold--Kan image of the Deligne complex $\bigl(\underline{\uU(1)}\xrightarrow{\tilde d}\Omega^{1}\bigr)$ concentrated in degrees $2,1$. 	
	There are canonical forgetful functors, obtained by consecutively discarding the $2$- and $1$-form data,
	\[
	B_{\mathrm{conn}}^{2}\uU(1) 
	\xrightarrow{\iota_2}
	B^2_\nabla\uU(1)
	\xrightarrow{\iota_1}
	B^2\uU(1).
	\]
	We now refine the topological obstruction map \eqref{topW3} by constructing a canonical natural transformation
	\begin{align}\label{nablaW3}
		W_3^\nabla:
		B_\nabla\SO
		\to
		B^2_\nabla\uU(1),
	\end{align}
	such that the diagram
	\[
	\begin{tikzcd}
		{B_{\nabla}\mathrm{SO}} & {	B_{\nabla}^{2}\mathrm{U}(1)} \\
		{B\mathrm{SO}} & {	B^{2}\mathrm{U}(1)}
		\arrow["{W_3^{\nabla}}", from=1-1, to=1-2]
		\arrow[from=1-1, to=2-1]
		\arrow[from=1-2, to=2-2]
		\arrow["{W_3}", from=2-1, to=2-2]
	\end{tikzcd}
	\]
	commutes, where $B_\nabla\SO$ denotes the simplicial presheaf of principal $\SO$-bundles with connection (cf. \cite{FH13}). 
	Concretely, \eqref{topW3} can be described in terms of cocycles as follows. 
	For a $0$-simplex in $B\SO(X)$ represented by a principal $\SO(n)$-bundle $P\to X$ with cocycles
	$
	g_{ij}\colon U_{ij}\to \SO(n),
	$
	choose local lifts $\widetilde g_{ij}\colon U_{ij}\to \mathrm{Spin}^{c}(n)$ covering $g_{ij}$.
	On triple overlaps, set
	\[
	\epsilon_{ijk}=
	\widetilde g_{jk}\,
	\widetilde g_{ki}\,
	\widetilde g_{ij}
	\colon U_{ijk}\to \uU(1),
	\]
	which is a \v{C}ech $2$-cocycle whose cohomology class is independent of the cover and the lifts, defining a $0$-simplex in $B^{2}\uU(1)(X)$. 
	
	For the differential refinement $W_3^{\nabla}$, take a $0$-simplex in $B_\nabla\SO(X)$ represented by a principal $\SO(n)$-bundle $P\to X$ and connection $\Gamma$, with local data
	\[
	g_{ij}\colon U_{ij}\to \SO(n), \qquad
	\Gamma_i \in \Omega^1(U_i,\so_n),
	\]
	satisfying
	\[
	\delta g=1, \qquad
	\Gamma_j= g_{ij}^{-1} \Gamma_i g_{ij} 
	+ g_{ij}^{-1}\,d g_{ij}. 
	\]
	Define $\epsilon_{ijk}$ as above. 
	Let $\mu\in \Omega^{1}(\Spinc(n),\spinc_n)$ be the Maurer--Cartan form and let $\pi_{\mathfrak{u}_1}$ denote projection to the central $i\mathbb{R}$-summand. Then the principal $\uU(1)$-connection on $\Spinc(n)\to \SO(n)$ is given by
	$
	\pi_{\mathfrak{u}_1}\mu
	\in
	\Omega^{1}(\Spinc(n), i\mathbb{R}).
	$
	Pulling back along the chosen lifts $\widetilde g_{ij}$, set
	\begin{align}\label{connection}
		A_{ij}
		=
		\tfrac{1}{2\pi i}
		\widetilde g_{ij}^{*}
		(	\pi_{\mathfrak{u}_1} \mu)
		\in\Omega^{1}(U_{ij}).
	\end{align}
	On triple overlaps one checks
	\[
	(\delta A)_{ijk}
	=
	A_{jk}-A_{ik}+A_{ij}
	=
	\tfrac{1}{2\pi i}
	\big(
	\widetilde g_{jk}^{*}
	-\widetilde g_{ik}^{*}
	+\widetilde g_{ij}^{*}
	\big)
	(\pi_{\mathfrak{u}_1} \mu)
	=
	\tfrac{1}{2\pi i}\,d\log(\epsilon_{ijk}).
	\]
	Thus $(\epsilon_{ijk},A_{ij})$ is a \v{C}ech-Deligne $2$-cocycle in $B^2_{\nabla}\uU(1)(X)$.

	By the same method, for a $1$-simplex in $B_\nabla \SO(X)$  given by a gauge transformation, 
	one may construct a corresponding
	$1$-simplex 
	in $B_{\nabla}^{2}\uU(1)(X)$ 
	from chosen lifts of the gauge maps. 
	Thus one obtains the natural transformation \eqref{nablaW3}, which clearly refines \eqref{topW3} by construction.

	In fact, upon a choice of local 2-forms, \eqref{nablaW3} admits a further lift. By the Maurer--Cartan equation one has $d(\pi_{\mathfrak{u}_1}\mu)=0$, hence $dA_{ij}=0$ on each $U_{ij}$. We may thereby choose $B_i=0$ on each $U_i$, obtaining a \v{C}ech--Deligne cocycle $(\epsilon_{ijk},A_{ij},0)$ and a lift
	\begin{align}\label{connW3}
		W_{3}^{\mathrm{conn}}\colon  B_{\nabla}\SO
		\longrightarrow
		B_{\mathrm{conn}}^{2}\uU(1),
	\end{align}
	but since we may choose different $B_i$, this lift is not canonical. 
	
	\medskip
	Now we are ready to give our definition for differential twisted $\Spinc$-structure, as a differential refinement to Definition~\ref{def:topstruc}. 
	Fix a differential twist $\widehat{\tau}$ with underlying topological twist $\tau$. Let $M$ be a compact manifold with a smooth map $f\colon M\to X$, and let $E\to M$ be an oriented vector bundle with connection and stable classifying map 
	$f^{\nabla}_E\colon M \to B_\nabla \SO$. 
	
	\begin{defi}\label{def: difftwiSpincStru}
		A \emph{differential $\widehat{\tau}$-twisted $\Spinc$-structure} on $E$ 
		is a $1$-simplex $\widehat{\eta}$ in $B_\nabla^{2}\uU(1)(M)$ connecting the
		$0$-simplices
		\[
		\widehat\eta\colon  \iota_2\widehat{\tau} \circ f 
		\to
		W_3^\nabla\circ f^{\nabla}_E,
		\]
		in the following diagram
		\[
		\begin{tikzcd}
			M & {B_{\nabla}\mathrm{SO}} \\
			X & {B^2_{\nabla}\mathrm{U}(1)}
			\arrow["{f^{\nabla}_{E}}", from=1-1, to=1-2]
			\arrow["f"', from=1-1, to=2-1]
			\arrow["{\widehat{\eta}}"', Rightarrow, dashed, from=1-2, to=2-1]
			\arrow["{W_3^\nabla}", from=1-2, to=2-2]
			\arrow["{\iota_2\widehat{\tau}}"', from=2-1, to=2-2]
		\end{tikzcd}
		\]
	\end{defi}
	\noindent
	Two differential $\widehat{\tau}$-twisted $\Spinc$-structures $\widehat\eta$ and $\widehat\eta'$ are said to be {equivalent} if there is a $2$-simplex in $B^{2}_{\nabla}\uU(1)(M)$ interpolating between them {relative to the endpoints}.
	Let $\difStruc{E}$ be the groupoid with these objects and morphisms.
	As shown in \cite{MS00}, there is a natural isomorphism of Deligne cohomology groups
	\[
	H^2\bigl(M; \underline{\uU(1)}
	\xrightarrow{\tilde d}
	\Omega^{1}\bigr)
	\cong
	H^2\bigl(M; \underline{\uU(1)}\bigr),
	\]
	hence the forgetful functor
	\begin{align}\label{prop:surj}
		\difStruc{E} \longrightarrow
		\Struc{E}
	\end{align}
	is essentially surjective on objects. 
	For vector bundles $E_i\to M$ with differential $\widehat{\tau}_i$-twisted $\Spinc$-structures ($i=1,2$), addition of \v{C}ech-Deligne cocycles yields a differential $(\widehat{\tau}_1+\widehat{\tau}_2)$-twisted $\Spinc$-structure on $E_1\oplus E_2$, i.e.
	\begin{align}\label{directsum}
		\widehat{\mathbf{Spin^c_{\widehat{\tau}_1}}}(E_1)
		\times
		\widehat{\mathbf{Spin^c_{\widehat{\tau}_2}}}(E_2)
		\to
		\widehat{\mathbf{Spin^c_{\widehat{\tau}_1+\widehat{\tau}_2}}}(E_1 \oplus E_2).
	\end{align}

	For each $\widehat{\tau}$-twisted $\Spinc$-structure on $E$, there is an associated canonical global $2$-form on the base $M$. Choose a good open cover $\pi\colon \mathcal{U}\to M$ and write the $0$-simplex
	$ \widehat{\tau}\circ f$ in
	$B_{\mathrm{conn}}^{2}\uU(1)(M)$ 
	as a cocycle
	$(\epsilon_{ijk} , A_{ij}, B_{i})$, 
	and 
	$W_3^\nabla \circ f_E^\nabla  $
	as a cocycle
	$(\epsilon'_{ijk} , A'_{ij})$ in $B^2_{\nabla}\uU(1)(M)$.  
	By definition, $\widehat{\eta}$ is a $1$-simplex
	$(h_{ij}, \lambda_{i})$
	in $B_{\nabla}^{2}\uU(1)(M)$ 
	from
	$(\epsilon_{ijk} , A_{ij})$ to 
	$(\epsilon'_{ijk} , A'_{ij})$,
	satisfying
	\[
	\delta h = \epsilon' \epsilon^{-1},
	\quad
	\delta \lambda 
	= A' - A - \frac{1}{2\pi i}d\log h.
	\]
	Clearly
	$\delta d\lambda
	=d \delta \lambda
	=d (A'-A)
	$.
	By the cocycle conditions we have 
	$dA=\delta B$ and $ dA'=0$, so
	$\delta(d\lambda+B)=0$. Hence there is a global $2$-form $\kappa(\widehat{\eta})\in\Omega^2(M)$ obtained by patching the local $2$-forms, with
	\begin{align}\label{kappa_local}
		\pi^*\kappa(\widehat\eta)=d\lambda+B,
		\quad 
		d\kappa(\widehat\eta)=f^*H. 
	\end{align}
	It is straightforward to check that this construction is independent of the cover. 
	For equivalent $\widehat{\tau}$-twisted $\Spinc$-structures $\widehat{\eta}$ and $\widehat{\eta}'$, there is a $0$-cochain $r$ such that $\delta r=h'h^{-1}$ and 
	$\lambda'=\lambda-\tfrac{1}{2\pi i}d\log r$, then one has $\kappa(\widehat{\eta})=\kappa(\widehat{\eta'})$ by construction. 
	
	Summarizing, we obtain a well-defined assignment
	\begin{align}\label{kappa}
		\kappa\colon 
		\pi_0\difStruc{E} 
		\longrightarrow
		\Omega^2(M).
	\end{align}
	This 2-form construction will be used repeatedly in defining the twisted Chern--Weil maps throughout this paper.

	\subsection{Differential twisted $\Spinc$-bordism}\label{sec:diff_spinc_bord}
	
	In this section, we give a differential refinement to twisted $\Spinc$-bordism in the sense of Definition~\ref{defi:diffextho}. 
	Namely, we assemble
	\begin{align}\label{diffSpincdata2}
		\big(
		\widehat{\Omega^{\Spinc}_*}(-,-),\;
		\mathcal{M}_*^{\Spinc}(-,-),\;
		\ch'^{\Spinc},\;
		R^{\Spinc},\;
		I^{\Spinc},\;
		a^{\Spinc}
		\big),
	\end{align}
	where $\widehat{\Omega^{\Spinc}_*}(-,-)$ and $\mathcal{M}_*^{\Spinc}(-,-)$ are covariant functors on the category of manifolds with differential twists, and
	$\ch'^{\Spinc}$, $R^{\Spinc}$, $I^{\Spinc}$, $a^{\Spinc}$ are natural transformations.
	Given an object $(X,\widehat{\tau})$, 
	\begin{itemize}
		\item The \emph{differential twisted} $\mathrm{Spin}^c$-\emph{bordism group} 
		$\widehat{\Omega^{\mathrm{Spin}^c}_n}(X,\widehat{\tau})$ 
		is generated by quintuples
		\[
		(M, f, f^{\nabla}_{TM}, \widehat{\eta}, \phi),
		\]
		called \emph{differential twisted} $\mathrm{Spin}^c$-\emph{cycles} over $X$, 
		where $\phi \in \Omega_{n+1}(X;V^{\mathrm{Spin}^c}_\bullet)/\mathrm{im}\,\partial_H$ 
		is represented by a de Rham current with coefficients in 
		$V^{\mathrm{Spin}^c}_\bullet = \Omega^{\mathrm{Spin}^c}_\bullet(\mathrm{pt}) \otimes \mathbb{R}$ 
		(see Definition~\ref{def: diffSpincbordism}).  
		The quadruple $(M, f, f^{\nabla}_{TM}, \widehat{\eta})$, 
		called a \emph{geometric twisted} $\mathrm{Spin}^c$-\emph{chain} over $X$, 
		consists of a compact oriented Riemannian manifold $M$ over $X$, 
		equipped with a differential twisted $\mathrm{Spin}^c$-structure $\widehat{\eta}$ on its tangent bundle $TM$, 
		whose classifying map is $f^{\nabla}_{TM}$.
		
		\item The chain complex 
		$\mathcal{M}_*^{\mathrm{Spin}^c}(X,\widehat{\tau}) := (\Omega_{*}(X;V^{\mathrm{Spin}^c}_\bullet),\,\partial_H)$ 
		is the \emph{twisted de Rham chain complex} with coefficients in the graded ring 
		$V^{\mathrm{Spin}^c}_\bullet$, with twisted boundary operator  
		\[
		\partial_H := \partial + H \wedge (\mathbb{C}\mathrm{P}^1 \times -),
		\]
		deformed by the curvature $3$-form $H$ of the twist $\widehat{\tau}$.  
		
		\item $\ch'^{\mathrm{Spin}^c}$ is a homomorphism from the topological twisted 
		$\mathrm{Spin}^c$-bordism group to the homology of 
		$\mathcal{M}^{\mathrm{Spin}^c}_*(X,\widehat{\tau})$.
		
		\item $R^{\mathrm{Spin}^c}$ is the \emph{curvature map} sending a differential cycle 
		in $\widehat{\Omega^{\mathrm{Spin}^c}_*}(X,\widehat{\tau})$ 
		to a closed current in $\Omega_{*}(X;V^{\mathrm{Spin}^c}_\bullet)$.
		
		\item $I^{\mathrm{Spin}^c}$ is the \emph{forgetful map} that discards the differential data.
		
		\item $a^{\mathrm{Spin}^c}$ maps 
		$\phi \in \Omega_{*}(X;V^{\mathrm{Spin}^c}_\bullet)/\mathrm{im}\,\partial_H$ 
		to the formal cycle $(\emptyset,\emptyset,\emptyset,\emptyset,-\phi)$.
	\end{itemize}
	
	\begin{thm}\label{thm:diffspinc}
		The model \eqref{diffSpincdata2} is a differential extension of twisted $\Spinc$-bordism theory in the sense of Definition~\ref{defi:diffextho}.
	\end{thm}
	
	The proof proceeds in three steps. 
	In Section~\ref{sec: chaincpx} we define 
	$\mathcal{M}_*^{\Spinc}(X, \widehat{\tau})$ as the twisted de Rham complex $({\Omega_{*}(X;V^{\mathrm{Spin}^c}_{\bullet})}, \partial_H)$ for each $\widehat{\tau}: X\to B^2_{\mathrm{conn}}\uU(1)$.  Then in Section~\ref{sec: cw} we construct the dual complex $({\Omega^{*}(X;N_{\mathrm{Spin}^c}^{\bullet})}, D_H)$ and a twisted Chern--Weil map, thereby realizing $\ch'^{\Spinc}$ and $R^{\Spinc}$. Finally, in Section~\ref{sec: diffext} we describe $I^{\Spinc}$ and $a^{\Spinc}$, and verify the properties required by Definition~\ref{defi:diffextho}.

	\subsubsection{Twisted de Rham complexes with coefficients}\label{sec: chaincpx}
	For a general structure group $G$, one denotes the graded coefficient rings by
	\[
	V_{\bullet}^{G}=V^{-\bullet}_{G}:=\Omega_{\bullet}^{G}(\pt)\otimes \mathbb{R},
	\quad
	N^{\bullet}_{G}:=\Hom\big(\Omega_{\bullet}^{G}(\pt),\mathbb{R}\big).
	\]
	In our case of interest $G=\Spinc$, we have \cite{stong1966relations}
	\[
	V^{\Spinc}_{\bullet}\cong \mathbb{R}[u,x_4,x_8,x_{12},\dots],
	\quad
	N^{\bullet}_{\Spinc}\cong \mathbb{R}[\zeta,p_1,p_2,\dots].
	\]
	In $V^{\Spinc}_{\bullet}$, the generator $u$ has homological degree $2$ and is represented by $\CP^1$ with its canonical $\Spinc$ structure; $x_{4i}$ has homological degree $4i$. 
	In $N^{\bullet}_{\Spinc}$, the generator $p_i$ is the universal Pontryagin class of cohomological degree $4i$;  $\zeta$ is the canonical $\Spinc$-class of cohomological degree $2$, corresponding to half of the determinant line. 
	For brevity, we write $p_I$ and $x_J$ for monomials in the generators $\{p_i\}$ and $\{x_{4i}\}$ with index $I$ and $J$, respectively.
	
	\medskip
	
	\begin{rmk}
		Note that we use both homological and cohomological gradings for $V_{\bullet}^{G}=V^{-\bullet}_{G}$, but only cohomological grading for $N^{\bullet}_{\Spinc}$. 
	\end{rmk}
	
	\medskip
	
	We introduce a natural product of the coefficient rings
	\begin{align}\label{star}
		\star\colon N^{p}_{\Spinc}\otimes V^{q}_{\Spinc}\rightarrow N^{p+q}_{\Spinc},
		\quad
		(\varphi\star v)(Y):=\varphi(v\times Y),
	\end{align}
	where $Y$ represents a $(p+q)$-dimensional $\Spinc$-bordism class. Clearly, $(\varphi\star v_1)\star v_2=\varphi\star(v_1v_2)$ and $\varphi\star 1=\varphi$. In particular, when $p+q=0$, we recover the degreewise evaluation pairing
	\begin{align}\label{coeffpairing}
		\langle-,-\rangle\colon N^{p}_{\Spinc}\otimes V^{-p}_{\Spinc}\rightarrow \mathbb{R},
	\end{align}
	which is nondegenerate in matching degrees.
	The following algebraic observation 
	gives the basic compatibility between 
	$V^\bullet_{\Spinc}$ and  $N^\bullet_{\Spinc}$,
	which allows us to define the twisted differentials and compare them by adjunction.

	\begin{prop}\label{prop:adjointness}
		For any $\varphi\in N^\bullet_{\Spinc}$ one has
		\begin{align}\label{adj1}
			\varphi\star u=\partial_\zeta \varphi,
		\end{align}
		where $u$ is the degree $2$ generator in $V^{\Spinc}_\bullet$ and $\partial_\zeta$ is the operator of degree $-2$ on $N^{\bullet}_{\Spinc}$ defined  by taking derivative with respect to the variable $\zeta$.  Moreover, for every $a\in V^{\Spinc}_\bullet$,
		\begin{align}\label{adj2}
			\varphi\star(u\times a)=(\partial_\zeta\varphi)\star a .
		\end{align}
	\end{prop}
	
	\begin{proof}
		It suffices to prove \eqref{adj1} for a monomial $\varphi=p_I\zeta^k\in N^{p}_{\Spinc}$. Let $Y$ be any $(p-2)$-dimensional $\Spinc$-bordism class. Then by definition, 
		\[
		(\varphi\star u)(Y)=\varphi(u\times Y)
		=\int_{u\times Y} p_I\big(T(u\times Y)\big)
		\smile
		\zeta(u\times Y)^{k}.
		\]
		Since $p_i(Tu)=0$ for $i\ge 1$, the Whitney product formula gives
		$p_I\big(T(u\times Y)\big)=\pr_2^{*}p_I(TY)$.
		Denote $L_u$ and $L_Y$ as the determinant line bundles for $u$ and $Y$, then
		$c_1(L_{u\times Y})=\pr_1^{*}c_1(L_u)+\pr_2^{*}c_1(L_Y)$,
		hence $\zeta(u\times Y)=\pr_1^{*}\zeta(u)+\pr_2^{*}\zeta(Y)$. Expanding the terms,
		\[
		\zeta(u\times Y)^k=\sum_{i=0}^{k}\binom{k}{i}\,\pr_1^{*}\zeta(u)^{i}
		\smile
		\pr_2^{*}\zeta(Y)^{k-i}.
		\]
		Observe that only the term $i=1$ contributes under integration along the fiber $\pr_2\colon u\times Y\to Y$. Therefore
		\[
		(\varphi\star u)(Y)
		= k\int_{Y} p_I(TY)
		\smile
		\zeta(Y)^{k-1}
		=(\partial_\zeta(p_I\zeta^k))(Y)
		=\partial_\zeta \varphi(Y), 
		\]
		where $\int_{u}\zeta(u)=1$ for the canonical $\Spinc$ structure on $\CP^1$. This concludes the proof for \eqref{adj1}. 
		The identity \eqref{adj2} follows from graded commutativity of the Cartesian product in the $\Spinc$-bordism ring.
	\end{proof}

	\medskip
	
	Now we study the twisted de Rham homology with $\Spinc$-coefficients. 
	Let
	\[
	\Omega_{i}(X):=\Hom_{\cts}\big(\Omega^{i}(X),\mathbb{R}\big)
	\]
	denote the space of compactly supported de Rham $i$-currents on $X$, viewed as continuous linear functionals on smooth $i$-forms. 
	The current differential $\partial\colon \Omega_{i}(X)\to \Omega_{i-1}(X)$ is characterized by
	\[
	\langle \partial T,\alpha\rangle=\langle T,d\alpha\rangle, 
	\quad
	\text{for }\alpha\in\Omega^{i-1}(X).
	\]
	The complex $(\Omega_*(X),\partial)$ models the de Rham homology of $X$. 
	For $\omega\in\Omega^r(X)$ and $T\in\Omega_i(X)$, define the left $\Omega^*(X)$-action on currents by
	\[
	\langle 
	\omega\wedge T, -
	\rangle:=
	\langle T, \omega\wedge-
	\rangle.
	\]
	We refer to \cite[III.8]{dR} for a detailed account of de Rham theory.

	In the twisted case, fix a differential twist $\widehat{\tau}$ on $X$ with curvature denoted by $H\in\Omega_{\clo}^3(X)$, and let 
	\begin{align}\label{complexV}
		\Omega_{k}(X;V^{\Spinc}_{\bullet})
		:=
		\bigoplus_{i+j=k}
		\Omega_i(X)\otimes V^{\Spinc}_{j}
	\end{align}
	be the group of compactly supported currents with $V^{\Spinc}_{\bullet}$-coefficients with total degree $k$. 
	We deform the usual homological current differential $\partial$ by
	\[
	\partial_H:=\partial+H\wedge (u\times -)\colon
	\
	\Omega_{k}(X;V^{\Spinc}_{\bullet})
	\rightarrow
	\Omega_{k-1}(X;V^{\Spinc}_{\bullet}),
	\]
	where $(u\times-)$ denotes multiplication by $u\in V^{\Spinc}_2$ on the coefficient factor.
	Since $H$ is a closed odd form, one checks $\partial_H^2=0$.
	We denote the resulting homology by
	\[
	H_{k}(X;V^{\Spinc}_{\bullet},H)
	:=
	H_k\big(\Omega_{*}(X;V^{\Spinc}_{\bullet}),\partial_H\big),
	\]
	and write 
	\[
	\qquad
	\mathcal{M}_*^{\Spinc}(X,\widehat{\tau})
	:=
	\big(\Omega_{*}(X;V^{\Spinc}_{\bullet}),\partial_H\big),
	\]
	which is the crucial ingredient in our model \eqref{diffSpincdata2}. 
	Also, we construct its continuous dual by defining
	\begin{align}\label{complexN}
		\Omega^{k}(X;N_{\Spinc}^{\bullet})
		:=
		\bigoplus_{i+j=k}\,
		\Omega^{i}(X)\otimes N_{\Spinc}^{j},
	\end{align}
	and deform the usual exterior derivative $d$ by 
	\[
	D_H:=d+H\wedge\partial_\zeta\colon
	\Omega^{k}(X;N_{\Spinc}^{\bullet})
	\rightarrow
	\Omega^{k+1}(X;N_{\Spinc}^{\bullet}),
	\]
	where $\partial_\zeta$ is the operator on $N_{\Spinc}^{\bullet}$ defined in Proposition~\ref{prop:adjointness}. 
	More precisely, for a pure tensor $\omega\otimes p_I \zeta^k$,
	\[
	D_H(\omega\otimes p_I\zeta^k)
	=
	d\omega\otimes p_I\zeta^k
	+
	kH\wedge\omega\otimes p_I\zeta^{k-1}.
	\]
	$D_H^2=0$ follows from the fact that $H$ is a closed odd form. We denote the resulting cohomology by
	\[
	H^{k}(X;N_{\Spinc}^{\bullet},H)
	:=
	H^{k}\big(\Omega^{*}(X;N_{\Spinc}^{\bullet}),D_H\big).
	\]
	Since this complex will be used later for the twisted differential Anderson dual, we write
	\begin{align}\label{MIO}
		\mathcal{M}^*_{I\Omega}(X,\widehat{\tau})
		:=
		\big(\Omega^{*}(X;N_{\Spinc}^{\bullet}),D_H\big). 
	\end{align}
	The choice of the notation is justified in Theorem~\ref{thm:diffio}.
	
	\medskip
	
	\begin{rmk}
		By replacing $\Omega^*(X)$ by the complex of compactly supported forms $\Omega^*_c(X)$, we may similarly define
		$\Omega_c ^{*}(X;N_{\mathrm{Spin}^c}^{\bullet})$ with the restricted twisted differential $D_H$, which will be used in Section~\ref{sec:mul} when studying the twisted differential $\Spinc$-\textit{co}bordism and differential multiplication. 
	\end{rmk}
	
	\medskip
	
	Combining the pairing of forms and currents with the pairing (\ref{coeffpairing}), we define a pairing 
	\begin{align}\label{pairing}
		\langle-,-\rangle\colon \Omega^k(X;N_{\Spinc}^\bullet)\otimes \Omega_k(X;V^{\Spinc}_\bullet)
		\to\mathbb{R}, 
	\end{align}
	given by
	\[
	\langle\omega\otimes\varphi,  T\otimes v\rangle:=
	T(\omega)\cdot
	\langle\varphi,v\rangle, 
	\]
	for $\omega\in\Omega^i(X)$, $T\in\Omega_i(X)$, $\varphi\in N^{j}_{\Spinc}$ and $v\in V^{\Spinc}_{j}$ with $i+j=k$.
	Since $N_{\Spinc}^{j}$ and $V^{\Spinc}_{j}$ are finite dimensional in each degree, \eqref{pairing} induces a natural identification
	\begin{align}\label{duality}
		\Omega_{k}(X;V^{\Spinc}_{\bullet})
		\cong
		\Hom_{\cts}\big(\Omega^{k}(X;N_{\Spinc}^{\bullet}),\mathbb{R}\big).
	\end{align}
	Moreover, by Proposition~\ref{prop:adjointness}, the two differentials are adjoint under the pairing
	\begin{align}\label{duality_adjoint}
		\langle D_H\alpha,\beta\rangle
		=
		\langle \alpha,\partial_H\beta\rangle,
	\end{align}
	for $\alpha\in\Omega^{k-1}(X;N_{\Spinc}^{\bullet})$ and  
	$\beta\in\Omega_{k}(X;V^{\Spinc}_{\bullet})$. 
	In particular, $D_H$-closed forms pair trivially with $\partial_H$-boundaries.
	Hence the above pairing descends to the twisted (co)homology groups. This fact will be used in the construction of the twisted differential models in the following sections.
	
	\medskip
	
	\subsubsection{Twisted Chern--Weil construction}\label{sec: cw}
	In this subsection, 
	we first refine the notion of twisted $\Spinc$-manifold (Definition~\ref{def:topman}) with differential data, obtaining the group of geometric twisted $\Spinc$-chains $\widetilde{C}^{\Spinc}_i(X,\widehat{\tau})$. 
	Then we construct a twisted Chern--Weil map
	\begin{align}\label{cw}
		\cw\colon \widetilde{C}^{\Spinc}_i(X,\widehat{\tau})\longrightarrow \Omega_{i}(X;V^{\Spinc}_{\bullet}),
	\end{align}
	sending an $i$-dimensional geometric chain to a compactly supported $i$-current. This realizes the structure maps $\ch'^{\Spinc}$ and $R^{\Spinc}$ in \eqref{diffSpincdata2}.
	
	\begin{defi}\label{def: Spincchain}
		An $n$-dimensional geometric $\widehat{\tau}$-twisted $\Spinc$-chain over $X$ is a quadruple
		\[
		(M, f, f^{\nabla}_{TM}, \widehat{\eta}),
		\]
		where
		\begin{itemize}
			\item $M$ is a compact oriented Riemannian $n$-manifold with boundary, equipped with a collar embedding of $\partial M$, along which all data are assumed to be constant;
			\item $f\colon M\to X$ is a smooth map;
			\item $f^{\nabla}_{TM}\colon M\to B_{\nabla}\SO$ is the stabilized classifying map of $TM$ with connection.
			\item $\widehat{\eta}\in\difStruc{TM}$ is a differential $\widehat{\tau}$-twisted $\Spinc$-structure on $TM$.
		\end{itemize}
	\end{defi}
	\noindent
	An isomorphism $(M,f,f^{\nabla}_{TM},\widehat{\eta})\to(M',f',f^{\nabla}_{TM'},\widehat{\eta}')$ 
	is an orientation and collar-preserving isometry $h\colon M\to M'$,
	such that $f= f'\circ h$, and the differential twisted
	$\Spinc$-structures are identified by a 2-morphism
	$W_3^\nabla
	\circ
	(dh)
	\circ \widehat\eta
	\Longrightarrow
	h^*\widehat\eta'
	$
	in $B^2_{\nabla}\uU(1)(M)$ relative to boundary collars. 
	In particular, we have $\kappa(\widehat{\eta})=h^{*}\kappa(\widehat{\eta}')$. 
	The collection of such chains forms an abelian group $\widetilde{C}^{\Spinc}_n(X,\widehat{\tau})$ under disjoint union.
	The boundary map is defined by restricting along the collar:
	\[
	\partial(M,f,f^{\nabla}_{TM},\widehat{\eta})
	:=
	\big(\partial M,\ \partial f,\ f^{\nabla}_{T\partial M},\ \partial\widehat{\eta}\big),
	\]
	This makes $\widetilde{C}^{\Spinc}_*(X,\widehat{\tau})$ a chain complex.
	
	\begin{rmk}
		In this paper, 
		\emph{geometric} (co)chains are given by \emph{quadruples} carrying connections but no extra differential form or current term, while
		\emph{differential} (co)chains are \emph{quintuples} with an additional form or current term.
	\end{rmk}

	Let $E\to M$ be an oriented bundle with connection and a differential $\widehat{\tau}$-twisted $\Spinc$-structure $\widehat{\eta}\in\difStruc{E}$. Define a cochain map
	\[
	\widehat{\eta}^{*}\colon \Omega^{*}(X;N_{\Spinc}^{\bullet})\longrightarrow \Omega^{*}(M),
	\]
	such that on monomial generators
	\[
	\widehat{\eta}^{*}(\omega\otimes p_I \zeta^k)
	:=
	f^{*}\omega \wedge (f^{\nabla}_{E})^{*}p_I \wedge \kappa(\widehat{\eta})^{k},
	\]
	where $\kappa(\widehat{\eta})$ is the global $2$-form \eqref{kappa}, and $(f_{E}^{\nabla})^*p_I $ denotes the closed Pontryagin form on $M$ associated to the connection on $E$ and the Pontryagin class $p_I$. 
	Since $(f^{\nabla}_E)^{*}p_I$ is closed and $d\kappa(\widehat{\eta})=f^{*}H$, a direct computation gives
	\[
	\widehat{\eta}^{*}(D_H(\omega\otimes p_I \zeta^k))
	=
	d\big(\widehat{\eta}^{*}(\omega\otimes p_I \zeta^k)\big),
	\]
	so $\widehat{\eta}^{*}$ is indeed a cochain map between $(\Omega^{*}(X;N_{\Spinc}^{\bullet}),D_H)$ and $(\Omega^{*}(M),d)$.

	For an $n$-dimensional geometric chain $(M,f,f^{\nabla}_{TM},\widehat{\eta})$, define a current
	\begin{align}\label{cw2}
		\cw(M,f,f^{\nabla}_{TM},\widehat{\eta})\colon
		\Omega^{n}(X;N_{\Spinc}^{\bullet})\longrightarrow \mathbb{R},
	\end{align}
	such that for a generator
	$\omega \otimes p_I  \zeta^k  \in
	\Omega^{n}(X;N_{\Spinc}^{\bullet})$, 
	\begin{align}\label{cw3}
		\cw(M, f, f^{\nabla}_{TM}, \widehat{\eta}):
		\omega \otimes p_I  \zeta^k
		\mapsto
		\int_M 
		\widehat\eta^*(\omega \otimes p_I \zeta^k).
	\end{align}
	This yields
	\[
	\cw\colon \widetilde{C}^{\Spinc}_n(X,\widehat{\tau})\longrightarrow \Omega_{n}(X;V^{\Spinc}_{\bullet}),
	\]
	which is a chain map by Stokes theorem. By construction, we see that isomorphic geometric chains give identical Chern--Weil currents. 
	Along a change of Riemannian connections on $TM$, one observes the Pontryagin forms and $\kappa(\widehat{\eta})$ change by an exact form, then the resulting currents differ by a $\partial_H$-exact term. Equivalently, the homology class of $\cw(M,f,f^{\nabla}_{TM},\widehat{\eta})$ in $H_{n}(X;V^{\Spinc}_{\bullet},H)$ is independent of the choice of connection. 
	Hence $\cw$ descends to a homomorphism
	\begin{align}\label{eq: ch'^Spinc}
		\ch'^{\Spinc}\colon \Omega^{\Spinc}_n(X,\tau)\longrightarrow H_{n}(X;V^{\Spinc}_{\bullet},H). 
	\end{align}
	Furthermore, $\ch'^{\Spinc}$ is a rational isomorphism: 
	\begin{prop}\label{prop: dRSpinc}
		Tensoring with $\mathbb{R}$, the Chern--Weil map
		\[
		\ch'^{\Spinc}\otimes\mathbb{R}:
		\Omega^{\Spinc}_n(X,\tau)\otimes\mathbb{R}
		\rightarrow
		H_n\big(X;V^{\Spinc}_\bullet,H\big)
		\]
		is an isomorphism.
	\end{prop}
	
	\begin{proof}
		
		We have the twisted Atiyah-Hirzebruch spectral sequence, cf. Proposition~\ref{prop:AHSS}, 
		\[
		E^2_{p,q}\cong H_p\big(X;L_q(X,P_\tau(MT\Spinc))\big)
		\Longrightarrow \Omega^{\Spinc}_{p+q}(X,\tau).
		\]
		Tensoring with $\mathbb{R}$, we have 
		\[
		E^2_{p,q}\otimes\mathbb{R}\cong H_p(X;\mathbb{R})\otimes \big(\Omega^{\Spinc}_q(\pt)\otimes\mathbb{R}\big).
		\]
		The first possible nonzero higher differential is
		\[
		d^3=H\wedge(u\times-)\colon E^3_{p,q}\to E^3_{p-3,q+2}.
		\]

		On the other hand, we define an increasing filtration on the de Rham chain complex by filtering current degree,
		\[
		F_p \Omega_*
		(X;V^{\Spinc}_\bullet)=
		\bigoplus_{i\le p}
		\Omega_i(X)\otimes V^{\Spinc}_\bullet,
		\]
		which is compatible with the differential $\partial_H$. 
		The associated spectral sequence satisfies
		\[
		E'^0_{p,q}\cong \Omega_p(X)\otimes V^{\Spinc}_q.
		\]
		Since $V^{\Spinc}_q=0$ for $q$ odd, all even differentials $d'^{2r}$ vanish by parity, so
		\[
		E'^2_{p,q}\cong H_p(X;\mathbb{R})\otimes V^{\Spinc}_q
		\Longrightarrow
		H_{p+q}\big(X;V^{\Spinc}_\bullet,H\big),
		\]
		while the first possible nonzero higher differential is
		\[
		d'^3=H\wedge(u\times-)\colon E'^3_{p,q}\to E'^3_{p-3,q+2}.
		\]
		
		By construction, the twisted Chern--Weil map $\cw$ yields a morphism of filtered complexes which intertwines differentials. On the second page, we obtain the map
		\[
		\ch'^{\Spinc}_{E^2}\otimes\mathbb{R}
		\colon
		H_p(X;\mathbb{R})\otimes\big(\Omega^{\Spinc}_q(\pt)\otimes\mathbb{R}\big)
		\longrightarrow
		H_p(X;\mathbb{R})\otimes V^{\Spinc}_q,
		\]
		which is the identity on $H_p(X;\mathbb{R})$ and the canonical identification
		$\Omega^{\Spinc}_q(\pt)\otimes\mathbb{R}\cong V^{\Spinc}_q$ on coefficients. Thus $\ch'^{\Spinc}_{E^2}\otimes\mathbb{R}$ is an isomorphism, and naturality implies it commutes with $d^3$ and hence with all higher differentials. By the comparison theorem,  
		\[
		\ch'^{\Spinc}\otimes\mathbb{R}\colon
		\Omega^{\Spinc}_n(X,\tau)\otimes\mathbb{R}\rightarrow
		H_n\big(X;V^{\Spinc}_\bullet,H\big)
		\]
		is an isomorphism.
	\end{proof}
	
	\medskip
	
	\subsubsection[The differential cycle model]{The cycle model for differential twisted $\Spinc$-bordism}\label{sec: diffext}
	
	We now record our differential model for twisted $\Spinc$-bordism, and verify it satisfies Definition~\ref{defi:diffextho}.
	
	\begin{defi}\label{def: diffSpincbordism}
		Let $\widehat{\tau}\colon X\to B^2_{\mathrm{conn}}\uU(1)$ be a differential twist with curvature $H$. Define the differential $\widehat{\tau}$-twisted $\Spinc$-bordism group by
		\[
		\widehat{\Omega^{\Spinc}_{n-1}}(X,\widehat{\tau})
		:=\big\{(M,f,f^{\nabla}_{TM},\widehat{\eta},\phi)\big\}\big/\sim,
		\]
		where $(M,f,f^{\nabla}_{TM},\widehat{\eta})$ is a {closed} $(n-1)$-dimensional geometric $\widehat{\tau}$-twisted $\Spinc$-chain over $X$ (Def.~\ref{def: Spincchain}), and
		$\phi\in \Omega_n(X;V^{\Spinc}_{\bullet})/\mathrm{im}\partial_H$.
		The relation $\sim$ is generated by:
		\begin{itemize}
			\item \emph{Isomorphisms:}
			\[
			(M,f,f^{\nabla}_{TM},\widehat{\eta},\phi) \sim (M',f',f^{\nabla}_{TM'},\widehat{\eta}',\phi),
			\]
			for isomorphic geometric chains $(M,f,f^{\nabla}_{TM},\widehat{\eta})$ and 
			$(M',f',f^{\nabla}_{TM'},\widehat{\eta}')$.

			\item \emph{Additivity:} disjoint union on geometric chains and addition on $\phi$.
			
			\item \emph{Bordism:}
			\[
			(\partial W,\partial F,f^{\nabla}_{T\partial W},\partial\widehat{\eta},0)\ \sim\ (\emptyset,\emptyset,\emptyset,\emptyset,-\cw(W,F,f^{\nabla}_{TW},\widehat{\eta})),
			\]
			for any $n$-dimensional geometric $\widehat{\tau}$-twisted chain $(W,F,f^{\nabla}_{TW},\widehat{\eta})$.
		\end{itemize}
	\end{defi}

	Define the structure maps
	\[
	\begin{aligned}
		R^{\Spinc}\colon\ 
		\widehat{\Omega^{\Spinc}_{n-1}}(X,\widehat{\tau})
		&\longrightarrow
		\Omega_{n-1}^{\partial_H\text{-clo}}(X;V^{\Spinc}_{\bullet}),
		&
		(M,f,f^{\nabla}_{TM},\widehat{\eta},\phi)
		&\longmapsto
		\cw(M,f,f^{\nabla}_{TM},\widehat{\eta})-\partial_H\phi,\\
		I^{\Spinc}\colon\ 
		\widehat{\Omega^{\Spinc}_{n-1}}(X,\widehat{\tau})
		&\longrightarrow
		\Omega^{\Spinc}_{n-1}(X,\tau),
		&
		(M,f,f^{\nabla}_{TM},\widehat{\eta},\phi)
		&\longmapsto
		(M,f, f_{TM}, \eta),\\
		a^{\Spinc}\colon\ 
		\Omega_n(X;V^{\Spinc}_{\bullet})/\mathrm{im}\partial_H
		&\longrightarrow
		\widehat{\Omega^{\Spinc}_{n-1}}(X,\widehat{\tau}),
		&
		\phi
		&\longmapsto
		(\emptyset,\emptyset,\emptyset,\emptyset,-\phi).
	\end{aligned}
	\]
	It is straightforward to check the structure maps are well-defined.
	Now we are ready to prove Theorem~\ref{thm:diffspinc}, showing that our model \eqref{diffSpincdata2} defines a differential extension of twisted $\Spinc$-bordism theory.

	\begin{proof}[Proof of Theorem~\ref{thm:diffspinc}]
		The functoriality in $(X,\widehat{\tau})$ is clear from naturality of pushforward. 
		We verify the requirements (i)-(iii) in Definition \ref{defi:diffextho}. 
		The canonical isomorphism (i) is shown in Proposition~\ref{prop: dRSpinc}; the commutativity condition (ii) follows from the construction. It remains to check the exactness condition (iii), for the following sequence
		\begin{align}\label{spincSES}
			{\Omega}^{\mathrm{Spin}^c}_{n}(X, \tau)
			\xrightarrow{\ch'^{\Spinc}}
			\Omega _n(X;V^{\mathrm{Spin}^c}_{\bullet})/\mathrm{im}\partial_H
			\xrightarrow{a^{\mathrm{Spin}^c}}
			\widehat{\Omega^{\mathrm{Spin}^c}_{n-1}}(X, \widehat{\tau})
			\xrightarrow{I^{\mathrm{Spin}^c}}
			{\Omega}^{\mathrm{Spin}^c}_{n-1}(X, \tau)
			\longrightarrow
			0.
		\end{align}
		The surjectivity of $I^{\Spinc}$ is due to (\ref{prop:surj}). 
		
		For the exactness at $\widehat{\Omega^{\mathrm{Spin}^c}_{n-1}}(X, \widehat{\tau})$, 
		pick a representative $(M,f,f^{\nabla}_{TM},\widehat{\eta},\phi)$ in $\ker(I^{\Spinc})$, and we have $M$ bounds some $n$-dimensional geometric chain $B$. By the bordism equivalence relation, we have
		\[
		(M,f,f^{\nabla}_{TM},\widehat{\eta},\phi) \sim (\emptyset,\emptyset,\emptyset,\emptyset,-\cw(B)+\phi)=
		a^{\Spinc}([\cw(B)-\phi]),
		\]
		showing that 
		$\ker(I^{\Spinc}) \subset \mathrm{im}(a^{\Spinc})$. The reverse inclusion is clear.

		Next we show \(\im(\ch'^{\Spinc})=\ker(a^{\Spinc})\).
		For a representative $\phi$ in $\ker(a^{\Spinc})$, there exists an $n$-dimensional geometric chain
		$B=(W,F,f^\nabla_{TW},\widehat{\eta})$ such that
		\[
		(\partial W,\partial F,f^{\nabla}_{T\partial W},\partial\widehat{\eta},\cw(B))
		= (\emptyset,\emptyset,\emptyset,\emptyset,-\phi),
		\] 
		equivalently, 
		\[
		\partial W=\emptyset, \quad 
		\phi\equiv 
		-\cw(B)
		\mod \mathrm{im}\partial_H. 
		\]
		Thus $-W$ represents an $n$-bordism class, and $[\phi]=\ch'^{\Spinc}([-W])$, showing that  $\ker(a^{\Spinc})\subset\mathrm{im}(\ch'^{\Spinc})$. For the reverse inclusion, pick a closed $n$-chain $(W, F, f_{TW}, \eta)$ which represents a bordism class in $\Omega_n^{\Spinc}(X,\tau)$, and choose a geometric
		refinement $B=(W,F,f^\nabla_{TW},\widehat{\eta})$. 
		Since $\partial W=\emptyset$, the bordism equivalence relation gives
		\[
		0\sim
		(\partial W,\partial F,f^{\nabla}_{T\partial W},\partial\widehat{\eta},\cw(B))
		= (\emptyset,\emptyset,\emptyset,\emptyset,\cw(B)),
		\]
		implying 
		$
		a^{\Spinc}(\ch'^{\Spinc}([W]))=0.
		$ This concludes the proof. 
	\end{proof}
	
	\medskip
	\subsection{Differential Anderson dual to twisted $\Spinc$-bordism}\label{sec:Andersondual}
	In this subsection, we construct our differential model for the Anderson dual to twisted $\Spinc$-bordism
	\begin{align}\label{diffIOdata2}
		\Big(
		\bigl(
		\widehat{
			I\Omega^{\Spinc}_{\mathrm{dR}}
		}
		\bigr)^{*}(-, -),
		\mathcal{M}^*_{I\Omega}(-, -),
		\ch'_{I\Omega}, 
		R_{I\Omega},
		I_{I\Omega},
		a_{I\Omega}
		\Big).
	\end{align}
	For an object $(X, \widehat{\tau})$, 
	\begin{itemize}
		\item
		$(\widehat{I\Omega_{\mathrm{dR}}^{\mathrm{Spin}^c}})^n(X,\widehat{\tau})$ 
		is the abelian group of pairs $(\omega,h)$, where $\omega$ is a twisted closed differential form 
		valued in the $\mathrm{Spin}^c$-characteristic classes of total degree $n$, 
		and $h$ is an $\mathbb{R}/\mathbb{Z}$-valued functional on 
		$(n-1)$-dimensional differential twisted $\mathrm{Spin}^c$-bordism cycles, 
		satisfying a natural compatibility condition.
		
		\item The complex 
		$\mathcal{M}^*_{I\Omega}(X,\widehat{\tau}) := (\Omega^{*}(X;N_{\mathrm{Spin}^c}^{\bullet}), D_H)$ 
		is a twisted de~Rham \emph{cochain} complex 
		with coefficients $N_{\mathrm{Spin}^c}^{\bullet} 
		= \mathrm{Hom}(\Omega^{\mathrm{Spin}^c}_{\bullet}(\mathrm{pt}), \mathbb{R})$, 
		whose differential is deformed by the curvature $3$-form $H$,
		\[
		D_H = d + H \wedge \partial_{\zeta},
		\]
		where $\zeta$ is a degree $2$ generator in $N_{\mathrm{Spin}^c}^{\bullet}$.
		
		\item $\ch'_{I\Omega}$ is a homomorphism from topological twisted Anderson dual into the cohomology of $\mathcal{M}^*_{I\Omega}(X,\widehat{\tau})$.
		
		\item $R_{I\Omega}$ is the \emph{curvature map}, sending a pair $(\omega,h)$ to its curvature form $\omega$.
		
		\item $I_{I\Omega}$ is the \emph{forgetful map} that discards the differential data.
		
		\item $a_{I\Omega}$ assigns to each 
		$\alpha \in \Omega^{n-1}(X;N_{\mathrm{Spin}^c}^{\bullet})/\mathrm{im}\,D_H$ 
		a compatible pair in 
		$(\widehat{I\Omega_{\mathrm{dR}}^{\mathrm{Spin}^c}})^n(X,\widehat{\tau})$.
	\end{itemize}
	
	\begin{thm}\label{thm:diffio}
		The model (\ref{diffIOdata2}) 
		\[
		\Big(
		\bigl(
		\widehat{
			I\Omega^{\Spinc}_{\mathrm{dR}}
		}
		\bigr)^{*}(-, -),
		\mathcal{M}^*_{I\Omega}(-, -),
		\ch'_{I\Omega}, 
		R_{I\Omega},
		I_{I\Omega},
		a_{I\Omega}
		\Big)
		\]
		is a differential extension of the Anderson dual to twisted $\Spinc$-bordism theory, in the sense of Definition~\ref{defi:diffextco}. 
	\end{thm}
	To prove the theorem, 
	we start by reviewing the definition of Anderson duals, and Yamashita--Yonekura's differential models in Section~\ref{sec: untwistedAnderson}. 
	Then 
	in Section~\ref{sec: twistedAnderson}, we first give a topological definition of the Anderson dual to twisted $\Spinc$-bordism via parametrized spectra, then establish our differential model (\ref{diffIOdata2}). 
	After that, we 
	finish the proof of Theorem~\ref{thm:diffio} by verifying the desired properties of a twisted differential extension in Definition~\ref{defi:diffextco}.
	
	\subsubsection{Review of Anderson dual and Yamashita--Yonekura's model}\label{sec: untwistedAnderson}
	We start with a review of Anderson duality. 
	For an injective $\mathbb{Z}$-module $R$, the functor $$
	\Hom(\pi_*(-), R ): \mathsf{Sp} \to \mathsf{Ab}^{\mathbb{Z}}
	$$ is a cohomology theory, whose representing spectrum is denoted by $I_R$. 
	Setting $R= \mathbb{Q}$ and $\mathbb{Q}/\mathbb{Z}$, the quotient map $\mathbb{Q}\to\mathbb{Q}/\mathbb{Z}$ induces a natural transformation 
	\[
	\Hom(\pi_*(-), \mathbb{Q})
	\to
	\Hom(\pi_*(-), \mathbb{Q}/\mathbb{Z}),
	\]
	which corresponds to a map of spectra 
	\begin{align}\label{IQIQZ}
		I_{\mathbb{Q}} \to I_{\mathbb{Q}/\mathbb{Z}}. 
	\end{align}
	The \emph{Anderson dual of the sphere spectrum} $I_{\mathbb{Z}}$ is defined to be the homotopy fiber of (\ref{IQIQZ}). 
	For a general spectrum $E$, the \emph{Anderson dual of $E$} is defined as the function spectrum
	\[
	I_{\mathbb{Z}}E:=F(E, I_{\mathbb{Z}}). 
	\]
	There is a natural exact sequence: 
	\begin{align}\label{AndersonSES}
		0
		\to
		\mathrm{Ext}(E_{n-1}(X),\mathbb{Z})
		\to
		(I_{\mathbb{Z}}E)^n(X)
		\to
		\mathrm{Hom}(E_{n}(X),\mathbb{Z})
		\to
		0, 
	\end{align}
	together with a Picard groupoid description
	\begin{align}\label{AndersonPic}
		(I_{\mathbb{Z}}E)^n(X)
		\simeq 
		\pi_{0}\FunPic
		\big(
		\pi_{\le1}
		L(E\wedge X)
		_{1-n},
		(\mathbb{R}
		\to
		\mathbb{R}/\mathbb{Z})
		\big).
	\end{align}
	Here $L$ denotes $\Omega$-spectrification,  $\pi_{\le1}$ the fundamental Picard groupoid, $\pi_{0}\FunPic$ the group of natural isomorphism classes of functors of Picard groupoids, and 
	$
	\big(
	\mathbb{R}
	\to
	\mathbb{R}/\mathbb{Z}
	\big)
	$
	the Picard groupoid with objects in $\mathbb{R}/\mathbb{Z}$, morphisms given by $x_0\xrightarrow{y} x_1$ for each 
	$x_1-x_0=y \ \modZ$. Readers are referred to \cite{HS} for a detailed account. 
	
	\medskip
	
	In \cite{Yam1},  Yamashita and Yonekura construct a de Rham model for the Anderson dual to $G$-bordism theories. 
	Further in \cite{Yam23}, Yamashita describes a general framework for differential Anderson duals $(\widehat{IE})^*$ when a differential homology theory $\widehat{E}_*$ is provided. 
	For a structure group $G$, their de Rham model consists of the following data $
	\big(
	\bigl(
	\widehat{I\Omega^{G}_{\dR}}
	\bigr)^{*}, R, I, a
	\big)
	$, together with an $S^1$-integration $\int$. 
	Concretely, for each manifold $X$ and integer $n$, let 
	\[
	\bigl(
	\widehat{I\Omega^{G}_{\dR}}
	\bigr)^{n}(X)
	:=\{(\omega, h)\},
	\]
	where
	\begin{itemize}
		\item 
		$\omega \in
		\Omega^n_{\text{clo}}
		(X; N_{G}^{\bullet})  $ 
		is a $d$-closed 
		$N_{G}^{\bullet}$-valued form,
		\item $h:
		\widehat{
			\Omega^{G}_{n-1}
		}
		(X)
		\to\mathbb{R}/\mathbb{Z}$ 
		is a group homomorphism,
		\item they satisfy the compatibility condition 
		\[
		h\circ a=
		\modZ \circ
		\langle- ,\omega\rangle. 
		\]
	\end{itemize}
	
	\begin{thm}[\cite{Yam1}]
		$
		\big(
		\bigl(
		\widehat{I\Omega^{G}_{\dR}}
		\bigr)^{*}, R, I, a, \int
		\big)
		$
		is a differential extension with $S^1$-integration to the Anderson dual to $G$-bordism,   in the sense of \cite{BS10}.
	\end{thm}
	Our next task is describing a twisted refinement of the above model.
	
	\medskip
	\subsubsection{Twisted Anderson dual and the differential model}\label{sec: twistedAnderson}
	We first define the Anderson dual to twisted $\Spinc$-bordism via parametrized spectra, and then present our construction of its differential model.

	\begin{defi}\label{defi: toptwiAnd}
		For a manifold $X$ and a topological twist ${\tau}$ over $X$, define the Anderson dual to twisted $\Spinc$-bordism as the fiberwise dual
		of the twisted parametrized spectrum $P_\tau(MT\Spinc)$ over $X$. Concretely, 
		\[
		I^{\tau}_{\mathbb{Z}}MT\mathrm{Spin}^c
		:=
		F_X\bigl(P_\tau(MT\mathrm{Spin}^c),\;X\times I_{\mathbb{Z}}\bigr). 
		\]
	\end{defi}
	
	\begin{rmk}\label{rmk: eq}
		There is a canonical equivalence over $X$, cf. \cite[Ch.\ 12]{MaySig}
		\[
		F_X\bigl(P_\tau(MT\mathrm{Spin}^c),\;X\times I_{\mathbb{Z}}\bigr)
		\simeq
		P_{\tau}(I_{\mathbb{Z}}MT\Spinc),
		\]
		where $P_{\tau}(I_{\mathbb{Z}}MT\Spinc)$ is the associated bundle of spectra with fiber the Anderson dual $I_{\mathbb{Z}}MT\mathrm{Spin}^c$, with the structure group $K(\mathbb{Z},2)$ acting on $I_{\mathbb{Z}}MT\mathrm{Spin}^c$ via its action on $MT\mathrm{Spin}^c$. This equivalence
		identifies the left hand side, \emph{the Anderson dual of the twisted theory}, with the right hand side, \emph{the twisted Anderson dual}. Hence no ambiguity will arise from these terminologies in the sequel.
	\end{rmk}

	As in \eqref{def: twicoho}, define the twisted cohomology groups by
	\begin{align}\label{homAnderson}
		\bigl(I\Omega^{\mathrm{Spin}^c}\bigr)^n(X,\tau)
		:=
		\pi_{-n}
		\bigl( 
		r_*(I^{\tau}_{\mathbb{Z}}MT\mathrm{Spin}^c)
		\bigr)=
		\pi_{-n}\bigl(F(P_\tau(MT\mathrm{Spin}^c)/X,\;I_{\mathbb{Z}})\bigr).
	\end{align}
	Then by the properties of Anderson duals \eqref{AndersonSES} and \eqref{AndersonPic}, we have the following short exact sequence
	\begin{align}\label{AndersonSES_twistedSpinc}
		0
		\to
		\mathrm{Ext}(\Omega_{n-1}^{\mathrm{Spin}^c}(X,\tau),\mathbb{Z})
		\to
		(I\Omega^{\mathrm{Spin}^c})^n(X,\tau)
		\to
		\mathrm{Hom}(\Omega_{n}^{\mathrm{Spin}^c}(X,\tau),\mathbb{Z})
		\to
		0,
	\end{align}
	and the Picard groupoid identification
	\begin{align}\label{picard_spinc}
		\big(
		I\Omega^{\Spinc}
		\bigr)^{n}
		(X, \tau)
		\simeq 
		\pi_{0}\FunPic
		\big(
		\pi_{\le1}
		L
		(
		P_{\tau}(MT\Spinc)/X
		)_{1-n},
		(\mathbb{R}
		\to
		\mathbb{R}/\mathbb{Z})
		\big)
		.
	\end{align}

	We are now ready to construct the differential extension promised in (\ref{diffIOdata2})
	\[
	\Big(
	\bigl(
	\widehat{
		I\Omega^{\Spinc}_{\mathrm{dR}}
	}
	\bigr)^{*}(-, -),
	\mathcal{M}^*_{I\Omega}(-, -),
	\ch'_{I\Omega}, 
	R_{I\Omega},
	I_{I\Omega},
	a_{I\Omega}
	\Big),
	\]
	as a twisted generalization of Yamashita--Yonekura's model  \cite{Yam1}. Given a manifold $X$ and a differential twist $\widehat{\tau}$ on $X$ with underlying topological twist $\tau$, we define the differential Anderson dual to twisted $\Spinc$-bordism as follows. 
	
	\begin{defi}\label{def: diffAnderson}
		For each integer $n$, 
		\[
		\bigl(
		\widehat{
			I\Omega^{\Spinc}_{\mathrm{dR}}
		}
		\bigr)^{n}(X, \widehat{\tau})
		:=\{(\omega, h)\},
		\]
		where
		\begin{itemize}
			\item 
			$\omega \in
			\Omega^n_{D_H\mathrm{-clo}}
			(X; N_{\Spin^c}^{\bullet})  $,
			%is a $D_H$-closed $N_{\Spin^c}^{\bullet}$-valued form,
			\item $h:
			\widehat{
				\Omega^{\Spinc}_{n-1}
			}
			(X, \widehat{\tau})
			\to\mathbb{R}/\mathbb{Z}$ 
			is a group homomorphism,
			\item $\omega$ and $h$ satisfy the compatibility condition 
			\[
			h\circ a^{\Spinc}=
			\mathrm{mod}\mathbb{Z} \circ
			\langle- ,\omega\rangle.
			\]
		\end{itemize}
	\end{defi}
	\noindent
	We now describe the remaining ingredients in (\ref{diffIOdata2}).
	Recall the twisted de Rham complex
	\[
	\mathcal{M}^*_{I\Omega}(X, \widehat{\tau})
	:=
	\big(
	\Omega^{*}
	(X;N_{\mathrm{Spin}^c}^{\bullet}),
	D_H
	\big),
	\]
	is already defined in (\ref{MIO}). 
	The curvature map $R_{I\Omega}$ is defined as the projection to the first component
	\[
	R_{I\Omega}:
	\bigl(
	\widehat{I\Omega^{\Spinc}_{\mathrm{dR}}}
	\bigr)^{n}(X, \widehat{\tau})
	\to
	\Omega^n_{D_H\mathrm{-clo}}
	(X; N_{\Spin^c}^{\bullet}).
	\]
	$a_{I\Omega}$ is defined by the universal property of the pullback,
	\[
	a_{I\Omega}: 
	{\Omega ^{n-1}(X;N_{\mathrm{Spin}^c}^{\bullet})/\mathrm{im}D_H}
	\to
	\bigl(
	\widehat{
		I\Omega^{\Spinc}_{\mathrm{dR}}
	}
	\bigr)^{n}(X, \widehat{\tau}), \quad
	\alpha\mapsto 
	\big(D_H\alpha, j(\alpha)\big),
	\]
	where $j(\alpha)=\mathrm{mod}\mathbb{Z}\circ (R^{\Spinc})^*\circ\langle-,\alpha\rangle$ is given by the following composition. 
	
	\[
	\begin{aligned}
		j: 
		{\Omega^{n-1}(X; N^{\bullet}_{\Spinc})/\mathrm{im}D_H}
		& \to
		\mathrm{Hom}\big(\Omega_{n-1}^{\partial_H\mathrm{-clo}}(X;V^{\Spinc}_{\bullet}),\mathbb{R}\big)
		\\
		&\xrightarrow{(R^{\Spinc})^*}
		\mathrm{Hom}\big({\widehat{\Omega^{\mathrm{Spin}^c}_{n-1}}(X, \widehat{\tau})} ,\mathbb{R}\big)
		\xrightarrow{\mathrm{mod}\mathbb{Z}} 
		\mathrm{Hom}\big(
		{\widehat{\Omega^{\mathrm{Spin}^c}_{n-1}}(X, \widehat{\tau})} ,
		\mathbb{R}/\mathbb{Z}\big). 
	\end{aligned}
	\]
	Here the first map is induced from the dual pairing in \eqref{pairing},  sending a $D_H$-exact form to a continuous functional vanishing on all $\partial_H$-closed currents. 
	$I_{I\Omega}$ is simply defined as the quotient map 
	\[
	I_{I\Omega}:
	\bigl(
	\widehat{I\Omega^{\Spinc}_{\mathrm{dR}}}
	\bigr)^{n}(X, \widehat{\tau})
	\to
	\bigl(
	\widehat{I\Omega^{\Spinc}_{\mathrm{dR}}}
	\bigr)^{n}(X, \widehat{\tau})
	/\im (a_{I\Omega})
	\]
	for now, later in Lemma~\ref{lem:pic} we will identify the quotient with the topological theory. 
	Finally, 
	by the definition of $a_{I\Omega}$, the curvature map $R_{I\Omega}$ descends to the following
	\[
	\ch'_{I\Omega}: 
	\bigl(\widehat{I\Omega^{\Spinc}_{\mathrm{dR}}}\bigr)^{n}(X,\widehat{\tau})\big/\im\bigl(a_{I\Omega}\bigr)
	\xrightarrow{}
	H^n(X; N^{\bullet}_{\Spinc},H). 
	\]
	Functoriality follows directly from the constructions. 
	
	\medskip
	
	To compare the differential model with the topological Anderson dual, we need the following homomorphism
	\[
	p:
	\mathrm{Hom}\bigl(\Omega^{\mathrm{Spin}^c}_{n-1}(X,\tau),\mathbb{R}/\mathbb{Z}\bigr)
	\xrightarrow{}
	\Bigl(\widehat{I\Omega^{\Spinc}_{\mathrm{dR}}}\Bigr)^{n}(X,\widehat{\tau})\big/\im\bigl(a_{I\Omega}\bigr),
	\]
	given by
	\[
	h\mapsto
	I_{I\Omega} (0, (I^{\Spinc})^*(h)).
	\]
	By the exactness of \eqref{spincSES}, one sees $p$ is well-defined. Furthermore, 
	
	\begin{prop}
		We have the following long exact sequence: 
		\begin{align}\label{AndersonLES}
			\begin{aligned}
				\mathrm{Hom}\bigl(\Omega^{\mathrm{Spin}^c}_{n-1}(X,\tau),\mathbb{R}/\mathbb{Z}\bigr)
				&\xrightarrow{p}
				\Bigl(\widehat{I\Omega^{\Spinc}_{\mathrm{dR}}}\Bigr)^{n}(X,\widehat{\tau})\big/\im\bigl(a_{I\Omega}\bigr)
				\\
				&\xrightarrow{\ch'}
				\mathrm{Hom}\bigl(\Omega^{\mathrm{Spin}^c}_{n}(X,\tau),\mathbb{R}\bigr)
				\xrightarrow{q}
				\mathrm{Hom}\bigl(\Omega^{\mathrm{Spin}^c}_{n}(X,\tau),\mathbb{R}/\mathbb{Z}\bigr),
			\end{aligned}
		\end{align}
		where $q$ is induced by $\modZ$.
	\end{prop}
	
	\begin{proof}
		We check the exactness at
		\(\mathrm{Hom}(\Omega_n^{\Spinc}(X,\tau),\mathbb{R})\).
		Let
		\[
		[\omega]\in
		H^n(X;N^\bullet_{\Spinc},H)
		\cong
		\mathrm{Hom}\bigl(
		\Omega_n^{\Spinc}(X,\tau),\mathbb{R}
		\bigr)
		\]
		be represented by
		\(\omega\in \Omega^n_{D_H\mathrm{-clo}}(X;N^\bullet_{\Spinc})\).
		Then \(\omega\) determines a homomorphism
		\[
		\langle -, \omega \rangle
		\in
		\mathrm{Hom}\bigl(
		\Omega_n(X;V^{\Spinc}_{\bullet})/\mathrm{im}\partial_H,
		\mathbb{R}
		\bigr).
		\]
		We first show \(\im(\ch')\subset \ker q\).
		Suppose that \([\omega]\in \im(\ch')\).
		Then there exists
		\[
		h\in
		\mathrm{Hom}\bigl(
		\widehat{\Omega_{n-1}^{\Spinc}}(X,\widehat{\tau}),
		\mathbb{R}/\mathbb{Z}
		\bigr)
		\]
		such that \((\omega,h)\) satisfies the compatibility condition. 
		\[
		(a^{\Spinc})^*(h)
		=
		h\circ a^{\Spinc}
		=
		\modZ\circ \langle -, \omega\rangle.
		\]
		Since \(\mathbb{R}/\mathbb{Z}\) is divisible, applying
		\(\mathrm{Hom}(-,\mathbb{R}/\mathbb{Z})\) to the exact sequence
		\eqref{spincSES} yields a dual exact sequence
		\[
		\mathrm{Hom}\bigl(
		\widehat{\Omega_{n-1}^{\Spinc}}(X,\widehat\tau),\mathbb{R}/\mathbb{Z}
		\bigr)
		\xrightarrow{(a^{\Spinc})^*}
		\mathrm{Hom}\bigl(
		\Omega_n(X;V^{\Spinc}_{\bullet})/\mathrm{im}\partial_H,
		\mathbb{R}/\mathbb{Z}
		\bigr)
		\xrightarrow{(\ch'^{\Spinc})^*}
		\mathrm{Hom}\bigl(
		\Omega_n^{\Spinc}(X,\tau),\mathbb{R}/\mathbb{Z}
		\bigr).
		\]
		Hence the image of \(\modZ\circ \langle -, \omega\rangle\) vanishes in
		$
		\mathrm{Hom}\bigl(
		\Omega_n^{\Spinc}(X,\tau),\mathbb{R}/\mathbb{Z}
		\bigr)
		$. 
		Therefore \(\im(\ch')\subset \ker q\).
		
		Conversely, suppose that \([\omega]\in \ker q\).
		Then the image of
		\(\modZ\circ \langle -, \omega\rangle\) in
		$
		\mathrm{Hom}\bigl(
		\Omega_n^{\Spinc}(X,\tau),\mathbb{R}/\mathbb{Z}
		\bigr)
		$
		is zero. By the same exact sequence above, there exists
		\[
		h\in
		\mathrm{Hom}\bigl(
		\widehat{\Omega_{n-1}^{\Spinc}}(X,\widehat{\tau}),
		\mathbb{R}/\mathbb{Z}
		\bigr)
		\]
		such that the compatibility condition holds
		\[
		h\circ a^{\Spinc}
		=
		\modZ\circ \langle -, \omega\rangle.
		\]
		Hence the pair  \((\omega,h)\) defines a well-defined element
		\[
		I_{I\Omega}(\omega,h)
		\in
		\bigl(
		\widehat{I\Omega^{\Spinc}_{\mathrm{dR}}}
		\bigr)^n(X,\widehat{\tau})
		/\im(a_{I\Omega}).
		\]
		By construction,
		\[
		\ch'\bigl(I_{I\Omega}(\omega,h)\bigr)=[\omega].
		\]
		Therefore \(\ker q\subset \im(\ch')\), proving the exactness. 
		The exactness elsewhere
		is checked in the same way as in \cite{Yam23}. 
	\end{proof}
	
	With the above result, we may identify the quotient $	\bigl(
	\widehat{I\Omega^{\Spinc}_{\mathrm{dR}}}
	\bigr)^{n}(X, \widehat{\tau})
	/\im (a_{I\Omega})$ with the topological Anderson dual ${(I\Omega^{\mathrm{Spin}^c})^n(X,\tau)}$, thus the notation $I_{I\Omega}$ is justified.  
	\begin{prop}\label{lem:pic}
		We have an isomorphism
		\[
		\bigl(
		\widehat{I\Omega^{\Spinc}_{\mathrm{dR}}}
		\bigr)^{n}(X, \widehat{\tau})
		/\im (a_{I\Omega})\cong 
		{(I\Omega^{\mathrm{Spin}^c})^n(X,\tau)} . 
		\]
	\end{prop}

	\begin{proof}
		
		By picturing the two exact sequences 
		(\ref{AndersonSES_twistedSpinc}) and (\ref{AndersonLES}) into the same 
		diagram, 
		% https://q.uiver.app/#q=WzAsMTAsWzAsMCwiMCJdLFswLDEsIjAiXSxbMSwwLCJcXG1hdGhybXtFeHR9KFxcT21lZ2Ffe24tMX1ee1xcbWF0aHJte1NwaW59XmN9KFgsXFx0YXUpLFxcbWF0aGJie1p9KSJdLFsyLDAsIlxcYmlnbChcXHdpZGVoYXR7SVxcT21lZ2Fee1xcbWF0aHJte1NwaW59XmN9X3tcXG1hdGhybXtkUn19fSBcXGJpZ3IpXntufShYLCBcXHdpZGVoYXR7XFx0YXV9KSAvXFxtYXRocm17aW19IChhX3tJXFxPbWVnYX0pIl0sWzEsMSwiXFxtYXRocm17RXh0fShcXE9tZWdhX3tuLTF9XntcXG1hdGhybXtTcGlufV5jfShYLFxcdGF1KSxcXG1hdGhiYntafSkiXSxbMiwxLCIoSVxcT21lZ2Fee1xcbWF0aHJte1NwaW59XmN9KV5uKFgsXFx0YXUpIl0sWzMsMCwiXFxtYXRocm17SG9tfShcXE9tZWdhX3tufV57XFxtYXRocm17U3Bpbn1eY30oWCxcXHRhdSksXFxtYXRoYmJ7Wn0pIl0sWzMsMSwiXFxtYXRocm17SG9tfShcXE9tZWdhX3tufV57XFxtYXRocm17U3Bpbn1eY30oWCxcXHRhdSksXFxtYXRoYmJ7Wn0pIl0sWzQsMCwiMCJdLFs0LDEsIjAiXSxbMiwzXSxbNCw1XSxbMiw0LCJcXG1hdGhybXtpZH0iXSxbMyw1LCJGIl0sWzMsNl0sWzUsN10sWzYsNywiXFxtYXRocm17aWR9Il0sWzYsOF0sWzcsOV0sWzAsMl0sWzEsNF1d
		\[\begin{tikzcd}[column sep=tiny]
			0 & {\mathrm{Ext}(\Omega_{n-1}^{\mathrm{Spin}^c}(X,\tau),\mathbb{Z})} & {\bigl(\widehat{I\Omega^{\mathrm{Spin}^c}_{\mathrm{dR}}} \bigr)^{n}(X, \widehat{\tau}) /\mathrm{im} (a_{I\Omega})} & {\mathrm{Hom}(\Omega_{n}^{\mathrm{Spin}^c}(X,\tau),\mathbb{Z})} & 0 \\
			0 & {\mathrm{Ext}(\Omega_{n-1}^{\mathrm{Spin}^c}(X,\tau),\mathbb{Z})} & {(I\Omega^{\mathrm{Spin}^c})^n(X,\tau)} & {\mathrm{Hom}(\Omega_{n}^{\mathrm{Spin}^c}(X,\tau),\mathbb{Z})} & 0
			\arrow[from=1-1, to=1-2]
			\arrow[from=1-2, to=1-3]
			\arrow["{\mathrm{id}}", from=1-2, to=2-2]
			\arrow[from=1-3, to=1-4]
			\arrow["F", from=1-3, to=2-3]
			\arrow[from=1-4, to=1-5]
			\arrow["{\mathrm{id}}", from=1-4, to=2-4]
			\arrow[from=2-1, to=2-2]
			\arrow[from=2-2, to=2-3]
			\arrow[from=2-3, to=2-4]
			\arrow[from=2-4, to=2-5]
		\end{tikzcd}\]
		we construct a middle vertical $F$ such that the diagram commutes, following a similar Picard groupoid argument in \cite{Yam1, Yam23}. Consider the Picard groupoid 
		\[
		\big(
		\Omega _n(X;V^{\mathrm{Spin}^c}_{\bullet})
		/\im \partial_H
		\xrightarrow{a^{\Spinc}}
		\widehat{
			\Omega^{\mathrm{Spin}^c}_{n-1}}(X, \widehat{\tau})
		\big)
		\]
		whose objects are in
		$\widehat{
			\Omega^{\mathrm{Spin}^c}_{n-1}}(X, \widehat{\tau})$ 
		and morphisms are given by $x_0\xrightarrow{y} x_1$ for each $x_1-x_0=a^{\Spinc}(y)$. 
		
		For
		$
		(\omega, h)\in
		\bigl(
		\widehat{
			I\Omega^{\Spinc}_{\mathrm{dR}}
		}
		\bigr)^{n}(X, \widehat{\tau})
		$, 
		there is an associated functor of Picard groupoids
		\[
		\widetilde F(\omega,h):
		\big(
		\Omega _n(X;V^{\mathrm{Spin}^c}_{\bullet})
		/\im \partial_H
		\xrightarrow{a^{\Spinc}}
		\widehat{
			\Omega^{\mathrm{Spin}^c}_{n-1}}(X, \widehat{\tau})
		\big)
		\to
		\big(
		\mathbb{R}
		\xrightarrow{\modZ}
		\mathbb{R}/\mathbb{Z}
		\big)
		\]
		by applying $h$ on objects and $\omega$ on morphisms. The well-definedness is due to the compatibility condition. 
		Moreover, for two elements $(\omega,h)$ and $(\omega',h')$ differ by an image $a_{I\Omega}(\alpha)$, we have the natural transformation
		\[
		\langle R^{\Spinc}(-),\alpha\rangle: 
		\widetilde F(\omega,h)\Rightarrow\widetilde F(\omega',h'). 
		\]
		By the property of Picard groupoids \cite{HS}, 
		we have a natural equivalence
		\[
		\big(
		\Omega _n(X;V^{\mathrm{Spin}^c}_{\bullet})
		/\im \partial_H
		\xrightarrow{a^{\Spinc}}
		\widehat{
			\Omega^{\mathrm{Spin}^c}_{n-1}
		}(X, \widehat{\tau})
		\big)
		\simeq(\ker(a^{\Spinc})
		\xrightarrow{0}
		\coker(a^{\Spinc})).
		\]
		By the exactness of~\eqref{spincSES}, 
		we have isomorphisms: 
		\[
		\ker(a^{\Spinc})
		\simeq 
		\im(
		\Omega^{\mathrm{Spin}^c}_{n}(X, \tau)
		\xrightarrow{\ch'^{\Spinc}}
		H_{n}(X;V^{\Spinc}_{\bullet}, H)
		),
		\quad
		\coker(a^{\Spinc})
		\simeq
		\Omega^{\mathrm{Spin}^c}_{n-1}(X, \tau)
		.
		\]
		
		Summarizing, we have defined the homomorphism
		\[
		\widetilde F: 
		\bigl(
		\widehat{I\Omega^{\Spinc}_{\mathrm{dR}}}
		\bigr)^{n}(X, \widehat{\tau})
		/\im (a_{I\Omega})
		\to \pi_{0}\FunPic
		\big(
		\big(
		\im(\ch'^{\Spinc})
		\xrightarrow{0}
		\Omega^{\Spinc}_{n-1}(X, \tau)
		\big)
		\to
		\big(
		\mathbb{R}
		\xrightarrow{\modZ}
		\mathbb{R}/\mathbb{Z}
		\big)						
		\big)
		,
		\]
		
		Now we construct a functor of Picard groupoids
		\begin{align}\label{pic_func}
			\pi_{\le1}
			L
			(P_{\tau}(MT\Spinc)/X)
			_{1-n}
			\to
			\big(
			\im(\ch'^{\Spinc})
			\xrightarrow{0}
			\Omega^{\mathrm{Spin}^c}_{n-1}(X, \tau)
			\big).
		\end{align}
		Let $S\mathbb{R}/\mathbb{Z}$ denote the Moore spectrum for $\mathbb{R}/\mathbb{Z}$,
		and set $J$ to be the homotopy cofiber:
		\[
		J:=
		\mathrm{hoCofib}\big(
		\Sigma^{-1}(
		P_{\tau}(MT\Spinc)/X
		\wedge S\mathbb{R}/\mathbb{Z
		})\langle n\rangle)\to 
		P_{\tau}(MT\Spinc)/X
		\big),
		\]
		where $(P_{\tau}(MT\Spinc)/X
		\wedge S\mathbb{R}/\mathbb{Z})\langle n\rangle$ is the $n$-connected cover of 
		$P_{\tau}(MT\Spinc)/X\wedge S\mathbb{R}/\mathbb{Z}$. 
		By Proposition~\ref{prop: dRSpinc}, we have $$
		\pi_{n}(J)
		\simeq 
		H_{n}(X;V^{\Spinc}_{\bullet}, H),
		\quad
		\pi_{n-1}(J)
		\simeq 
		\Omega^{\Spinc}_{n-1}(X, \tau).
		$$
		Since $\pi_{n}(J)$ is a real vector space,  the $k$-invariant for the Picard groupoid $\pi_{\le1}(LJ_{1-n})$ 
		vanishes. Thus we have the equivalence of Picard groupoids
		\[
		\pi_{\le1}(LJ_{1-n})
		\simeq
		\big(
		H_{n}(X;V^{\Spinc}_{\bullet}, H)
		\xrightarrow{0}
		\Omega^{\Spinc}_{n-1}(X, \tau)
		\big),
		\]
		Precomposing with the functor induced by $P_{\tau}(MT\Spinc)/X\to J$, 
		\[
		\pi_{\le1}
		L
		(P_{\tau}(MT\Spinc)/X)
		_{1-n}
		\to
		\pi_{\le1}(LJ_{1-n})
		\simeq
		\big(
		H_{n}(X;V^{\Spinc}_{\bullet}, H)
		\xrightarrow{0}
		\Omega^{\Spinc}_{n-1}(X, \tau)
		\big),
		\]
		we arrive at the desired functor~\eqref{pic_func}. 
		By construction, this functor induces the identity map on $\pi_{0}$, and the twisted homological Chern--Dold character~\eqref{eq: ch'^Spinc} 
		\[
		\ch'^{\Spinc}\colon \Omega^{\Spinc}_n(X,\tau)\longrightarrow H_{n}(X;V^{\Spinc}_{\bullet},H)
		\]
		on $\pi_{1}$. 
		This further induces a homomorphism
		% https://q.uiver.app/#q=WzAsNCxbMCwyLCJcXHBpX3swfVxcbWF0aHJte0Z1blBpY30gXFxiaWcoXFxwaV97XFxsZTF9IFx0TCAoUF97XFx0YXV9KE1cXG1hdGhybXtTcGlufV5jKS9YKV97MS1ufSBcdFxcdG8gXFxiaWcoXFxtYXRoYmJ7Un1cXHhyaWdodGFycm93e1xcbWF0aHJte21vZH1cXG1hdGhiYntafX1cXG1hdGhiYntSfS9cXG1hdGhiYntafVxcYmlnKVxcYmlnKSJdLFswLDEsIlxccGlfezB9XFxtYXRocm17RnVuUGljfSBcXGJpZyhcXGJpZyhcXG1hdGhybXtpbX0oXFxjaCkgXHRcXHhyaWdodGFycm93ezB9IFxcT21lZ2Fee1xcbWF0aHJte1NwaW59XmN9X3tuLTF9KFgsIFxcdGF1KSBcXGJpZykgXHRcXHRvIFxcYmlnKFxcbWF0aGJie1J9XFx4cmlnaHRhcnJvd3tcXG1hdGhybXttb2R9XFxtYXRoYmJ7Wn19XFxtYXRoYmJ7Un0vXFxtYXRoYmJ7Wn1cXGJpZylcXGJpZykiXSxbMCwzLCJcXGJpZyggXHRJXFxPbWVnYV57XFxtYXRocm17U3Bpbn1eY31fe1xcbWF0aHJte0hTfX1cXGJpZ3IpXntufSAoWCwgXFx0YXUpIl0sWzAsMCwiXFxiaWdsKFxcd2lkZWhhdHtJXFxPbWVnYV57XFxtYXRocm17U3Bpbn1eY31fe1xcbWF0aHJte2RSfX19IFxcYmlncilee259KFgsIFxcd2lkZWhhdHtcXHRhdX0pIC9cXG1hdGhybXtpbX0gKGFfe0lcXE9tZWdhfSkiXSxbMSwwXSxbMCwyXSxbMywxXV0=
		\[
		\begin{aligned}
			&F: 	\bigl(\widehat{I\Omega^{\mathrm{Spin}^c}_{\mathrm{dR}}}\bigr)^{n}(X,\widehat{\tau})/\mathrm{im}(a_{I\Omega})
			\\
			&\to
			\pi_{0}\mathrm{FunPic}\Bigl(\bigl(\mathrm{im}(\ch'^{\Spinc})\xrightarrow{0}\Omega^{\mathrm{Spin}^c}_{n-1}(X,\tau)\bigr)\to\bigl(\mathbb{R}\xrightarrow{\mathrm{mod}\mathbb{Z}}\mathbb{R}/\mathbb{Z}\bigr)\Bigr)
			\\
			&\to
			\pi_{0}\mathrm{FunPic}\Bigl(\pi_{\le 1}L
			\bigl(
			P_{\tau}(MT\mathrm{Spin}^c)/X\bigr)_{1-n}\to\bigl(\mathbb{R}\xrightarrow{\mathrm{mod}\mathbb{Z}}\mathbb{R}/\mathbb{Z}\bigr)\Bigr)
			\\
			&\to
			\bigl(I\Omega^{\mathrm{Spin}^c}
			\bigr)^{n}(X,\tau).
		\end{aligned}
		\]
		The compatibility of $F$ and the desired  diagram is clear from construction. 
		Hence by the five lemma, we have 
		\[
		F: 
		\bigl(
		\widehat{I\Omega^{\Spinc}_{\mathrm{dR}}}
		\bigr)^{n}(X, \widehat{\tau})
		/\im (a_{I\Omega})
		\xrightarrow{\cong} 
		{(I\Omega^{\mathrm{Spin}^c})^n(X,\tau)}
		\]
		is an isomorphism. 
		This concludes the proof. 
	\end{proof}
	
	The last result we need for the proof of Theorem~\ref{thm:diffio} is a dual statement of Proposition~\ref{prop: dRSpinc}.
	
	\begin{prop}\label{lem: dRIO}
		Tensoring with $\mathbb{R}$,
		\[
		\ch'_{I\Omega}\otimes\mathbb{R}\colon (I\Omega^{\mathrm{Spin}^c})^n(X,\tau)\otimes\mathbb{R}\rightarrow H^n(X;N_{\Spinc}^\bullet,H)
		\]
		is a natural isomorphism.
	\end{prop}
	
	\begin{proof}
		From \eqref{AndersonSES_twistedSpinc} we have a short exact sequence
		\[
		0\rightarrow \mathrm{Ext}\bigl(\Omega^{\mathrm{Spin}^c}_{n-1}(X,\tau),\mathbb{Z}\bigr)\rightarrow (I\Omega^{\mathrm{Spin}^c})^n(X,\tau)\rightarrow \mathrm{Hom}\bigl(\Omega^{\mathrm{Spin}^c}_n(X,\tau),\mathbb{Z}\bigr)\rightarrow 0.
		\]
		The twisted Atiyah-Hirzebruch spectral sequence shows $\Omega^{\mathrm{Spin}^c}_k(X,\tau)$ is finitely generated for all $k$. Therefore $\mathrm{Ext}(\Omega^{\mathrm{Spin}^c}_{n-1}(X,\tau),\mathbb{Z})$ is torsion, so by tensoring with $\mathbb{R}$, one has
		\[
		(I\Omega^{\mathrm{Spin}^c})^n(X,\tau)\otimes\mathbb{R}\cong \mathrm{Hom}\bigl(\Omega^{\mathrm{Spin}^c}_n(X,\tau)\otimes\mathbb{R},\mathbb{R}\bigr).
		\]
		By Proposition~\ref{prop: dRSpinc} we have a natural isomorphism
		\[
		\ch'^{\Spinc}\otimes\mathbb{R}\colon \Omega^{\Spinc}_n(X,\tau)\otimes\mathbb{R}\xrightarrow{\cong} H_n\bigl(X;V^{\mathrm{Spin}^c}_\bullet,H\bigr).
		\]
		Thus
		\[
		(I\Omega^{\mathrm{Spin}^c})^n(X,\tau)\otimes\mathbb{R}\cong \mathrm{Hom}\bigl(H_n(X;V^{\mathrm{Spin}^c}_\bullet,H),\mathbb{R}\bigr).
		\]
		Recall in Section~\ref{sec: chaincpx} we have shown that the complexes
		\[
		(\Omega_\ast(X;V^{\mathrm{Spin}^c}_\bullet),\partial_H)\quad
		\text{and} \quad (\Omega^\ast(X;N_{\Spinc}^\bullet),D_H)
		\]
		are continuous duals, where $\partial_H$ and $D_H$ are adjoints under the evaluation pairing. Since in each total degree the coefficients are finite dimensional, the evaluation pairing is perfect and yields
		\[
		H^n
		\bigl(
		X;N_{\Spinc}^\bullet,H
		\bigr)\cong \mathrm{Hom}\bigl(H_n(X;
		V^{\Spinc}_\bullet,H),\mathbb{R}\bigr),
		\]
		which concludes the proof.
	\end{proof}
	
	Now we summarize the proof for the main theorem of this section.
	
	\begin{proof}[Proof of Theorem~\ref{thm:diffio}]
		We only need to verify the properties (i)-(iii) in Definition~\ref{defi:diffextco}.
		The isomorphism (i) is given by Proposition~\ref{lem: dRIO}. 
		The commutativity condition (ii) follows from Proposition~\ref{lem:pic}, and exactness condition (iii) follows directly from construction. 
		Hence \eqref{diffIOdata2} is a differential extension of the Anderson dual to twisted $\Spinc$-bordism, in the sense of Definition~\ref{defi:diffextco}.
	\end{proof}
	
	\medskip

	\subsection{Differential twisted $\Spinc$-cobordism}\label{sec: cob}
	
	For differential \textit{co}bordism, 
	Bunke--Schick--Schr\"oder--Wiethaup~\cite{bunke2009landweber} construct concrete differential cocycle models for $MU$ and Landweber-exact $MU_*$-modules, and their constructions are generalized to cobordism theories of any (tangential) structure group by Yamashita--Yonekura \cite{Yam1}.
	Following their constructions, we describe a \textit{twisted} cocycle model for differential twisted $\Spinc$-cobordism theory in this subsection.  
	
	We begin by recalling the notion of
	stable relative tangent bundles~\cite{Yam1}. 
	Let $p:N\to X$ be a map of relative dimension $r=\dim N-\dim X$. 
	Choose a bundle map
	\[
	\phi:\underline{\mathbb{R}}^k\to p^*TX
	\]
	such that
	\[
	\phi\oplus dp:\underline{\mathbb{R}}^k\oplus TN\to p^*TX
	\]
	is surjective.  We define \emph{the stable relative tangent bundle for $p$ associated to $\phi$} to be the stable vector bundle over $N$ represented by the following rank $(k+r)$ bundle: 
	\[
	T(\phi, p):=
	\ker(
	\phi\oplus dp:
	\underline{\mathbb{R}}^k\oplus
	TN\to p^*TX
	). 
	\]
	We now formulate the definition for a geometric twisted stable relative tangential $\Spinc$-cochain.
	
	\begin{defi}\label{def:relSpincCochain}
		A geometric $\widehat{\tau}$-twisted \emph{stable relative tangential} $\Spinc$-cochain over $X$ of relative dimension $r$ is a quadruple
		\[
		(N,p,f^{\nabla}_{T(\phi,p)},\widehat{\eta}),
		\]
		where
		\begin{itemize}
			\item $N$ is an oriented Riemannian $(\dim X+r)$-manifold with boundary, equipped with a collar embedding of $\partial N$, along which all data are assumed to be constant;
			\item $p\colon N\to X$ is a  proper map of relative dimension $r$;
			\item $f^{\nabla}_{T(\phi,p)}\colon N\to B_{\nabla}\SO$ is the classifying map for the {stable relative tangent bundle} with its chosen connection; 
			\item $\widehat{\eta}\in\difStruc{T(\phi,p)}$ is a differential $\widehat{\tau}$-twisted $\Spinc$-structure on $T(\phi,p)$.
		\end{itemize}
	\end{defi}
	\noindent
	An isomorphism $
	(N,p,f^{\nabla}_{T(\phi,p)},\widehat{\eta})\to (N',p',f^{\nabla}_{T(\phi',p')},\widehat{\eta}')
	$
	is an orientation and collar-preserving isometry $h\colon N\to N'$,
	such that $p= p'\circ h$, and the differential twisted
	$\Spinc$-structures are identified by a 2-morphism
	$W_3^\nabla
	\circ
	(dh)
	\circ \widehat\eta
	\Longrightarrow
	h^*\widehat\eta'
	$
	in $B^2_{\nabla}\uU(1)(N)$ relative to boundary collars. 
	In particular, we have $\kappa(\widehat{\eta})=h^{*}\kappa(\widehat{\eta}')$. 
	The collection of
	geometric $\widehat{\tau}$-twisted
	stable relative tangential
	$\Spinc$-cochains over $X$
	of relative dimension $r$ forms an abelian group under disjoint union, which we denote by $\widetilde{C}^{-r}_{\Spinc}
	(X, \widehat{\tau})$. 
	
	Consider a differential twisted stable relative tangential $\Spinc$-cochain over $X\times \mathbb{R}$
	\[
	\big(
	W, q, f^\nabla_{T(\phi, q)},\widehat\eta
	\big), 
	\]		
	such that $q$ is proper on the restriction to $X\times (-\infty,0]$ and transverse to $0$.
	Denote $W_0:={q^{-1}(
		X   \times \{0\}
		)}$, there is an induced differential
	$\widehat{\tau}$-twisted stable relative tangential $\Spinc$-cochain
	\[
	\big(
	W_0,  {q}|_{W_0},f^\nabla_{T(\phi, q)}|_{W_0}, \widehat\eta|_{W_0}
	\big).
	\]
	Such a cochain is called a \emph{geometric bordism datum}.

	As in the bordism case, 
	we need to introduce currents for the twisted Chern--Weil construction.
	Recall the de Rham $i$-currents are defined as continuous functionals on compactly supported $(\dim X-i)$-forms:
	\[
	\Omega^{i}_{-\infty}(X)
	:=
	\Hom_{\cts} \big(
	\Omega^{\dim X-i}_c(X),\mathbb{R}
	\big). 
	\]
	The current differential $b: \Omega^{i}_{-\infty}(X) \to \Omega^{i+1}_{-\infty}(X)$ is characterized as follows
	\[
	\big\langle bT, \omega \big\rangle
	=
	(-1)^{|T|+1}
	\big\langle T, d\omega \big\rangle,
	\quad
	\omega \in 
	\Omega^{\dim X-i-1}_{c}(X).
	\]
	The natural product
	\begin{align}\label{current_wedge}
		\wedge: 
		\Omega^{j}(X)
		\otimes
		\Omega^{i}_{-\infty}(X)
		\rightarrow
		\Omega^{i+j}_{-\infty}(X)
	\end{align}
	satisfies the Leibniz rule
	\begin{align}\label{leibniz_b}
		b(\alpha \wedge T) = 
		d\alpha \wedge T
		+
		(-1)^{|\alpha|} 
		\alpha \wedge (bT). 
	\end{align}
	
	\medskip
	
	When studying \textit{bordism} theory in Section~\ref{sec: chaincpx}, we used the chain complex of \emph{compactly supported currents} 
	$
	\big(
	\Omega_{i}(X), \partial
	\big)$
	whose homology computes the real homology of $X$. 
	Now in the case of \textit{cobordism}, we need the cochain complex of \emph{currents} 
	$
	\big(
	\Omega_{-\infty}^i(X), b
	\big)
	$ whose cohomology computes the real cohomology of $X$. We write	
	\[
	\Omega^{-r}_{-\infty}(X;V^{\bullet}_{\Spinc})
	:=\bigoplus_{i+j=-r}\Omega^{i}_{-\infty}(X)\otimes V^{\Spinc}_{-j}
	\cong
	\Hom_{\cts} \big(
	\Omega^{\dim X+r}_c(X;N_{\Spinc}^{\bullet}),\mathbb{R}\big), 
	\]
	and deform the current differential by
	\[
	\delta_H: = b - H\wedge(u \times -)
	:\Omega^{-r}_{-\infty}(X;V^{\bullet}_{\Spinc})
	\rightarrow
	\Omega^{-r+1}_{-\infty}(X;V^{\bullet}_{\Spinc}).
	\]
	Similarly, we deform the differential for the complex of forms with coefficients 
	\[
	d_H:=d-H\wedge (u\times -)\colon
	\Omega^{-\ast}(X;V^\bullet_{\Spinc})
	\rightarrow
	\Omega^{-\ast+1}(X;V^\bullet_{\Spinc}).
	\]
	There is a natural inclusion
	\[
	\iota \colon \Omega^{i}(X;V^\bullet_{\Spinc})
	\hookrightarrow \Omega^{i}_{-\infty}(X;V^\bullet_{\Spinc}),
	\]
	sending an $i$-form $\tau$ to a functional 
	$\int_X (\tau\wedge-)$ acting on compactly supported $(\dim X-i)$-forms. One may check that $\iota$ is a chain map between the twisted complexes. Moreover, as shown in the following lemma, $\iota$
	is a quasi-isomorphism. 
	
	\begin{lem}\label{lem:quasi_iso}
		$\iota$ induces an isomorphism on cohomology groups
		\[
		H^{-\ast}\bigl(\Omega^{-\ast}(X;V^\bullet_{\Spinc}),d_H\bigr)
		\stackrel{\cong}{\longrightarrow}
		H^{-\ast}\bigl(\Omega^{-\ast}_{-\infty}(X;V^\bullet_{\Spinc}),\delta_H\bigr).
		\]
	\end{lem}
	
	\begin{proof}
		Filtering by the coefficient degree in
		$V^\bullet_{\Spinc}$, 
		we equip both complexes with the increasing filtrations
		\[
		F^p\Omega^{n}(X;V^\bullet_{\Spinc})
		:=
		\bigoplus_{\substack{i+j=n\\ j\leq p}}
		\Omega^i(X)\otimes V^{j}_{\Spinc},
		\]
		and
		\[
		F^p\Omega^{n}_{-\infty}(X;V^\bullet_{\Spinc})
		:=
		\bigoplus_{\substack{i+j=n\\ j\leq p}}
		\Omega^i_{-\infty}(X)\otimes V^{j}_{\Spinc},
		\]
		which are clearly preserved by $d_H$ and $\delta_H$. 		
		Let $E_r^{p,q}$ and $ \widetilde E_r^{p,q}$ be the spectral sequences associated to the above
		filtrations on $\Omega^{-\ast}(X;V^\bullet_{\Spinc})$ and
		$\Omega^{-\ast}_{-\infty}(X;V^\bullet_{\Spinc})$, respectively. 
		By construction, the $E_0$-pages are
		\[
		E_0^{p,q}\cong \Omega^q(X)\otimes V^p_{\Spinc},
		\qquad
		\widetilde E_0^{p,q}\cong \Omega^q_{-\infty}(X)\otimes V^p_{\Spinc},
		\]
		where the zeroth differentials are induced by $d$ and $b$, respectively. 
		It follows that
		\[
		E_1^{p,q}\cong H^q\bigl(\Omega^\ast(X),d\bigr)\otimes V^p_{\Spinc},
		\qquad
		\widetilde E_1^{p,q}\cong H^q\bigl(\Omega^\ast_{-\infty}(X),b\bigr)\otimes V^p_{\Spinc}.
		\]
		Since the natural inclusion $\iota$ also preserves the filtrations, it induces a morphism of spectral
		sequences
		\[
		\iota_r \colon E_r^{p,q}\longrightarrow \widetilde E_r^{p,q}.
		\]
		On the $E_0$-page, this is given by
		$
		\Omega^\ast(X)\hookrightarrow \Omega^\ast_{-\infty}(X)
		$
		tensored with $V^p_{\Spinc}$, which is a
		quasi-isomorphism by de Rham regularization \cite{dR}. Therefore
		\[
		\iota_1 \colon E_1^{p,q}\stackrel{\cong}{\longrightarrow}\widetilde E_1^{p,q}
		\]
		is an isomorphism for all $p,q$.
		Since the dimension of $X$ is finite, both spectral sequences converge strongly to the cohomology of
		\[
		\bigl(\Omega^{-\ast}(X;V^\bullet_{\Spinc}),d_H\bigr)
		\qquad\text{and}\qquad
		\bigl(\Omega^{-\ast}_{-\infty}(X;V^\bullet_{\Spinc}),\delta_H\bigr),
		\]
		respectively. The lemma then follows from the comparison theorem.
	\end{proof}

	We now describe the twisted Chern--Weil construction for twisted $\Spinc$-cobordism.
	For a
	differential $\widehat{\tau}$-twisted stable relative tangential $\Spinc$-cochain
	$(N, p, f^\nabla_{T(\phi,p)}, \widehat\eta)$ of relative dimension $r$, 
	we associate a current of total degree $-r$, defined on the test form
	$\omega \otimes p_I  \zeta^k\in
	\Omega_c^{\dim X+r}(X;N_{\mathrm{Spin}^c}^{\bullet})$ by
	\[
	\cw
	(N, p, f^\nabla_{T(\phi,p)}, \widehat\eta)
	\colon
	\omega \otimes p_I  \zeta^k
	\mapsto
	\int_N
	p^*\omega \wedge 
	p_I\bigl(\nabla^{T(\phi,p)}\bigr) \wedge
	\kappa(\widehat\eta)^k. 
	\]
	Since $p$ is proper, the integrand is a compactly supported $(\dim X+r)$-form on $N$, thus the integral is well-defined. 
	Similar to the homological case, we may check that
	\begin{align} 
		\cw: \widetilde{C}_{\Spinc}^{-r}(X, \widehat{\tau})
		\xrightarrow{}
		\Omega_{-\infty}^{-r}(X;V_{\mathrm{Spin}^c}^{\bullet})
	\end{align}
	is a chain map, invariant for isomorphic chains and descends to 
	\[
	\ch'_{\Spinc}\colon
	{\Omega_{\mathrm{Spin}^c}^{-r}}(X,{\tau})
	\to
	H^{-r}(X;V^{\mathrm{Spin}^c}_{\bullet}, H).
	\]
	Moreover, we record the pre-Chern--Weil form $\cw_{\mathrm{pre}}(c)
	\in \Omega^{-r}\bigl(N;V^\bullet_{\Spinc}\bigr)$ which is characterized by
	\begin{align}\label{cwpre}
		\big\langle p_I\zeta^k,\cw_{\mathrm{pre}}(c)\big\rangle
		:=
		p_I\bigl(\nabla^{T(\phi,p)}\bigr)
		\wedge
		\kappa(\widehat\eta)^k, \quad 
		\forall
		p_I\zeta^k\in N^\bullet_{\Spinc}.
	\end{align}
	Note that the homological Chern--Weil current $\cw(c)$ on $X$
	is the pushforward of this form. 
	In the case of trivial twist, our construction recovers the Chern--Weil construction in \cite{bunke2009landweber} and \cite{Yam1}. 
	
	\begin{defi}\label{def:diffSpincCobordismCochain}
		Let $\widehat{\tau}\colon X\to B^2_{\mathrm{conn}}\uU(1)$ be a differential twist with curvature $H$.
		Define the differential $\widehat{\tau}$-twisted $\Spinc$-\emph{cobordism} group
		\[
		\widehat{\Omega_{\Spinc}^{-r}}(X,\widehat{\tau})
		:=\big\{
		(N,p,f^{\nabla}_{T(\phi,p)},\widehat{\eta},\alpha)
		\big\}\big/\sim,
		\]
		where $(N,p,f^{\nabla}_{T(\phi,p)},\widehat{\eta})$ is a geometric $\widehat{\tau}$-twisted stable relative tangential $\Spinc$-cochain over $X$ of relative dimension $r$ without boundary, and
		$
		\alpha\in \Omega^{-r-1}_{-\infty}(X;V^{\bullet}_{\Spinc})/\im\delta_H$ satisfies
		\[
		\cw(N,p,f^{\nabla}_{T(\phi,p)},\widehat{\eta})-\delta_H\alpha\in \Omega^{-r}(X;V^{\bullet}_{\Spinc})
		\subset \Omega^{-r}_{-\infty}(X;V^{\bullet}_{\Spinc}).
		\]
		The existence of such $\alpha$ is due to Lemma~\ref{lem:quasi_iso}. 
		The relation $\sim$ is generated by:
		\begin{itemize}
			\item \emph{Isomorphisms:}
			\[
			(N,p,f^{\nabla}_{T(\phi,p)},\widehat{\eta},\alpha)\sim
			(N',p',f^{\nabla}_{T(\phi',p')},\widehat{\eta}',\alpha)
			\]
			for isomorphic geometric cochains $(N,p,f^{\nabla}_{T(\phi,p)},\widehat{\eta})$ and $(N',p',f^{\nabla}_{T(\phi',p')},\widehat{\eta}')$.
			\item \emph{Additivity:} disjoint union on geometric cochains and addition on $\alpha$.
			\item \emph{Bordism:}
			\[
			\big(
			W_0,  {q}|_{W_0},f^\nabla_{T(\phi, q)}|_{W_0}, \widehat\eta|_{W_0},
			\int_{(-\infty, 0]}
			\cw(y)
			\big)
			\sim 0
			\]
			for any geometric bordism datum $y=(W,q,f^{\nabla}_{T(\phi,q)},\widehat{\eta})$.
		\end{itemize}
	\end{defi}
	\noindent
	The structure maps $R_{\Spinc}$, $I_{\Spinc}$ and $a_{\Spinc}$ can be defined analogously as in the bordism case. Set 
	\[
	\mathcal{M}^*_{\Spinc}:=
	\bigl(
	\Omega^{*}
	(X;V^{\bullet}_{\Spin^c}),
	d-H\wedge(u\times-)
	\bigr). 
	\]
	One may similarly verify that 	\[
	\bigl(
	\widehat{\Omega_{\Spinc}^{*}}
	(-,-),
	\mathcal{M}^*_{\Spinc}(-,-),  \ch'_{\Spinc}, 
	R_{\mathrm{Spin}^c}, I_{\mathrm{Spin}^c}, a_{\mathrm{Spin}^c}
	\bigr)
	\] is a differential extension to twisted $\Spinc$-cobordism theory in the sense of Definition~\ref{defi:diffextco}.

	\medskip

	\subsection{Differential multiplication and pushforward}\label{sec:mul}
	In this subsection, we give the twisted versions of differential multiplication and pushforward constructed in \cite{Yam1}. 
	For a manifold $X$, let $\widehat{\tau}_1$, $\widehat{\tau}_2$, and 
	$\widehat{\tau}_3=\widehat{\tau}_1+\widehat{\tau}_2$ be differential twists over $X$, with curvatures $H_1$, $H_2$, and $H_3=H_1+H_2$, respectively. 
	For the sake of simplicity, we assume $X$ is oriented in this section, however, our arguments remain valid in the general case when taking the orientation bundle of $X$ into account. 
	We establish the following operations
	\begin{itemize}
		\item \textbf{Differential multiplication.}
		\begin{align}\label{diffmul}
			(\widehat{I\Omega^{\Spinc}_\dR})^{n}
			(X, \widehat{\tau}_3)
			\otimes
			\widehat{\Omega_{\Spinc}^{-r}}
			(X, \widehat{\tau}_2)
			\to
			(\widehat{I\Omega^{\Spinc}_\dR})^{n-r}
			(X, \widehat{\tau}_1). 
		\end{align}
		\item \textbf{Differential pushforward.}\\
		When $p:N\to X$ is a proper submersion for a geometric $\widehat\tau_2$-twisted stable relative tangential
		$\Spinc$-cochain 
		$
		c=(N,p,f^\nabla_{T(\phi,p)},\widehat\eta)
		$
		of relative dimension $r$ over $X$, 
		\begin{align}\label{diffpush}
			{c}_*: 
			( \widehat{I\Omega^{\mathrm{Spin}^c}_{\mathrm{dR}}} )
			^{n}
			\big(
			N, p^*\widehat{\tau}_3\big) \longrightarrow
			( \widehat{I\Omega^{\mathrm{Spin}^c}_{\mathrm{dR}}} )^{n-r}(X,\widehat{\tau}_1). 
		\end{align}
		In particular, taking $N=X\times S^1$ gives an $S^1$-integration map
		\begin{align}
			\int: 
			( \widehat{I\Omega^{\mathrm{Spin}^c}_{\mathrm{dR}}} )
			^{n+1}
			\big(
			X\times S^1, \widehat{\tau}_1\big) \longrightarrow
			( \widehat{I\Omega^{\mathrm{Spin}^c}_{\mathrm{dR}}} )^{n}(X,\widehat{\tau}_1). 
		\end{align}
	\end{itemize}
	
	\medskip
	
	To begin with, we define a fiber product map for geometric chains. Let
	\begin{align*}
		c=(N,p,f^{\nabla}_{T(\phi,p)},\widehat\eta)
		\in \widetilde C^{-r}_{\Spinc}(X,\widehat\tau_2)
	\end{align*}
	be a geometric $\widehat\tau_2$-twisted stable relative tangential
	$\Spinc$-cochain of relative dimension $r$.  
	For a geometric $\widehat\tau_1$-twisted $\Spinc$-chain
	$z=(M,f,f^{\nabla}_{TM},\widehat\eta_M)$ of dimension $n-r-1$ such that
	$f$ is transverse to $p$, the fiber product
	\[
	\begin{tikzcd}
		{M\times_X N} \arrow[r] \arrow[d] & N \arrow[d,"p"] \\
		M \arrow[r,"f"] & X
	\end{tikzcd}
	\]
	is an $(n-1)$-dimensional chain over $N$. 
	Choosing a splitting
	$\underline{\mathbb R}^{k}\oplus TN=H_p\oplus T(\phi,p)$ and a Riemannian
	metric on $X$ gives the stable isomorphism
	\begin{align}\label{stable-iso-corrected-v2}
		\underline{\mathbb R}^{k}\oplus T(M\times_XN)
		\cong
		TM\oplus T(\phi,p).
	\end{align}
	The differential $\widehat\tau_1$-twisted $\Spinc$-structure on $TM$ and
	the differential $\widehat\tau_2$-twisted $\Spinc$-structure on $T(\phi,p)$
	therefore induce a differential $p^*\widehat\tau_3$-twisted
	$\Spinc$-structure on $T(M\times_XN)$ by \eqref{directsum}.  Hence we have a fiber product map
	\begin{align}\label{fiber}
		\times_X c:
		\widetilde C^{\Spinc}_{n-r-1}(X,\widehat\tau_1)_{p}
		\longrightarrow
		\widetilde C^{\Spinc}_{n-1}(N,p^*\widehat\tau_3),
	\end{align}
	where the subscript denotes the transversality condition with $p$.
	After composing the map $M\times_XN\to N$ with $p:N\to X$, we obtain a
	geometric $\widehat\tau_3$-twisted $\Spinc$-chain over $X$.  
	The above construction gives
	\begin{align}\label{eq:Pi-c-geometric}
		\Pi_c:=p_*\circ (\times_X c): 
		\widetilde C^{\Spinc}_{n-r-1}(X,\widehat\tau_1)_{p}
		\longrightarrow
		\widetilde C^{\Spinc}_{n-1}(X,\widehat\tau_3).
	\end{align}
	As in \cite{Yam1}, different choices of the splitting, the metric, and
	transverse representatives give canonically bordant chains. 
	
	\medskip
	
	Next we define a mixed product
	\[
	\wedge_\star:
	\Omega^n(X;N^\bullet_{\Spinc})\otimes
	\Omega_{-\infty}^{-r}(X;V^\bullet_{\Spinc})
	\longrightarrow
	\Omega_{-\infty}^{n-r}(X;N^\bullet_{\Spinc}),
	\]
	characterized by
	\begin{align}\label{eq:star-product-even-coeff}
		(\eta\otimes\varphi)\wedge_\star(\rho\otimes v)
		:=
		(\eta\wedge\rho)\otimes(\varphi\star v),
		\quad
		\eta\otimes\varphi\in \Omega^a(X)\otimes N^p_{\Spinc}, \ 
		\rho\otimes v\in \Omega_{-\infty}^b(X)\otimes V^q_{\Spinc}.
	\end{align}	
	The differentials are given respectively by
	\[
	D_{H_3}=d+H_3\wedge\partial_\zeta
	\]
	on $N^\bullet_{\Spinc}$-valued forms,  
	\[
	\delta_{H_2}=b-H_2\wedge(u\times-)
	\]
	on
	$V^\bullet_{\Spinc}$-valued cohomological currents, and by abusing notation,
	\[
	D_{H_1}=b+H_1\wedge\partial_\zeta
	\]
	on $N^\bullet_{\Spinc}$-valued cohomological currents. The choices of differentials are compatible in the following sense: 
	
	\begin{lem}\label{lem:leibniz}
		One has the twisted Leibniz rule
		\begin{align}\label{eq:star-leibniz-even-coeff}
			D_{H_1}(\omega\wedge_\star\theta)
			=
			(D_{H_3}\omega)\wedge_\star\theta
			+
			(-1)^n\omega\wedge_\star \delta_{H_2}\theta ,
		\end{align}
		for $	\omega\in\Omega^n(X;N^\bullet_{\Spinc})$ and
		$\theta\in\Omega_{-\infty}^m(X;V^\bullet_{\Spinc})$.
	\end{lem}
	
	\begin{proof}
		It is enough to prove the identities on pure tensors. Write
		\[
		\omega=\eta\otimes\varphi,
		\qquad
		\theta=\rho\otimes v,
		\]
		and note that $
		|\omega|\equiv |\eta|, 
		|\theta| \equiv |\rho|
		\mod 2$ since the coeffients are even. For the non-twisted differential $b$, we have
		\[
		\begin{aligned}
			b(\omega\wedge_\star\theta)
			=
			(d\omega)\wedge_\star\theta
			+
			(-1)^n\omega\wedge_\star b\theta .
		\end{aligned}
		\]
		It remains to compare the twisting terms on both sides. 
		On the left hand side, the $H_1$-twisting term is
		\begin{align}\label{eq:H-left-even}
			H_1\wedge\partial_\zeta(\omega\wedge_\star\theta)
			=
			(H_1\wedge\eta\wedge\rho)\otimes
			\bigl((\partial_\zeta\varphi)\star v\bigr).
		\end{align}
		On the right hand side, the $H_3$-twisting term from
		$(D_{H_3}\omega)\wedge_\star\theta$ is
		\[
		(H_3\wedge\eta\wedge\rho)\otimes
		\bigl((\partial_\zeta\varphi)\star v\bigr).
		\]
		The $H_2$-twisting term from
		$(-1)^n\omega\wedge_\star \delta_{H_2}\theta$ is
		\[
		\begin{aligned}
			(-1)^n\omega\wedge_\star\bigl(-H_2\wedge(u\times\theta)\bigr)
			&=
			-(-1)^n
			(\eta\wedge H_2\wedge\rho)\otimes
			\bigl(\varphi\star(u\times v)\bigr) \\
			&=
			-(-1)^n(-1)^{3n}
			(H_2\wedge\eta\wedge\rho)\otimes
			\bigl((\partial_\zeta\varphi)\star v\bigr) \\
			&=
			-(H_2\wedge\eta\wedge\rho)\otimes
			\bigl((\partial_\zeta\varphi)\star v\bigr),
		\end{aligned}
		\]
		which follows by Proposition~\ref{prop:adjointness}. 
		Since $H_3=H_1+H_2$, the total twisting term on the right hand side is
		\[
		(H_1\wedge\eta\wedge\rho)\otimes
		\bigl((\partial_\zeta\varphi)\star v\bigr),
		\]
		which agrees with~\eqref{eq:H-left-even}. This proves
		\eqref{eq:star-leibniz-even-coeff}.
	\end{proof}
	
	\medskip

	We now define the differential multiplication in the twisted setting, following
	the strategy of \cite{Yam1}. Fix a differential $\widehat{\tau_2}$-twisted $\Spinc$-cocycle $(c,\alpha)$ of relative dimension $r$, and let
	\[
	(\omega,h)
	\in
	\bigl(\widehat{I\Omega^{\Spinc}_{\dR}}\bigr)^n(X,\widehat\tau_3),
	\]
	where the $D_{H_3}$-closed form $\omega$ and
	$
	h:\widehat{\Omega^{\Spinc}_{n-1}}(X,\widehat\tau_3)
	\rightarrow \mathbb R/\mathbb Z
	$
	satisfies the compatibility condition. We set
	\begin{align}\label{eq:star-product-definition}
		\omega_{(c,\alpha)}
		&:=
		\omega\wedge_\star R_{\Spinc}(c,\alpha),
		\\
		h_{(c,\alpha)}(z,\phi)
		&:=
		h\bigl(\Pi_c(z),0\bigr)
		-\bigl\langle \phi,
		\omega\wedge_\star R_{\Spinc}(c,\alpha)\bigr\rangle
		-(-1)^{|\omega|}
		\bigl\langle
		\cw(z),
		\omega\wedge_\star\alpha
		\bigr\rangle
		\quad \mod \mathbb Z .
	\end{align}
	where $(z,\phi)$ is a test differential
	$\widehat\tau_1$-twisted $\Spinc$-cycle of dimension $n-r-1$, and the term $\big\langle
	\cw(z),
	\omega\wedge_\star\alpha
	\big\rangle$ is defined by using 
	the currential generalization of Chern--Weil evaluation, as in \cite[Sec. 5.2]{Yam1}.

	\begin{prop}\label{prop:star-product-well-defined}
		The element $
		\bigl(
		\omega_{(c,\alpha)},
		h_{(c,\alpha)}
		\bigr)
		$ defines an element of
		$
		\bigl(\widehat{I\Omega^{\Spinc}_{\dR}}\bigr)^{n-r}
		(X,\widehat\tau_1)
		$, which depends only on the differential cobordism class of $(c,\alpha)$ and the
		Anderson class of $(\omega,h)$. Hence this construction gives the desired  map
		\[
		(\widehat{I\Omega^{\Spinc}_\dR})^{n}
		(X, \widehat{\tau}_3)
		\otimes
		\widehat{\Omega_{\Spinc}^{-r}}
		(X, \widehat{\tau}_2)
		\to
		(\widehat{I\Omega^{\Spinc}_\dR})^{n-r}
		(X, \widehat{\tau}_1). 
		\]
	\end{prop}
	
	\begin{proof}
		The form component is $D_{H_1}$-closed by
		Lemma~\ref{lem:leibniz}.  
		We now check that $h_{(c,\alpha)}$ is a well-defined homomorphism on the
		differential twisted bordism group. Additivity follows from the additivity of
		$\Pi_c$, $\cw$, and $h$. Invariance under isomorphisms follows from the
		isomorphism invariance of $\cw$ and the fact that in~\eqref{eq:Pi-c-geometric} different transverse
		representatives give bordant
		geometric chains. It remains to check invariance under the bordism relation
		$
		(\partial y,0)\sim (\emptyset,-\cw(y))
		$
		for a geometric $\widehat\tau_1$-twisted $\Spinc$-chain $y$.
		Indeed,
		\begin{equation}\label{eq:h_bord_check}
			\begin{aligned}
				h_{(c,\alpha)}(\partial y,0)
				&=
				h(\Pi_c(\partial y),0)
				-(-1)^n
				\bigl\langle
				\cw(\partial y),
				\omega\wedge_\star\alpha
				\bigr\rangle
				\quad \mod \mathbb Z
				\\
				&=
				\bigl\langle\cw(\Pi_c(y)),\omega\bigr\rangle
				-
				\bigl\langle
				\cw(y),
				\omega\wedge_\star\delta_{H_2}\alpha
				\bigr\rangle
				\quad \mod \mathbb Z .
			\end{aligned}
		\end{equation}
		Here the second equality uses the compatibility condition for $(\omega,h)$, $\Pi_c(\partial y)=\partial\Pi_c(y)$, and the identity
		\begin{align}\label{eq:currential-stokes-star}
			\bigl\langle
			\cw(\partial y),
			\omega\wedge_\star\alpha
			\bigr\rangle
			\equiv
			(-1)^n
			\bigl\langle
			\cw(y),
			\omega\wedge_\star\delta_{H_2}\alpha
			\bigr\rangle
			\quad \mod \mathbb Z ,
		\end{align}
		which follows from $\cw(\partial y)=\partial_{H_1}\cw(y)$ and Lemma~\ref{lem:leibniz}.

		For the first term in~\eqref{eq:h_bord_check}, we have 
		\begin{align}\label{eq:currential-fiber-star}
			\bigl\langle\cw(\Pi_c(y)),\omega\bigr\rangle 
			=
			\bigr\langle 
			\cw(y),
			\omega\wedge_\star\cw(c)
			\bigr\rangle . 
		\end{align}		
		This follows from~\cite[(5.44)]{Yam1} and the naturality of $\Spinc$-characteristic classes. 
		Combining~\eqref{eq:h_bord_check},  \eqref{eq:currential-fiber-star}, and $R_{\Spinc}(c,\alpha)=\cw(c)-\delta_{H_2}\alpha$, we get
		\[
		\begin{aligned}
			h_{(c,\alpha)}(\partial y,0)
			&=
			\bigr\langle 
			\cw(y),
			\omega\wedge_\star\cw(c)
			\bigr\rangle 
			-
			\bigr\langle 
			\cw(y),
			\omega\wedge_\star\delta_{H_2}\alpha
			\bigr\rangle
			\quad \mod \mathbb Z ,
			\\
			&=
			\bigl\langle\cw(y),
			\omega\wedge_\star R_{\Spinc}(c,\alpha)\bigr\rangle
			\quad \mod \mathbb Z .
		\end{aligned}
		\]
		which coincides with the right hand side. 
		Thus the bordism relation is respected. 
		
		For the compatibility
		condition, let
		$\varphi\in\Omega_{n-r}(X; V^{\Spinc}_\bullet)/\im\partial_{H_1}$, we have
		\[
		\begin{aligned}
			h_{(c,\alpha)}\bigl(a_{\Spinc}(\varphi)\bigr)
			=
			h_{(c,\alpha)}\bigl(\emptyset,-\varphi\bigr)
			&=
			-\bigl\langle -\varphi,
			\omega\wedge_\star R_{\Spinc}(c,\alpha)\bigr\rangle \quad \mod \mathbb Z \\
			&=
			\bigl\langle \varphi,
			\omega\wedge_\star R_{\Spinc}(c,\alpha)\bigr\rangle
			\quad \mod \mathbb Z \\
			&= \bigl\langle \varphi,
			\omega_{(c,\alpha)}\bigr\rangle
			\quad \mod \mathbb Z. 
		\end{aligned}
		\]
		The independence of the representative $(c,\alpha)$ is proved by the same argument as in~\cite[Lem. 5.48]{Yam1}, 
		and the
		proposition follows.
	\end{proof}

	\medskip

	Now we describe the differential pushforward. 
	Let
	\[
	c=(N,p,f^\nabla_{T(\phi,p)},\widehat\eta)
	\]
	be a geometric $\widehat\tau_2$-twisted stable relative tangential
	$\Spinc$-cochain of relative dimension $r$ over $X$, where
	$p:N\to X$ is a proper submersion.	
	Using the fiber product map~\eqref{fiber}, we now construct the twisted
	differential pushforward map~\eqref{diffpush}
	\begin{align}
		{c}_*: 
		( \widehat{I\Omega^{\mathrm{Spin}^c}_{\mathrm{dR}}} )
		^{n}
		\bigl(N, p^*\widehat{\tau}_3\bigr)
		\longrightarrow
		( \widehat{I\Omega^{\mathrm{Spin}^c}_{\mathrm{dR}}} )^{n-r}
		(X,\widehat{\tau}_1).
	\end{align}
	For a representative $(\omega,h)$ with
	\[
	\omega\in
	\Omega^n_{D_{p^*H_3}\text{-}\clo}
	\bigl(N;N^\bullet_{\Spinc}\bigr)
	\quad \text{and} \quad
	h:
	\widehat{\Omega^{\Spinc}_{n-1}}
	\bigl(N,p^*\widehat\tau_3\bigr)
	\to\mathbb R/\mathbb Z
	\]
	satisfying the compatibility condition, set
	\[
	{c}_*(\omega,h)
	:=
	\left(
	p_!\bigl(\omega\wedge_\star\cw_{\mathrm{pre}}(c)\bigr),
	\,
	h\circ(\times_X c)
	\right).
	\]
	Here $p_!$ is ordinary fiber integration along the proper
	submersion $p$, and $\cw_{\mathrm{pre}}(c)$ is defined as in~\eqref{cwpre}. 
	The compatibility of the pair may be verified directly as in the untwisted case in \cite[Sec.~5]{Yam1}. 
	In particular, consider the trivial fibration $X\times S^1\to X$. The trivial $\Spinc$-structure on $S^1$ induces a differential stable relative tangential $\Spinc$-structure on $X\times S^1$. 
	Then we have the $S^1$-integration map:
	\begin{align}
		\int: 
		( \widehat{I\Omega^{\mathrm{Spin}^c}_{\mathrm{dR}}} )
		^{n+1}
		\big(
		X\times S^1, \widehat{\tau}_1\big) \longrightarrow
		( \widehat{I\Omega^{\mathrm{Spin}^c}_{\mathrm{dR}}} )^{n}(X,\widehat{\tau}_1),
	\end{align}
	which is the analogue of the $S^1$-integration map in the presence of a background twist.

	\section{Gerbe-theoretic models and twisted anomaly map}

	The advantage of gerbe-theoretic models is due to their closer relations to the index-theoretic objects such as spinor bundles and Dirac operators. 
	For a differential-geometric construction of the anomaly map, we first review bundle gerbes and gerbe modules developed in \cite{hitchin1999lectures, Chatterjee1998, Mur96, BCMMS}, among others, and 
	then give gerbe-theoretic models for differential twisted $\Spinc$-bordism and Anderson dual.

	\subsection{Review of bundle gerbes and gerbe modules}

	We begin with a brief account of bundle gerbes following \cite{Mur96}, then we recall gerbe modules and twisted $K$-theory in the sense of \cite{BCMMS}. For a categorical formulation, see \cite{NikWal}. After that, we record the $\Spinc$-gerbe and its canonical module, which play a crucial role in our models for differential twisted $\Spinc$-bordism and Anderson dual.

	\subsubsection{Bundle gerbes with connection and curving}
	
	Let $\pi:Y\to X$ be a surjective submersion. For $p\ge1$ write
	\[
	Y^{[p]}:=\underbrace{Y\times_X\cdots\times_X Y}_{p\ \text{times}}.
	\]
	Let $\pi_i:Y^{[p+1]}\to Y^{[p]}$ be the projection omitting the $i$th factor, and define
	\[
	\delta:\Omega^k(Y^{[p]})\longrightarrow \Omega^k(Y^{[p+1]}),\qquad
	\delta=\sum_{i=1}^{p+1}(-1)^{i-1}\pi_i^*.
	\]
	Then $\delta^2=0$ and there is an exact sequence
	\begin{align}\label{exact_BG}
		0\longrightarrow \Omega^k(X)\xrightarrow{\ \pi^*\ }\Omega^k(Y)\xrightarrow{\ \delta\ }\Omega^k(Y^{[2]})\xrightarrow{\ \delta\ }\Omega^k(Y^{[3]})\longrightarrow\cdots.
	\end{align}
	A \emph{bundle gerbe} on $X$ is a pair $(L,Y)$ consisting of a complex line bundle $L\to Y^{[2]}$ and an isomorphism over $Y^{[3]}$
	\[
	\mu:\ \pi_3^*L\otimes \pi_1^*L \xrightarrow{\ \simeq\ }\pi_2^*L
	\]
	satisfying the usual coherence over $Y^{[4]}$. Fiberwise, we may write the isomorphism at a point $(y_1,y_2,y_3)\in Y^{[3]}$ as
	\[
	\mu_{(y_1,y_2,y_3)}:L_{(y_1,y_2)}\otimes L_{(y_2,y_3)}\xrightarrow{\ \simeq\ }L_{(y_1,y_3)}.
	\]
	Two bundle gerbes $(L,Y)$ and $(L',Y')$ over $X$ are \emph{stably isomorphic} if there is a line bundle $J$ on $Y\times_X Y'$ with
	\[
	L\otimes \delta(J)\cong L'
	\]
	where $\delta(J):=\pi_1^*J\otimes(\pi_2^*J)^{-1}$ is the trivial gerbe on $Y^{[2]}\times_X Y'^{[2]}$. Stable isomorphism classes of bundle gerbes on $X$ are classified by $H^3(X;\mathbb{Z})$ via the Dixmier--Douady class.
	
	\medskip
	
	Lifting bundle gerbes arise naturally from lifting obstructions.
	Let $Y\to X$ be a principal $G$-bundle. Set $\gamma:Y^{[2]}\to G$ by $y_2=y_1\cdot\gamma(y_1,y_2)$. For a central extension
	\[
	1\longrightarrow \uU(1)\longrightarrow \widehat{G}\longrightarrow G\longrightarrow 1,
	\]
	we may pull back $\widehat{G}\to G$ along $\gamma$ to get a $\uU(1)$-bundle $L\to Y^{[2]}$.
	The bundle gerbe structure for $L$ is given by the multiplication of $\widehat{G}$.
	
	\medskip
	
	A \emph{bundle gerbe connection} on $(L,Y)$ is a connection $\nabla^L$ on $L$ compatible with $\mu$. Its curvature $F^L\in\Omega^2(Y^{[2]})$ satisfies $\delta(F^L)=0$, hence by \eqref{exact_BG} there exists a \emph{curving} $\omega\in\Omega^2(Y)$ with $\delta\omega=F^L$. Given a choice of $\omega$, one checks
	\[
	\delta(d\omega)=d(\delta\omega)=dF^L=0,
	\]
	so $d\omega$ descends to a closed $3$-form $H\in\Omega^3(X)$, whose normalized de Rham class corresponds to the Dixmier--Douady class of the bundle gerbe. We write
	\[
	\widehat{\mathcal{G}}=(L,Y,\nabla^L,\omega),
	\]
	as a bundle gerbe with connection and curving. 
	Given a line bundle $P\to Y$ with connection $\nabla^P$, the trivial gerbe
	\[
	\delta P=\pi_1^*P\otimes (\pi_2^*P)^{-1}
	\]
	carries the induced connection $\delta\nabla^P=\pi_1^*\nabla^P-\pi_2^*\nabla^P$ and curving $F^{ P}$. We call $(\delta P,Y,\delta\nabla^P,F^{P})$ the trivial bundle gerbe with induced connection and curving \cite{MS00}.
	For bundle gerbes with connection and curving $(L,Y,\nabla^L,\omega)$ and $(L',Y',\nabla^{L'},\omega')$, a \emph{connection-preserving stable isomorphism}
	$
	(J,\nabla^J)$
	consists of a line bundle $J\to Y\times_X Y'$ with connection $\nabla^J$ such that:
	\begin{itemize}
		\item $L\otimes \delta J \cong L'$ as bundle gerbes;
		\item the induced connection $\nabla^{\delta J}$ is preserved by the isomorphism.
	\end{itemize}
	Additionally, a {connection-preserving stable isomorphism}
	$
	(J,\nabla^J)$ is called a \emph{differential stable isomorphism} if it satisfies the following
	\begin{itemize}
		\item the curvings are related by
		\begin{equation}\label{curving_preserving}
			\omega+F^{J}=\omega' \quad\text{on }Y\times_X Y'.
		\end{equation}
	\end{itemize}
	In particular, \eqref{curving_preserving} implies $H=H'$ on $X$.
	
	\medskip
	
	As in \cite{MS00}, differential stable isomorphism classes of bundle gerbes with connection and curving are classified by Deligne cohomology
	\[
	\widehat{H}^3(X;\mathbb{Z})\cong {H}^2\bigl(X;\ \underline{\uU(1)}\xrightarrow{\tilde d}\Omega^1\xrightarrow{\,d\,}\Omega^2\bigr).
	\]
	Let $\mathsf{Grb}_{\mathrm{conn}}(X)$ be the $2$-groupoid whose objects are bundle gerbes with connection and curving on $X$, $1$-morphisms are
	differential stable isomorphisms, and $2$-morphisms are isomorphisms between $1$-morphisms preserving the induced connections. There is a canonical equivalence
	\[
	\mathsf{Grb}_{\mathrm{conn}}(X) \simeq B^2_{\mathrm{conn}}\uU(1)(X).
	\]
	Thus our notion of degree-3 differential twist can also be formulated using bundle gerbes with connection and curving.
	
	\subsubsection{Gerbe modules with module connection}
	
	Fix a bundle gerbe with connection and curving
	$
	\widehat{\mathcal{G}}=(L,Y,\nabla^L,\omega)
	$
	over $X$. From now on, we temporarily assume its underlying topological bundle gerbe $(L,Y)$ is torsion, i.e., its \DD class is torsion in $H^3(X;\mathbb{Z})$. 
	A \emph{bundle gerbe module} for $\widehat{\mathcal{G}}$ is a complex vector bundle $E\to Y$ together with an isomorphism over $Y^{[2]}$
	\[
	\psi:\ L\otimes \pi_1^*E \xrightarrow{\simeq} \pi_2^*E
	\]
	compatible with the gerbe multiplication. A \emph{module connection} is a connection $\nabla^E$ on $E$ such that
	\[
	\psi\circ(\nabla^L\otimes \mathrm{id}+\pi_1^*\nabla^E)=\pi_2^*\nabla^E\circ\psi.
	\]
	Taking curvature forms and recalling $F^L=\delta\omega=\pi_1^*\omega-\pi_2^*\omega$,
	\[
	F^L\otimes \mathrm{id}+\pi_1^*F^E=\psi^{-1}\circ \pi_2^*F^E\circ\psi,
	\]
	equivalently,
	\[
	\pi_1^*(F^E+\omega\,\mathrm{id})=\psi^{-1}\circ \pi_2^*(F^E+\omega\,\mathrm{id})\circ\psi.
	\]
	Hence there exists a unique $2$-form
	\[
	\widetilde{F^E}\in \Omega^2(X;\End(E))
	\quad\text{with}\quad
	\pi^* \widetilde{F^E}=F^E+\omega\,\mathrm{id},
	\]
	which we call the \emph{descended curvature} of $E$. Here we note $\End(E)$ descends to an Azumaya bundle on $X$. 
	The \emph{twisted Chern character} of $(E,\nabla^E)$ is then defined by
	\[
	\ch_{\widehat{\mathcal{G}}}
	(\nabla^E):=\mathrm{tr}\exp\Bigl(\widetilde{F^E}\Bigr)\in \Omega^{\mathrm{even}}(X).
	\]
	One checks $(d-H)\ch_{\widehat{\mathcal{G}}}
	(\nabla^E)=0$, hence it defines a class in $H^{*}(X, H)$.
	
	In the language of gerbe modules, a \emph{differential trivialization} of $\widehat{\mathcal{G}}$ is a rank one $\widehat{\mathcal{G}}$-module $J$ with module connection $\nabla^J$, 
	such that the descended curvature $\widetilde{F^J}$ vanishes. Hence a differential
	stable isomorphism between $\widehat{\mathcal{G}}$ and $\widehat{\mathcal{G}}'$ is equivalent to a {differential trivialization} of the tensor bundle gerbe 
	$
	\widehat{\mathcal{G}}^{-1}\otimes\widehat{\mathcal{G}'}
	$. 
	Locally speaking, fix a trivialization $(h_{ij},\lambda_i)$ of the \v{C}ech--Deligne cocycle $(g_{ijk},A_{ij},B_i)$ arising from $\widehat{\mathcal{G}}$.
	The local function $h$ defines a rank one $\widehat{\mathcal{G}}$-module, and the local 1‐form $\lambda$ defines a module connection, such that the descended curvature vanishes. Thus the notion of \emph{differential trivialization}  corresponds precisely to the  trivialization of cocycles in the \v{C}ech--Deligne picture.

	\medskip
	Now we turn to twisted $K$-theory. 
	The stable isomorphism classes of $\mathcal{G}$-modules \cite{BCMMS} give a geometric model for twisted $K$-theory 
	$K^0(X, \mathcal{G})$, whose differential extension is constructed in \cite{park}. 
	However, for non-torsion twists $\mathcal{G}$, there are no finite rank $\mathcal{G}$-modules. 
	Instead, we can consider infinite rank modules to model
	twisted $K$-theory $K^0(X, \mathcal{G})$ with non-torsion twist. 
	Let $U_{\mathrm{tr}}\subset \mathcal{U}(\mathcal{H})$ be the subgroup of unitaries that differ from the identity by a trace class operator.
	A \emph{(super) $U_{\mathrm{tr}}$-module} $\mathcal{E}$ consists of a pair of Hilbert bundles $(E^+, E^-)$ with structure groups reduced to $U_{\mathrm{tr}}$, and an isomorphism
	\[
	\psi:\ L\otimes \pi_1^*\mathcal E \xrightarrow{\simeq} \pi_2^*\mathcal E
	\]
	which is compatible with bundle gerbe multiplication. Two such modules are equivalent if they differ by pulling back a line bundle from $X$. 
	The twisted $K$-theory $K^0(X, \mathcal{G})$ is modelled by the stable isomorphism classes of $U_\tr$-modules.
	Since there is a $\PUH$-equivariant
	homotopy equivalence
	\[
	\Fred(\mathcal{H}) \simeq 
	B U_{\tr} \times \mathbb{Z},
	\]
	we may identify $K^0(X, \mathcal{G})$ with the twisted $K$-theory defined via parametrized spectra, see \cite{BCMMS} for a detailed exposition.

	\medskip
	
	For a $U_{\mathrm{tr}}$-module $\mathcal{E}=(E^+,E^-)$, there exist \emph{(super) $U_\tr$-module connections} $\nabla^\mathcal{E}:=(\nabla^{E^+},\nabla^{E^-})$ such that  $\nabla^{E^+}-\nabla^{E^-}$ is trace class (cf. \cite{MS03}). 
	Then for each $p\ge 1$, 
	\[
	\bigl(F^{E^+}+\omega I\bigr)^p-\bigl(F^{E^-}+\omega I\bigr)^p
	\]
	is trace class, and the even form
	\begin{align}\label{ch_Z2}
		\mathrm{tr}\left(
		\exp\bigl((F^{E^+}+\omega I)\bigr)
		-
		\exp\bigl((F^{E^-}+\omega I)\bigr)
		\right)=
		\exp(\omega)\,\mathrm{tr}\bigl(\exp(F^{E^+})-\exp(F^{E^-})\bigr)
	\end{align}
	is globally defined on $Y$ and descends to an even $(d-H)$-closed form on $X$, which we denote as the \emph{twisted Chern character form} $\ch_{\widehat{\mathcal{G}}}
	(\nabla^\mathcal{E})$.

	\subsubsection{The $\Spin^c$-gerbe and its canonical module}
	
	We now introduce $\Spin^c$-gerbe and the associated canonical module, which are closely related to twisted $\Spin^c$-structures. 
	Let $E\to X$ be an oriented rank $n$ vector bundle with connection $\nabla^E$. Its frame bundle
	$P_{\SO}(E)\to X$ is a principal $\SO(n)$-bundle equipped with the induced principal connection.
	Consider the lifting bundle gerbe construction
	\[
	W \rightarrow P_{\SO}(E)^{[2]},
	\] 
	for the central extension $\uU(1)\to \Spin^c(n)\to \SO(n)$, and denote
	\[
	\mathcal{G}^{\Spin^c}_E=(W,P_{\SO}(E)),
	\]
	as the \emph{$\Spinc$-gerbe} associated to $E$. 
	There is a canonical bundle gerbe connection $\nabla^W$ induced from
	the local $1$-forms $A_{ij}$ in \eqref{connection}.
	Upon the choice of zero curving, we obtain a bundle gerbe with connection and curving
	\[
	\widehat{\mathcal{G}^{\Spin^c}_E}=(W,\;P_{\SO}(E),\;\nabla^W,\;0),
	\]
	corresponding to the chosen lift~\eqref{connW3}
	\[
	W_3^{\mathrm{conn}}\colon B_\nabla\SO \longrightarrow B_{\mathrm{conn}}^2\uU(1).
	\]
	There is a canonical gerbe module $\mathcal{S}$ over $\widehat{\mathcal{G}^{\Spin^c}_E}$ which plays the role of spinor bundle for non-$\Spin^c$ vector bundles.
	Let $\rho^c_n\colon \Spinc(n)\to \GL(\Delta_n)$ be the complex spin representation.
	Define the trivial vector bundle
	\[
	\mathcal{S}:=\Delta_n\times P_{\SO}(E)\ \longrightarrow\ P_{\SO}(E).
	\]
	The module structure
	\[
	\phi_{(p_1,p_2)}\colon W_{(p_1,p_2)}\otimes \mathcal{S}_{p_2}\xrightarrow{\;\simeq\;}\mathcal{S}_{p_1},\qquad
	(p_1,p_2)\in P_{\SO}(E)^{[2]},
	\]
	can be described as follows. For a point $(p_1,p_2,\widetilde g)$ of the fiber $W_{(p_1,p_2)}$ with $p_1=p_2\,g$ for $g\in \SO(n)$ and $\widetilde g\in \Spinc(n)$ a lift of $g$, set
	\[
	\phi_{(p_1,p_2)}\bigl((p_1,p_2,\widetilde g)\otimes(p_2,v)\bigr)
	=(p_1,\rho^c_n(\widetilde g)v).
	\]
	One checks this is indeed a module structure over the $\Spinc$ gerbe. When $n=2k$, $\Delta_{2k}=
	\Delta_{2k}^+\oplus \Delta_{2k}^-$.
	In this case, the module $\mathcal{S}$ is of rank $2^k$, and splits as the direct sum $\mathcal{S}=\mathcal{S}^+\oplus \mathcal{S}^-$, where 
	$
	\mathcal{S}^{\pm}:= \Delta_{2k}^{\pm} \times P_{\SO}(E)
	$
	are both gerbe modules of rank $2^{k-1}$ with module connections over the $\Spinc$ gerbe. 
	When $n=2k+1$, $\Delta_{2k+1}$ is irreducible, thus $\mathcal{S}$ is an irreducible module of rank $2^{k}$.
	
	The module $\mathcal{S}$ carries a canonical module connection $\nabla^\mathcal{S}$ induced from $\nabla^E$. Indeed, for $\nabla^E$ there is a connection 1-form $\theta_E \in \Omega^1(P_{\SO}(E), \so_{n})$ on $E$.
	Under the canonical splitting $\so_n\hookrightarrow \spinc_n$ and the pushforward 
	$
	\rho^c_*: \spinc_{n} \to \End(\Delta_{n}),  
	$
	one gets a connection 1-form 
	$
	\rho^c_*({\theta_E})
	\in
	\Omega^1(P_{\SO}(E), \End(\Delta_{n})  )
	$, and the corresponding connection $\nabla^{\mathcal{S}}$ on $\mathcal{S}$ defines a module connection.

	There are two Azumaya bundles over $X$ which arise naturally from $E$. The first one is given by the endomorphism bundle  $\End(\mathcal{S})$ descends to an Azumaya bundle over $X$, which we still denote by $\End(\mathcal{S})$, with its natural induced module connection $\nabla^{\End(\mathcal{S})}$. 
	The second Azumaya bundle $\Cl^+(E)$ is defined as follows
	\[
	\Cl^+(E):=
	\begin{cases}
		\Cl(E),
		&  \mathrm{rank}(E)=2k  \\
		\Cl^0(E), 
		&   \mathrm{rank}(E)=2k+1
	\end{cases}
	\]	
	which carries a connection induced by $\nabla^{E}$. 
	In fact, these two Azumaya bundles are related by a natural connection-preserving isomorphism
	\begin{align}
		\Psi_0: 
		\Cl^+(E) 
		\cong 
		\End(\mathcal{S}).
	\end{align}
	We will use an analogue of this isomorphism in 
	our gerbe-theoretic definition for differential twisted $\Spinc$-structures.

	\subsection{Differential models via gerbe modules}

	In this section we give gerbe-theoretic models for differential twisted $\Spin^c$-bordism and for the corresponding Anderson dual. Namely, our differential twists are given by a bundle gerbe (not necessarily torsion) over $X$ with connection and curving
	\[
	\widehat{\mathcal{G}}=
	(L,Y,\nabla^L,\omega),
	\]
	and differential twisted $\Spinc$-structures are described
	using gerbe modules over the pullback of differential twist. 
	Given a map $f\colon M\to X$ and an oriented rank $n$ vector bundle $E\to M$ with connection $\nabla^E$, we first introduce the topological version of twisted $\Spinc$-structures with the underlying bundle gerbe  ${\mathcal{G}}$. 
	
	\begin{defi}
		A ${\mathcal{G}}$-twisted $\Spin^c$-structure on $E$ consists of
		\[
		(S^c,\;\Psi)
		\]
		where $S^c$ is a $f^*{\mathcal{G}}$-module over $M$, and
		\[
		\Psi\colon \Cl^+(E)\xrightarrow{\;\cong\;} \End(S^c)
		\]
		is an isomorphism of Azumaya bundles over $M$. If $n=2k$, $S^c\cong S^{c+}\oplus S^{c-}$ with each summand of rank $2^{k-1}$.  If $n=2k+1$, then $S^c$ is irreducible of rank $2^k$.
	\end{defi}
	
	For the differential version, we have
	
	\begin{defi}\label{def: diffSpincBG}
		A differential $\widehat{\mathcal{G}}$-twisted $\Spin^c$-structure on $E$ consists of
		\[
		(\nabla^{S^c},\;\Psi)
		\]
		where $\nabla^{S^c}$ is a module connection with underlying $f^*\widehat{\mathcal{G}}$-module ${S^c}$, such that 
		$
		(S^c,\;\Psi)
		$ is a
		${\mathcal{G}}$-twisted $\Spin^c$-structure, with the additional requirement that $\Psi$
		is connection-preserving. 
	\end{defi}
	\noindent
	For each $(\nabla^{S^c}, \Psi)$, there is an associated canonical 2-form $\kappa'(\nabla^{S^c}, \Psi)$, defined by
	\begin{align}\label{kappa_geo}
		\kappa'(\nabla^{S^c}, \Psi):=	\tr_0(\widetilde{F^{S^c}}),
	\end{align}
	where $\tr_0={\tr}/{\mathrm{rank}}$ is the normalized trace. We will show this assignment
	is compatible with (\ref{kappa}) in Theorem~{\ref{thm:equiv-mod}}.
	
	\begin{defi}
		An $n$-dimensional geometric $\widehat{\mathcal{G}}$-twisted $\Spinc$-chain over $X$ is a quadruple 
		\[
		(M, f,  \nabla^{S^c}, \Psi)
		\]
		where 
		\begin{itemize}
			\item $M$  is a compact oriented $n$-dimensional Riemannian manifold with collar boundary, along which all data are assumed to be constant;
			\item $f: M\to X$ is a smooth map;
			\item $( \nabla^{S^c}, \Psi)$ is a differential $\widehat{\mathcal{G}}$-twisted $\Spinc$-structure on the tangent bundle $TM$.
		\end{itemize} 		
	\end{defi}
	\noindent
	Two chains $(M,f,\nabla^{S^c},\Psi)$ and $(M',f',\nabla^{S'^c},\Psi')$ are said to be isomorphic if there exists an orientation-preserving diffeomorphism $h:M\to M'$ preserving collars such that $f = f'\circ h$, and a connection-preserving isomorphism of $f^*\widehat{\mathcal{G}}$-modules $S^c \xrightarrow{\cong} h^{*}S'^c$ over $M$, intertwining the gerbe action and Clifford action.
	The collection of differential $\widehat{\mathcal{G}}$-twisted $\Spinc$-chains forms an abelian group
	$\widetilde{C}^{\Spinc}_n
	(X, \widehat{\mathcal{G}})$, and
	there is a natural boundary map
	\[
	\partial: \widetilde{C}^{\Spinc}_n
	(X, \widehat{\mathcal{G}})
	\to
	\widetilde{C}^{\Spinc}_{n-1}
	(X, \widehat{\mathcal{G}})
	\]
	which we describe as follows. 
	
	For $n=2k$, $S^c_M=S_M^{c+} \oplus 
	S_M^{c-}$ is of rank $2^k$. 
	On the collar neighborhood $\partial M$,
	we have the decomposition
	$TM|_{\partial M}\cong T\partial M\oplus \underline{\mathbb{R}}$, where the trivial line bundle $\underline{\mathbb{R}}$ is spanned by the outward unit normal vector $\mathbf{n}$.  
	With $\mathbf{n}$, we pick out a canonical section of the associated sphere bundle, 
	which in turn reduces the principal $\SO(2k)$-bundle 
	$P_{\SO}(TM)|_{\partial M}$
	to the $\SO(2k-1)$-bundle
	$P_\SO(T\partial M)$. 
	Then by considering the global Clifford action of $\mathbf{n}$ on $\partial M$, we have the following identification $ 
	S_M^{c+}|_{\partial M} \cong 
	S_M^{c-}|_{\partial M}
	$
	of vector bundles over $P_\SO(T\partial M)$,
	together with the isomorphism:
	\[
	\Cl^+(T\partial M)\cong 
	\End(S_M^{c+}|_{\partial M})\cong
	\End(S_M^{c-}|_{\partial M}). 
	\]
	So we may set
	$
	\partial S^c_M:=
	S_M^{c+}|_{\partial M}
	$
	as the boundary module. 
	
	For $n=2k+1$, 
	the module $S_M^c$ on $M$ is of rank $2^k$ and irreducible. 
	We set the induced module on $\partial M$ by
	$
	\partial S^c_M:=S_M^{c}|_{\partial M}
	$. The Clifford action
	$
	\Cl^+(TM) \cong \End (S^c_M)
	$
	naturally restricts to the boundary,
	\[
	\Cl^+(T\partial M)\cong 
	\End(S_M^{c}|_{\partial M}).
	\]

	There is also a gerbe-theoretic version of the twisted Chern--Weil map~\eqref{cw}
	\begin{align}\label{cwgerbe}
		\cw:
		\widetilde{C}^{\Spinc}_n
		(X, \widehat{\mathcal{G}})
		\rightarrow
		\Omega _{n}(X;V^{\mathrm{Spin}^c}_{\bullet}). 
	\end{align}
	For each geometric $\widehat{\mathcal{G}}$-twisted $\Spinc$-chain $(M, f, \nabla^{S^c}, \Psi)$, the associated current sends a monomial generator $
	\omega \otimes p_I  \zeta^k
	\in \Omega^*(X; N_{\Spinc}^{\bullet})$ to the integral
	\[
	\int_M f^*\omega \wedge 
	p_I(M) \wedge
	\kappa'(\nabla^{S^c}, \Psi)^k,
	\]
	where $p_I (M)$ is the closed Pontryagin form on $M$ arising from the Pontryagin polynomial $p_I$ and the Riemannian connection of $M$. 
	With the map~\eqref{cwgerbe} in hand, 
	now we are ready to give the gerbe-theoretic definition for differential twisted $\Spinc$-bordism. 
	\begin{defi}\label{defi:gerbe-spinc-bord}
		For each integer $n$, set  
		\[
		{\widehat{\Omega^{\mathrm{Spin}^c}_{n-1}}(X, \widehat{\mathcal{G}})}
		:=\{
		(M, f,  \nabla^{S^c}, \Psi,  \phi)
		\}/\sim, 
		\]
		where
		$
		(M, f, \nabla^{S^c}, \Psi)
		$ is a closed $(n-1)$-dimensional differential twisted $\Spinc$-chain over $X$,  and
		\[
		\phi \in
		{\Omega _n(X;V^{\mathrm{Spin}^c}_{\bullet})/\mathrm{im}\partial_H}. 
		\]
		The equivalence relation $\sim$ is generated by isomorphism, disjoint union, together with the bordism relation
		\[
		(\partial W,\partial F,\nabla^{\partial S^{c}_{W}},\partial\Psi_{W},0)
		\sim
		(\emptyset^{},\emptyset,\emptyset,\emptyset,
		-\cw(W,F,\nabla^{S^{c}_{W}},\Psi_{W})),
		\]
		where $(W,F,\nabla^{S^{c}_{W}},\Psi_{W})
		\in \widetilde{C}^{\Spinc}_{n}(X,\widehat{\mathcal{G}})$.
	\end{defi}
	
	Similarly, there is a gerbe-theoretic model for the differential Anderson dual.
	
	\begin{defi} 
		For each integer $n$, set 
		\[
		\bigl(
		\widehat{
			I\Omega^{\Spinc}_{\mathrm{dR}}
		}
		\bigr)^{n}(X, \widehat{\mathcal{G}})
		:=\{(\omega, h)\},
		\]
		where
		\begin{itemize}
			\item 
			$\omega \in
			\Omega^n_{D_H\mathrm{-clo}}
			(X; N_{\Spin^c}^{\bullet})  $ 
			is a $D_H$-closed 
			$N_{\Spin^c}^{\bullet}$-valued form,
			\item $h:
			\widehat{
				\Omega^{\Spinc}_{n-1}
			}
			(X, \widehat{\mathcal{G}})
			\to\mathbb{R}/\mathbb{Z}$ 
			is a group homomorphism,
			\item they satisfy the compatibility condition 
			\[
			h\circ a^{\Spinc}=
			\mathrm{mod}\mathbb{Z} \circ
			\langle- ,\omega\rangle.
			\]
		\end{itemize}
	\end{defi}
	\noindent
	The structure maps for both gerbe-theoretic models can be defined in a completely analogous way as in Section 3. 
	Under the equivalence $\mathsf{Grb}_{\mathrm{conn}}(X) \simeq B^2_{\mathrm{conn}}\uU(1)(X)$, one may similarly formulate the axioms for differential extension in the context of gerbe-theoretic models, as in Definition~\ref{defi:diffextco} and Definition~\ref{defi:diffextho}.

	\medskip
	
	Let $\widehat{\mathcal G}$ be a bundle gerbe with connection and curving representing the same differential class as $\widehat\tau$ in $\widehat H^3(X;\mathbb Z)$.
	By functoriality, a morphism of differential twists $\widehat{\tau}\to\widehat{\tau}'$ induces a canonical isomorphism
	\begin{equation}\label{stable_tau}
		\widehat{\Omega^{\Spinc}_*}(X,\widehat{\tau})\xrightarrow{\ \cong\ }\widehat{\Omega^{\Spinc}_*}(X,\widehat{\tau}').
	\end{equation}
	Thus we may write $\widehat{\Omega^{\Spinc}_*}(X,[\widehat{\tau}])$ since it depends only on the isomorphism class of $\widehat\tau$. Similarly, we have 
	
	\begin{prop}\label{prop:stable-inv-gerbe}
		If $(K,\nabla^K)$ is a 
		differential stable isomorphism of bundle gerbes $\widehat{\mathcal G}\to\widehat{\mathcal G}'$ over $X$, then there is a canonical isomorphism
		\[
		\widehat{\Omega^{\Spinc}_*}(X,\widehat{\mathcal G})\xrightarrow{\ \cong\ }\widehat{\Omega^{\Spinc}_*}(X,\widehat{\mathcal G}'),
		\]
		which is compatible with the structure maps of the differential extension.
	\end{prop}
	
	\begin{proof}
		Let $(M, f,\nabla^{S^c},\Psi, \varphi)$ be a differential $\widehat{\mathcal G}$-twisted cycle where 
		$S^c$ is a module over $f^*\widehat{\mathcal G}$ and $\Psi:\Cl^+(TM)\xrightarrow{\cong}\End(S^c)$ a connection-preserving Clifford isomorphism. As in \cite{BCMMS}, set
		\[
		S^{c\,\prime}:=S^c\otimes K^{-1},
		\]
		over the correspondence space. 
		Then $S^{c\,\prime}$ descends to a $f^*\widehat{\mathcal G}'$-module with induced connection, and $\Psi$ passes to a connection-preserving isomorphism $\Cl^+(TM)\xrightarrow{\cong}\End(S^{c\,\prime})$. 
		This construction gives a well-defined isomorphism, whose inverse is constructed by tensoring $K$.		
		For compatibility with the structure maps, it suffices for us to check
		\[
		\tr_0(\widetilde{F^{S^c}})
		=
		\tr_0(\widetilde{F^{S^{c\prime}}}).
		\] 
		This follows from the fact that as a differential stable isomorphism,  $(K,\nabla^K)$ has zero descended curvature. 
	\end{proof}
	\noindent
	Thus we may write $\widehat{\Omega^{\Spinc}_*}(X,[\widehat{\mathcal{G}}])$ since it depends only on the differential stable isomorphism class of $\widehat{\mathcal{G}}$. 
	To show the equivalence of models, it is essential to prove the following

	\begin{thm}\label{thm:equiv-mod}
		The two bordism groups are isomorphic
		\[
		\widehat{\Omega^{\Spinc}_*}(X,[\widehat\tau])\;\cong\;\widehat{\Omega^{\Spinc}_*}(X,[\widehat{\mathcal G}]),
		\]
		and compatible with structure maps of differential extensions. 
		Consequently, 
		\[
		\bigl(
		\widehat{
			I\Omega^{\Spinc}_{\mathrm{dR}}
		}
		\bigr)^{n}(X, [\widehat{\tau}])
		\xrightarrow{\cong }
		\bigl(
		\widehat{
			I\Omega^{\Spinc}_{\dR}
		}
		\bigr)^{n}(X, [\widehat{\mathcal{G}}]),
		\]
		which are also compatible with structure maps.
	\end{thm}
	\begin{proof}
		Given a differential $\widehat{\mathcal G}$-twisted cycle $(M, f,\nabla^{S^c},\Psi, \varphi)$, recall that the canonical module $\mathcal{S}\to P_{\SO}(TM)$ is a gerbe module with module connection over the $\Spinc$-gerbe 
		\[
		\widehat{\mathcal G^{\Spinc}_{TM}}=
		(W, P_\SO(TM), \nabla^W, 0),
		\]
		together with a connection-preserving isomorphism
		\[
		\Psi_0: \Cl^+(TM) \cong \End(\mathcal{S}). 
		\] 
		Set the correspondence space $Z:=P_\SO(TM)\times_M f^*Y$ and denote $\pi_2, \pi_1$ by the projections to $P_{\SO}(TM)$ and $ f^*Y$ respectively. 
		With the isomorphisms $\Psi$ and $\Psi_0$, define a balanced tensor product
		\begin{align}\label{J}
			J:=\Hom_{\Cl^+(TM)}
			(\pi_2^*\mathcal{S},\,\pi_1^*S^c)
			\cong
			(\pi_2^*\mathcal S)^\vee
			\otimes_{\Cl^+(TM)}
			\pi_1^*S^c,
		\end{align}
		which is a complex line bundle over $Z$ by Schur’s lemma, with induced tensor connection denoted by $\nabla^J$. 
		Moreover, $J$ is equipped with a $(\widehat{\mathcal{G}^{\Spinc}_{TM}})^{-1} \otimes f^*\widehat{\mathcal{G}}$-module structure from the respective module structures of $\mathcal S$ and $S^c$, 
		with $\nabla^J$ a module connection. Summarizing, 
		$(J,\nabla^J)$ is a connection-preserving stable isomorphism between 
		$(\widehat{\mathcal{G}^{\Spinc}_{TM}})^{-1}$ and $f^*\widehat{\mathcal{G}}$, which is precisely a $1$-simplex between  $W_3^{\nabla}\circ f^\nabla_{TM}$ and $\iota_2\widehat\tau\circ f$ in $B^2_{\nabla}\uU(1)(M)$, i.e. a differential $\widehat\tau$-twisted $\Spinc$-structure on $TM$, by the local description of bundle gerbes with connections.

		\medspace
		Conversely, for a differential $\widehat{\tau}$-twisted $\Spinc$-cycle $(M,f, f^\nabla_{TM}, \widehat{\eta}, \varphi)$,  $\widehat{\eta}$ gives rise to
		a rank one module 
		$J$ over 
		$
		(\widehat{\mathcal{G}^{\Spinc}_{TM}})^{-1} \otimes f^*\widehat{\mathcal{G}}
		$
		with module connection $\nabla^J$. It suffices for us to construct a module $S^c$ over $f^*\mathcal{G}$ with module connection and Clifford action. 
		As in \cite{BCMMS}, 
		the product bundle
		$\pi_2^*\mathcal{S} \otimes J$ over $Z$ descends to a vector bundle ${S^c}$ over $f^{*}Y$
		which carries the induced $f^*\widehat{\mathcal{G}}$-module structure and the induced module connection $\nabla^{S^c}$. 
		Furthermore, the $\Cl^+(TM)$-action $\Psi_0$ of $\mathcal{S}$ carries over to $S^c$
		\[
		\Psi_{S^c}: 
		\Cl^+(TM) \to \End(S^c),
		\]
		which is connection-preserving because $\Psi_0$ is. 
		Hence $(M, f, \nabla^{S^c},\Psi_{S^c},\varphi)$ gives a differential $\widehat{\mathcal{G}}$-twisted $\Spinc$-cycle.

		To verify compatibility of the structure maps in the two models, it suffices to check for
		the curvature map, which boils down to the compatibility of \eqref{kappa} and \eqref{kappa_geo}.
		Let $(S^c,\nabla^{S^c},\Psi)$ be a differential $\widehat{\mathcal{G}}$-twisted $\Spinc$-structure on $TM$. The descended curvatures satisfy
		\begin{align}\label{eq: desc_cur}
			\widetilde{F^{ S^c}}
			=\widetilde{F^{\mathcal{S}}}+\widetilde{F^{ J}}\cdot \id.
		\end{align}
		Since the curving of $\widehat{\mathcal G^{\Spinc}_{TM}}$ is zero, the descended curvature
		of the canonical gerbe module $\mathcal S$ agrees with its ordinary curvature.
		Moreover, the connection on $\mathcal S$ is induced from the $\so_n$-connection
		on $P_{\mathrm{SO}}(TM)$, hence $\widetilde{F^{\mathcal S}}$ takes values in the image of
		$\so_n$ under the spin representation, and therefore
		\[
		\tr_0\widetilde{F^{\mathcal S}}=0.
		\]
		On the other hand, by the local expression of $\kappa$ in \eqref{kappa_local}, one has
		\[
		\kappa(\widehat\eta)=\widetilde{F^J}.
		\]
		Taking normalized trace in \eqref{eq: desc_cur}, we obtain
		\[
		\kappa'(\nabla^{S^c}, \Psi_{S^c})=
		\tr_0\widetilde{F^{S^c}}
		=
		\widetilde{F^J}
		=
		\kappa(\widehat\eta).
		\]
		The Anderson dual groups are defined functorially from the corresponding differential bordism theories; the induced isomorphism on Anderson duals therefore follows formally.
		This completes the proof.
	\end{proof}

	\subsection{The twisted anomaly map}\label{sec:twisted-anomaly-map}
	In this section we describe our differential-geometric construction of  the map
	\[
	\widehat{\Phi}_{\widehat{\mathcal G}}\colon
	\widehat K^{0}\bigl(X,\widehat{\mathcal G}^{-1}\bigr)
	\longrightarrow
	\bigl(\widehat{I\Omega^{\Spinc}_{\dR}}\bigr)^{2k}\bigl(X,\widehat{\mathcal G}\bigr).
	\]
	In the torsion case, Park \cite{park} gives a model for differential twisted $K$-theory using finite rank gerbe modules. For our purposes, we extend this model to the non-torsion setting by allowing super $U_{\tr}$-gerbe modules, thereby obtaining a geometric model for $\widehat K^{0}(X,\widehat{\mathcal G}^{-1})$ generated by representatives $x=(\mathcal E,\nabla^{\mathcal E},\rho)$.
	Then for each representative $x$, we define the curvature component $\omega_x$ using the twisted Chern character, and define the functional component $h_x$ using reduced eta-invariant.  We prove that the pair 
	$(\omega_x, h_x)$ defines an element in the twisted Anderson dual, and the assignment
	\[
	x\mapsto (\omega_x, h_x)
	\]
	gives rise to the desired twisted anomaly map.

	\subsubsection{Differential twisted $K$-theory}\label{sec: K}
	
	Fix a bundle gerbe with connection and curving $\widehat{\mathcal G}$ over $X$,
	with curving $\omega$ and global 3-curvature $H$.  We set
	\[
	x=(\mathcal E,\nabla^{\mathcal E},\rho),
	\]
	where $\mathcal E=(E^+,E^-)$ is a super $U_{\tr}$-module over
	$\widehat{\mathcal G}$, $\nabla^{\mathcal E}=(\nabla^{E^+},\nabla^{E^-})$ is a
	pair of compatible super $U_{\tr}$-module connections, and
	\[
	\rho\in \Omega^{\mathrm{odd}}(X)/\im(d-H).
	\]
	Recall in \cite{MS03}, the twisted Chern character form of $\nabla^{\mathcal{E}}$ is defined to be the global descended form
	\begin{equation}\label{eq:chern-form}
		\ch_{\widehat{\mathcal G}}(\nabla^{\mathcal E})
		:=\exp(\omega)\,\mathrm{tr}\bigl(\exp(F^{E^+})-\exp(F^{E^-})\bigr),
	\end{equation}
	which is $(d-H)$-closed. 
	The twisted Chern--Simons term for super  $U_{\mathrm{tr}}$-modules is described as follows.
	Choose smooth paths $t\mapsto\nabla_t^{E^{\pm}}=d+A_t^{E^\pm}$ with curvature $F_t^{E^\pm}$ and $\dot A_t^{E^\pm}=\tfrac d{dt}A_t^{E^\pm}$,
	such that the connection forms $A_t^{E^\pm}$, and hence
	$\dot A_t^{E^\pm}$, take values in $\mathrm{Lie}(U_{\tr})$.
	Define
	\begin{equation}\label{eq:twiCS}
		\CS_{\widehat{\mathcal G}}
		(\nabla_0^\mathcal E,\nabla_1^\mathcal E)
		:=
		\int_0^1\mathrm{tr}
		\Big(
		\dot A_t^{E^+}\exp(F_t^{E^+}+\omega I)-
		\dot A_t^{E^-}\exp(F_t^{E^-}+\omega I)
		\Big)dt,
	\end{equation}
	which descends to $X$ and satisfies
	\begin{equation}\label{eq:CS-transgression}
		(d-H)\,\CS_{\widehat{\mathcal G}}(\nabla_0^\mathcal E,\nabla_1^\mathcal E)
		=\ch_{\widehat{\mathcal G}}(\nabla_1^\mathcal E)-\ch_{\widehat{\mathcal G}}(\nabla_0^\mathcal E).
	\end{equation}
	
	\begin{defi}
		The differential twisted $K^0$-group
		$\widehat K^{0}(X,\widehat{\mathcal G})$ is generated by tuples
		\[
		x=(\mathcal E,\nabla^\mathcal E,\rho),
		\]
		modulo the relations:
		\begin{itemize}
			\item 
			$(\mathcal E_0,\nabla^{\mathcal E_0},\rho)\sim(\mathcal E_1,\nabla^{\mathcal E_1},\rho)$ for connection-preserving isomorphisms. 
			\item  $(\mathcal E_0,\nabla^{\mathcal E_0},\rho_0)+(\mathcal E_1,\nabla^{\mathcal E_1},\rho_1)
			\sim (\mathcal E_0\oplus\mathcal E_1,\nabla^{\mathcal E_0}\oplus\nabla^{\mathcal E_1},\rho_0+\rho_1)$.
			\item 
			$(\mathcal E,\nabla_0^\mathcal E,0)\sim(\mathcal E,\nabla_1^\mathcal E,\CS_{\widehat{\mathcal G}}(\nabla_0^\mathcal E,\nabla_1^\mathcal E))$.
			\item $(\mathcal E,\nabla^{\mathcal{E}},0)\sim0$, for $\mathcal{E}=(E,E)$, $\nabla^{\mathcal{E}}=(\nabla, \nabla)$. 
		\end{itemize}
		We write $[x]$ for the differential twisted $K$-class represented by $x$.
	\end{defi}
	\noindent
	The structure maps are defined as follows
	\[
	\begin{aligned}
		I_K:\ \widehat K^{0}(X,\widehat{\mathcal G})&\to K^{0}(X,{\mathcal G}),&
		[(\mathcal E,\nabla^\mathcal E,\rho)]&\mapsto [E^+]-[E^-],\\
		R_K:\ \widehat K^{0}(X,\widehat{\mathcal G})&\to \Omega^{\mathrm{even}}_{(d-H)\text{-clo}}(X),&
		[(\mathcal E,\nabla^\mathcal E,\rho)]&\mapsto \ch_{\widehat{\mathcal G}}(\nabla^\mathcal E)-(d-H)\rho,\\
		a_K:\ \Omega^{\mathrm{odd}}(X)/\mathrm{im}(d-H)&\to \widehat K^{0}(X,\widehat{\mathcal G}),&
		\rho&\mapsto [(0,0,-\rho]).
	\end{aligned}
	\]
	The well-definedness of $R_K$ follows from the transgression identity~\eqref{eq:CS-transgression}.
	Let
	\[
	\mathcal M_K^*(X,\widehat{\mathcal G})
	:=\bigl(\Omega^*(X)[u,u^{-1}],\ d_K=d-H\wedge u^{-1}\bigr),
	\qquad |u|=2,
	\]
	with the usual periodic grading.  The twisted Chern character form construction gives a natural
	map
	\[
	\ch'_K\colon K^0(X,\mathcal G)\longrightarrow H^{0}(\mathcal M_K^*(X,\widehat{\mathcal G})).
	\]
	
	\begin{prop}\label{prop:diff-twisted-K-model}
		The tuple
		\[
		\bigl(\widehat K^0(-,-),\mathcal M_K^*(-,-),\ch'_K,R_K,I_K,a_K\bigr)
		\]
		is a differential extension of twisted $K$-theory in degree zero, in
		the sense of Definition~2.2.
	\end{prop}
	
	\begin{proof}
		The underlying topological $K$-theory represented by super $U_{\tr}$-gerbe modules agrees with
		the gerbe module model of twisted $K$-theory, and
		the twisted
		Chern character form construction gives the twisted
		Chern character after passing to twisted de Rham
		cohomology, hence is a real isomorphism, 
		see \cite{BCMMS,MS03}.  
		This gives condition (i) of Definition~\ref{defi:diffextco},
		and the
		commutativity condition (ii) follows directly from construction. 
		
		For the exact sequence
		\[
		K^{-1}(X,\mathcal G)
		\xrightarrow{\ch_K}
		\Omega^{\mathrm{odd}}(X)/\im(d-H)
		\xrightarrow{a_K}
		\widehat K^0(X,\widehat{\mathcal G})
		\xrightarrow{I_K}
		K^0(X,\mathcal G)
		\to 0, 
		\]
		surjectivity of $I_K$ follows by choosing a super $U_{\tr}$-gerbe-module
		representative of a topological twisted $K$-class and then choosing compatible
		module connections.  The inclusion $\im(a_K)\subset\ker(I_K)$ is immediate.
		Conversely, for $x=(\mathcal E,\nabla^{\mathcal E},\rho)$, suppose $I_K[x]=0$, then there is a
		$U_{\tr}$-gerbe module $F$ and a module isomorphism
		\begin{equation}\label{K-stable}
			E^+\oplus F \xrightarrow{\cong} E^-\oplus F .
		\end{equation}
		Choose a compatible connection $\nabla^F$ on $F$, then 
		$
		(F,F,\nabla^F,\nabla^F,0),
		$
		represents the zero class. We have
		\[
		x\sim
		(E^+\oplus F,\;E^-\oplus F,\;
		\nabla^{E^+}\oplus\nabla^F,\;\nabla^{E^-}\oplus\nabla^F,\;\rho).
		\]
		Let $\nabla_0=(\nabla^{E^+}\oplus\nabla^F,\;\nabla^{E^-}\oplus\nabla^F)$, and 
		$\nabla_1$ be the diagonal connection on $\mathcal F=(E^+\oplus F,E^-\oplus F)$ compatible with \eqref{K-stable}. The change-of-connection
		relation gives
		\[
		x\sim
		(\mathcal{F}, \nabla_0, \rho)\sim
		(\mathcal{F}, \nabla_1, \rho+\CS_{\widehat{\mathcal G}}(\nabla_0,\nabla_1))
		\sim
		(0,0,\rho+\CS_{\widehat{\mathcal G}}(\nabla_0,\nabla_1)),
		\]
		since $(\mathcal F,\nabla_1,0)$ is degenerate.  On the other hand, we have
		\[
		[(0,0,\rho+\CS_{\widehat{\mathcal G}}(\nabla_0,\nabla_1))]
		=
		a_K\bigl(-\rho-\CS_{\widehat{\mathcal G}}(\nabla_0,\nabla_1)\bigr).
		\]
		Thus \(\ker(I_K)\subset \mathrm{im}(a_K)\), and hence
		\(\ker(I_K)=\mathrm{im}(a_K)\).
		
		Finally, $a_K(\rho)=0$ precisely when $\rho$ is the
		odd twisted Chern character of a class in $K^{-1}(X,\mathcal G)$, which gives
		exactness at the form term.  
		This is the same standard argument in differential $K$-theory, see \cite{simons2008structured, freed2010index}, 
		with $d$ replaced by $d-H$ and the usual
		Chern--Simons form replaced by \eqref{eq:twiCS}. 
	\end{proof}

	\subsubsection{Construction of the anomaly map}\label{sec:ano_cons}
	Consider the differential twisted $K$-theory $\widehat K^0(X,\widehat{\mathcal G}^{-1})$, and note that for its differential twist $\widehat{\mathcal G}^{-1}$,  the 3-curvature is $-H$. 
	Fix a representative
	\[
	x=(\mathcal E,\nabla^{\mathcal E},\rho),
	\qquad \mathcal E=(E^+,E^-),
	\]
	and write the curvature of its class as
	$
	R_K(x)=
	\ch_{\widehat{\mathcal G}^{-1}}(\nabla^{\mathcal E})-(d+H)\rho. 
	$
	
	\medskip
	
	Define the
	curvature component in the desired anomaly class by
	\begin{equation}\label{construction: om}
		\omega_x:=
		\bigl\{
		R_K(x)
		\otimes \widehat A e^\zeta
		\bigr\}^{(2k)}. 
	\end{equation}
	Combining the facts that $R_K(x)$ is $(d+H)$-closed and $\partial_\zeta(e^\zeta)=e^\zeta$, 
	we have $R_K(x)\otimes \widehat A e^\zeta$ is $D_H$-closed, hence 
	\[
	\omega_x \in \Omega^{2k}_{D_H-\mathrm{clo}}(X; N^\bullet_{\Spinc}).
	\]
	In the following, we construct a functional component
	\[
	h_x\in \Hom\big(
	\widehat{\Omega^{\Spinc}_{2k-1}}(X,\widehat{\mathcal G}),\RZ
	\big)
	\]
	satisfying the
	compatibility condition
	$
	h_x\circ a^{\Spinc}=\modZ\circ\langle - ,\omega_x\rangle.
	$
	
	\medskip
	
	Let
	\[
	C=(M,f,\nabla^{S^c},\Psi)
	\]
	be a closed $(2k-1)$-dimensional geometric $\widehat{\mathcal G}$-twisted
	$\Spinc$-chain over $X$.  Its twisted $\Spinc$-structure implies
	\[
	f^*DD(\mathcal G)=W_3(TM),
	\]
	and the pullback twist
	$f^*\widehat{\mathcal G}$ has torsion underlying topological class over $M$, so $f^*\widehat{\mathcal G}^{-1}$ is also torsion. Hence we may choose a finite rank representative
	\begin{equation}\label{eq:finite-rep-y}
		y=(V_y^+,V_y^-,\nabla^{V_y^+},\nabla^{V_y^-},\rho_y)
	\end{equation}
	for the pullback differential twisted $K$-class $f^*[x]$ in $\widehat K^0(M,f^*\widehat{\mathcal G}^{-1})$. 
	Tensoring with the $f^*\widehat{\mathcal G}$-twisted spinor module $S^c$, we write $S^c\otimes V_y^\pm$ as descended honest finite rank Clifford modules on $M$, and 
	$D_{C,y}^\pm$
	as the corresponding self-adjoint Dirac operators.  
	Define
	\begin{equation}\label{eq:cyclewise-bareta-y}
		\bar\eta_x(C;y):=
		\bar\eta(D_{C,y}^+)-\bar\eta(D_{C,y}^-)
		-
		\int_M \widehat A(TM)\wedge e^{\kappa_C}\wedge \rho_y
		\quad\in \RZ,
	\end{equation}
	where $\kappa_C=\tr_0(\widetilde{F^{S^c}})$ is the two-form associated to the differential twisted
	$\Spinc$-structure on $M$.  
	\begin{lem}\label{lem:bareta-independent-y}
		$\bar\eta_x(C;y)$ is independent of the choice of finite rank
		representative $y$ satisfying $[y]=[f^*x]$.
	\end{lem}
	
	\begin{proof}
		It suffices to check the claim for the change-of-connection relation. 
		For two representatives
		\[
		y_0=(V^\pm, \nabla_0^\pm, \rho_{y_0}), \quad
		y_1=(V^\pm, \nabla_1^\pm, \rho_{y_1}), \quad
		\rho_{y_1}=\rho_{y_0}+
		\CS_{f^*\widehat{\mathcal G}^{-1}}(\nabla_0^\pm,\nabla_1^\pm),
		\]
		the Atiyah--Patodi--Singer cylinder formula for the finite rank Clifford modules
		$S^c\otimes V^\pm$ gives
		\begin{align*}
			\bigl(\bar\eta(D_{C,y_1}^+)-\bar\eta(D_{C,y_1}^-)\bigr)
			-
			\bigl(\bar\eta(D_{C,y_0}^+)-\bar\eta(D_{C,y_0}^-)\bigr)
			\equiv
			\int_M \widehat A(TM)e^{\kappa_C}\,
			\CS_{f^*\widehat{\mathcal G}^{-1}}(\nabla_0^\pm,\nabla_1^\pm)
			\quad \modZ.
		\end{align*}
		Therefore
		\begin{align*}
			\bar\eta_x(C;y_1)-\bar\eta_x(C;y_0)
			\equiv
			\int_M \widehat A(TM)e^{\kappa_C}\,
			\CS_{f^*\widehat{\mathcal G}^{-1}}(\nabla_0^\pm,\nabla_1^\pm)
			-
			\int_M \widehat A(TM)e^{\kappa_C}\,(\rho_{y_1}-\rho_{y_0})
			=0\quad \modZ.
		\end{align*}
		Hence $\bar\eta_x(C;y)$ is independent of $y$.
	\end{proof}
	
	By Lemma~\ref{lem:bareta-independent-y}, we may simply write 
	\[
	\bar\eta_x(C):=\bar\eta_x(C;y).
	\]
	In fact, this construction may be interpreted as the differential $K$-theory pushforward in the untwisted case
	\cite{bunke2007smooth,freed2010index}. 
	Now we are ready to define the functional component in the desired anomaly class. 
	We set
	\begin{equation}\label{construction: h}
		h_x(C,\varphi):=
		\bar\eta_x(C)-\langle \varphi,\omega_x\rangle
		\quad\in \RZ,
	\end{equation}
	where $(C,\varphi)=(M,f,\nabla^{S^c},\Psi,\varphi)$ is a $(2k-1)$-dimensional differential twisted $\Spinc$-cycle. Furthermore, we have
	\begin{lem}\label{lem:h-welldef}
		The functional component $h_x$ descends to
		$\widehat{\Omega^{\Spinc}_{2k-1}}(X,\widehat{\mathcal G})$.
	\end{lem}
	\begin{proof}
		Additivity and isomorphism invariance are immediate from construction. 
		It suffices for us to verify that $h_x$ is invariant under the bordism relation. 
		Let
		\[
		B=(W,F,\nabla^{S^c_W},\Psi_W)
		\]
		be a $2k$-dimensional geometric $\widehat{\mathcal G}$-twisted $\Spinc$-chain
		with boundary $C=\partial B$.  
		Choose a finite rank representative
		\[
		Y=(U^+,U^-,\nabla^{U^+},\nabla^{U^-},\rho_Y)
		\]
		of the class $[F^*x]\in \widehat K^0(W,F^*\widehat{\mathcal G}^{-1})$.  Its
		restriction $\partial Y$ is a finite rank representative of $[f^*x]$ on
		$\partial W$.  By Lemma~\ref{lem:bareta-independent-y}, we may compute
		$\bar\eta_x(\partial B)$ using this boundary representative. Applying Atiyah--Patodi--Singer index theorem to the honest Clifford modules
		$S_W^c\otimes U^\pm$, we have
		\[
		\bar\eta(D_{\partial B,\partial Y}^{+})
		-
		\bar\eta(D_{\partial B,\partial Y}^{-})
		\equiv
		\int_W \widehat A(TW)e^{\kappa_W}
		\ch_{F^*\widehat{\mathcal G}^{-1}}(\nabla^{U^+},\nabla^{U^-}) \quad \modZ.
		\]
		By Stokes' theorem and $d\kappa_W=F^*H$,
		\[
		\int_{\partial W}\widehat A(T\partial W)e^{\kappa_{\partial W}}\rho_{\partial Y}
		=
		\int_W \widehat A(TW)e^{\kappa_W}(d+F^*H)\rho_Y.
		\]
		Since $[Y]=[F^*x]$, their curvatures agree
		\[
		R_K(Y)=
		\ch_{F^*\widehat{\mathcal G}^{-1}}(\nabla^{U^+},\nabla^{U^-})
		-(d+F^*H)\rho_Y
		=F^*R_K(x).
		\]
		Combining the above formulas gives
		\[
		\bar\eta_x(\partial B)
		\equiv
		\int_W\widehat A(TW)e^{\kappa_W}F^*R_K(x)
		=
		\langle \cw(B),\omega_x\rangle
		\quad\modZ,
		\]
		which is precisely the equality needed for the bordism relation in
		Definition~\ref{defi:gerbe-spinc-bord}. 
	\end{proof}
	
	To summarize the construction, we have
	
	\begin{prop}\label{prop:ano}
		The assignment
		$
		[x]\mapsto (\omega_x,h_x)
		$
		gives a well-defined homomorphism
		\[
		\widehat{\Phi}_{\widehat{\mathcal G}}\colon
		\widehat K^0(X,\widehat{\mathcal G}^{-1})
		\longrightarrow
		\bigl(\widehat{I\Omega^{\Spinc}_{\dR}}\bigr)^{2k}(X,\widehat{\mathcal G}).
		\]
	\end{prop}

	\begin{proof}
		
		The assignment is clearly a homomorphism. We verify the compatibility condition. 
		For a current class $\alpha$, the cycle $a^{\Spinc}(\alpha)$ has
		empty geometric part and current component $-\alpha$.  Therefore
		\[
		h_x(a^{\Spinc}(\alpha))
		=-\langle -\alpha,\omega_x\rangle
		=\langle \alpha,
		\omega_x\rangle\quad \modZ.
		\]
		It remains to check dependence on the representative $x$.
		By definition, the curvature component depends only on the differential $K$-class. 
		By
		Lemma~\ref{lem:bareta-independent-y}, $\bar\eta_x(C)$ is independent of the finite rank representative for $[f^*x]$, hence the functional component $h_x$ also depends only on the class $[x]$.
		Thus the assignment
		is a well-defined homomorphism.
	\end{proof}

\bibliographystyle{alpha}
\bibliography{references_anderson}

@article{ABGHR,
	author = {
	Ando, Matthew and Blumberg, Andrew J. and Gepner, David and Hopkins, Michael J. and Rezk, Charles},
	title = {An {$\infty$}-categorical approach to {$R$}-line bundles, {$R$}-module {T}hom spectra, and twisted {$R$}-homology},
	journal = {Journal of Topology},
	volume = {7},
	number = {3},
	pages = {869-893},
	year = {2014}
}

@article{ABGHR2,
	title={Units of ring spectra, orientations, and {Thom} spectra via rigid infinite loop space theory},
	author={Ando, Matthew and Blumberg, Andrew J. and Gepner, David and Hopkins, Michael J. and Rezk, Charles},
	journal={Journal of Topology},
	volume={7},
	number={4},
	pages={1077--1117},
	year={2014},
	publisher={Oxford University Press}
}

@article{NikWal,
	title={Four equivalent versions of nonabelian gerbes},
	author={Nikolaus, Thomas and Waldorf, Konrad},
	journal={Pacific Journal of Mathematics},
	volume={264},
	number={2},
	pages={355--420},
	year={2013},
	publisher={Mathematical Sciences Publishers}
}

@article{atiyah2004twisted,
  title={Twisted {$K$}-theory},
  author={Atiyah, Michael and Segal, Graeme},
  journal={arXiv preprint math/0407054},
  year={2004}
}

@article{simons2008structured,
  title={Structured vector bundles define differential {$K$}-theory},
  author={Simons, James and Sullivan, Dennis},
  journal={arXiv preprint arXiv:0810.4935},
  year={2008}
}

@article{bunke2007smooth,
	title={Smooth {$K$}-theory},
	author={Bunke, Ulrich and Schick, Thomas},
	journal={Ast{\'e}risque},
	volume={328},
	pages={45--135},
	year={2009}
}

@article{freed2010index,
  title={An index theorem in differential {$K$}-theory},
  author={Freed, Daniel S. and Lott, John},
  journal={Geometry \& Topology},
  volume={14},
  number={2},
  pages={903--966},
  year={2010},
  publisher={Mathematical Sciences Publishers}
}

@article{rosenberg1989continuous,
	title={Continuous-trace algebras from the bundle theoretic point of view},
	author={Rosenberg, Jonathan},
	journal={Journal of the Australian Mathematical Society},
	volume={47},
	number={3},
	pages={368--381},
	year={1989},
	publisher={Cambridge University Press}
}

@article{donovan,
  title={Graded {Brauer} groups and {$K$}-theory with local coefficients},
  author={Donovan, Peter and Karoubi, Max},
  journal={Publications Math{\'e}matiques de l'Institut des Hautes {\'E}tudes Scientifiques},
  volume={38},
  number={1},
  pages={5--25},
  year={1970},
  publisher={Springer}
}

@article{FHT1,
  title={Loop groups and twisted {$K$}-theory {I}},
  author={Freed, Daniel S. and Hopkins, Michael J. and Teleman, Constantin},
  journal={Journal of Topology},
  volume={4},
  number={4},
  pages={737--798},
  year={2011},
  publisher={London Mathematical Society}
}

@article{hitchin1999lectures,
  title={Lectures on special {Lagrangian} submanifolds},
  author={Hitchin, Nigel},
  journal={arXiv preprint math/9907034},
  year={1999}
}

@incollection{ando2010twists,
  author    = {Ando, Matthew and Blumberg, Andrew J. and Gepner, David},
  title     = {Twists of {$K$}-theory and {TMF}},
  booktitle = {Superstrings, Geometry, Topology, and {$C^*$}-algebras},
  series    = {Proceedings of Symposia in Pure Mathematics},
  volume    = {81},
  pages     = {27--63},
  publisher = {American Mathematical Society},
  address   = {Providence, RI},
  year      = {2010}
}

@article{stong1966relations,
  title={Relations among characteristic numbers--{II}},
  author={Stong, Robert E.},
  journal={Topology},
  volume={5},
  number={2},
  pages={133--148},
  year={1966},
  publisher={Elsevier}
}

@incollection{stolz2004elliptic,
	title={What is an elliptic object?},
	author={Stolz, Stephan and Teichner, Peter},
	booktitle={Topology, Geometry and Quantum Field Theory},
	series={London Mathematical Society Lecture Note Series},
	volume={308},
	pages={247--343},
	year={2004},
	publisher={Cambridge University Press},
}

@incollection{hohnhold2010minimal,
  title={From minimal geodesics to supersymmetric field theories},
  author={Hohnhold, Henning and Stolz, Stephan and Teichner, Peter},
  booktitle={A Celebration of the Mathematical Legacy of Raoul Bott},
  series={CRM Proceedings \& Lecture Notes},
  volume={50},
  pages={207--274},
  year={2010},
  publisher={American Mathematical Society},
  address={Providence, RI}
}

@article{freed2007uncertainty,
  title={The uncertainty of fluxes},
  author={Freed, Daniel S. and Moore, Gregory W. and Segal, Graeme},
  journal={Communications in Mathematical Physics},
  volume={271},
  number={1},
  pages={247--274},
  year={2007},
  publisher={Springer}
}

@article{BCMMS,
  title={Twisted {$K$}-theory and {$K$}-theory of bundle gerbes},
  author={Bouwknegt, Peter and Carey, Alan L. and Mathai, Varghese and Murray, Michael K. and Stevenson, Danny},
  journal={Communications in Mathematical Physics},
  volume={228},
  pages={17--49},
  year={2002},
  publisher={Springer}
}

@article{Wang07,
  title={Geometric cycles, index theory and twisted {$K$}-homology},
  author={Wang, Bai-Ling},
  journal={Journal of Noncommutative Geometry},
  year={2008},
  volume={2},
  pages={497--552}
}

@phdthesis{Chatterjee1998,
	title={On Gerbs},
	author={Chatterjee, David Saumitra},
	school={University of Cambridge},
	address={Cambridge},
	year={1998},
	note={Ph.D. thesis, Trinity College}
}

@article{Mur96,
	title={Bundle gerbes},
	author={Murray, Michael K},
	journal={Journal of the London Mathematical Society},
	volume={54},
	number={2},
	pages={403--416},
	year={1996},
	publisher={Wiley Online Library}
}

@article{MS00,
	title={Bundle gerbes: stable isomorphism and local theory},
	author={Murray, Michael K and Stevenson, Daniel},
	journal={Journal of the London Mathematical Society},
	volume={62},
	number={3},
	pages={925--937},
	year={2000},
	publisher={Cambridge University Press}
}

@article{FH13,
  title={{Chern--Weil} forms and abstract homotopy theory},
  author={Freed, Daniel and Hopkins, Michael},
  journal={Bulletin of the American Mathematical Society},
  volume={50},
  number={3},
  pages={431--468},
  year={2013}
}

@article{BS10,
	title={Uniqueness of smooth extensions of generalized cohomology theories},
	author={Bunke, Ulrich and Schick, Thomas},
	journal={Journal of Topology},
	volume={3},
	number={1},
	pages={110--156},
	year={2010},
	publisher={London Mathematical Society}
}

@article{Yam1,
	author = {Yamashita, Mayuko and Yonekura, Kazuya},
	title = {Differential models for the {Anderson} dual to bordism theories and invertible {QFT}'s. {I}},
	journal = {J. G{\"o}kova Geom. Topol. GGT},
	volume = {16},
	pages = {1--64},
	year = {2023}
}

@article{YamII,
	author = {Yamashita, Mayuko},
	title = {Differential models for the {Anderson} dual to bordism theories and invertible {QFT}'s. {II}},
	journal = {J. G{\"o}kova Geom. Topol. GGT},
	volume = {16},
	pages = {65--97},
	year = {2023}
}

@incollection{Yam23,
  author       = {Yamashita, Mayuko},
  title        = {Invertible {QFT}s and differential {Anderson} duals},
  booktitle    = {String-Math 2022},
  series       = {Proceedings of Symposia in Pure Mathematics},
  volume       = {107},
  pages        = {277--294},
  year         = {2024},
  publisher    = {American Mathematical Society},
  address      = {Providence, RI},
  editor       = {Donagi, Ron and Langer, Adrian and Su{\l}kowski, Piotr and Wendland, Katrin}
}

@article{Tachi, 
	title={Topological modular forms and the absence of all heterotic global anomalies},
	author={Tachikawa, Yuji and Yamashita, Mayuko},
	journal={Communications in Mathematical Physics},
	volume={402},
	number={2},
	pages={1585--1620},
	year={2023},
	publisher={Springer}
}

@article{park,
  title={Geometric models of twisted differential {$K$}-theory {I}},
  author={Park, Byungdo},
  journal={Journal of Homotopy and Related Structures},
  volume={13},
  number={1},
  pages={143--167},
  year={2018},
  publisher={Springer}
}

@article{HS,
  title={Quadratic functions in geometry, topology, and {$M$}-theory},
  author={Hopkins, Michael J. and Singer, Isadore M.},
  journal={Journal of Differential Geometry},
  volume={70},
  number={3},
  pages={329--452},
  year={2005},
  publisher={Lehigh University}
}

@article{MS03,
  title={Chern character in twisted {$K$}-theory: equivariant and holomorphic cases},
  author={Mathai, Varghese and Stevenson, Danny},
  journal={Communications in Mathematical Physics},
  volume={236},
  number={1},
  pages={161--186},
  year={2003},
  publisher={Springer}
}

@book{MaySig,
  title={Parametrized Homotopy Theory},
  author={May, J. P. and Sigurdsson, J.},
  series={Mathematical Surveys and Monographs},
  volume={132},
  publisher={American Mathematical Society},
  address={Providence, RI},
  year={2006}
}

@article{BunNik,
	title={Twisted differential cohomology},
	author={Bunke, Ulrich and Nikolaus, Thomas},
	journal={Algebraic \& Geometric Topology},
	volume={19},
	number={4},
	pages={1631--1710},
	year={2019},
	publisher={Mathematical Sciences Publishers}
}

@article{bunke2009landweber,
	title={Landweber exact formal group laws and smooth cohomology theories},
	author={Bunke, Ulrich and Schick, Thomas and Schr{\"o}der, Ingo and Wiethaup, Moritz},
	journal={Algebraic \& Geometric Topology},
	volume={9},
	number={3},
	pages={1751--1790},
	year={2009},
	publisher={Mathematical Sciences Publishers}
}

@article{bunke2016differential,
	title={Differential cohomology theories as sheaves of spectra},
	author={Bunke, Ulrich and Nikolaus, Thomas and V{\"o}lkl, Michael},
	journal={Journal of Homotopy and Related Structures},
	volume={11},
	number={1},
	pages={1--66},
	year={2016},
	publisher={Springer}
}

@incollection{ST11,
  title={Supersymmetric field theories and generalized cohomology},
  author={Stolz, Stephan and Teichner, Peter},
  booktitle={Mathematical Foundations of Quantum Field Theory and Perturbative String Theory},
  series={Proceedings of Symposia in Pure Mathematics},
  volume={83},
  pages={279--340},
  year={2011},
  publisher={American Mathematical Society},
  address={Providence, RI}
}

@book{dR,
  title={Differentiable manifolds: forms, currents, harmonic forms},
  author={{de Rham}, Georges},
  volume={266},
  year={2012},
  publisher={Springer Science \& Business Media}
}

@article{freed2021reflection,
	title={Reflection positivity and invertible topological phases},
	author={Freed, Daniel S and Hopkins, Michael J},
	journal={Geometry \& Topology},
	volume={25},
	number={3},
	pages={1165--1330},
	year={2021},
	publisher={Mathematical Sciences Publishers}
}

\bigskip
\begingroup
\footnotesize

\noindent\textsc{Department of Mathematics, National University of Singapore, Singapore 119076}

\noindent\textit{Email address:} \href{mailto:mathanf@nus.edu.sg}{mathanf@nus.edu.sg}

\medskip

\noindent\textsc{Department of Mathematics, National University of Singapore, Singapore 119076}

\noindent\textit{Email address:} \href{mailto:yuanchuli@u.nus.edu}{yuanchuli@u.nus.edu}

\endgroup

\end{document}